\newcount\Comments  
\documentclass[letterpaper,10pt]{scrartcl}

\usepackage[utf8]{inputenc}
\usepackage[english]{babel}
\usepackage[automark]{scrpage2}
\usepackage{amsmath,amstext,amsbsy,amsopn,amscd,amsxtra,upref,amssymb}
\usepackage{amsfonts,bm}
\usepackage{amsthm}
\usepackage{color}
\usepackage{graphicx}
\usepackage{sidecap}
\usepackage{multirow}
\usepackage{booktabs}
\usepackage{bm}
\usepackage{fullpage}

\usepackage{tabularx}
\usepackage[font=small]{caption}
\usepackage{subfig}

\usepackage{longtable}
\usepackage{rotating}

\usepackage{algorithm}
\usepackage{algpseudocode}
\extrafloats{100}

\usepackage{lineno,hyperref}

\graphicspath{{figures/}}
\DeclareGraphicsExtensions{.pdf,.eps,.png,.jpg,.jpeg}

\newcommand{\bfalpha}{\boldsymbol \alpha}
\newcommand{\bfbeta}{\boldsymbol \beta}

\newcommand{\bff}{\boldsymbol f}

\newcommand{\bfy}{\boldsymbol y}

\newcommand{\Dcal}{\mathcal{D}}

\newcommand{\Ucal}{\mathcal{U}}

\newcommand{\bfmu}{\boldsymbol \mu}

\newcommand{\bfe}{\boldsymbol e}
\newcommand{\bfp}{\boldsymbol p}
\newcommand{\bfq}{\boldsymbol q}

\newcommand{\bfr}{\boldsymbol r}
\newcommand{\bfs}{\boldsymbol s}
\newcommand{\bfu}{\boldsymbol u}

\newcommand{\bfC}{\boldsymbol C}

\newcommand{\bfF}{\boldsymbol F}

\newcommand{\bfP}{\boldsymbol P}
\newcommand{\bfR}{\boldsymbol R}
\newcommand{\bfU}{\boldsymbol U}

\newcommand{\bfQ}{\boldsymbol Q}

\newcommand{\bfS}{\boldsymbol S}

\newcommand{\winit}{w_{\text{init}}}
\newcommand{\win}{w}

\newcommand{\arate}{a}
\newcommand{\abas}{AADEIM}

\newcommand{\nh}{N}
\newcommand{\nr}{n}
\newcommand{\nab}{r} 
\newcommand{\nms}{m} 
\newcommand{\nz}{z} 

\newenvironment{keywords}%
   {\begin{trivlist}\item[]{\bfseries\sffamily Keywords:}\ }
   {\end{trivlist}}

\ihead[]{}
\ohead[]{}
\ifoot[]{}
\ofoot[]{}

\newtheorem{lemma}{Lemma}

\theoremstyle{definition}

\newtheorem{proposition}{Proposition}

\ifpdf
\hypersetup{
  pdftitle={Model reduction for transport-dominated problems via online adaptive bases and adaptive sampling},
  pdfauthor={Benjamin Peherstorfer}
}
\fi

\title{Model reduction for transport-dominated problems via online adaptive bases and adaptive sampling}
\author{Benjamin Peherstorfer\thanks{Courant Institute of Mathematical Sciences, New York University, New York, NY 10012. The work of Peherstorfer is supported in parts by the Air Force Center of Excellence on Multi-Fidelity Modeling of Rocket Combustor Dynamics, Award Number FA9550-17-1-0195.}}
\date{December 6, 2018 (revised June 9, 2020)}

\begin{document}

\maketitle

\begin{abstract}
This work presents a model reduction approach for problems with coherent structures that propagate over time such as convection-dominated flows and wave-type phenomena. Traditional model reduction methods have difficulties with these transport-dominated problems because propagating coherent structures typically introduce high-dimensional features that require high-dimensional approximation spaces. The approach proposed in this work exploits the locality in space and time of propagating coherent structures to derive efficient reduced models. Full-model solutions are approximated locally in time via local reduced spaces that are adapted with basis updates during time stepping. The basis updates are derived from querying the full model at a few selected spatial coordinates. A core contribution of this work is an adaptive sampling scheme for selecting at which components to query the full model to compute basis updates. The presented analysis shows that, in probability, the more local the coherent structure is in space, the fewer full-model samples are required to adapt the reduced basis with the proposed adaptive sampling scheme. Numerical results on benchmark examples with interacting wave-type structures and time-varying transport speeds and on a model combustor of a single-element rocket engine demonstrate the wide applicability of the proposed approach and runtime speedups of up to one order of magnitude compared to full models and traditional reduced models.
\end{abstract}


\section{Introduction}
\label{sec:Intro}
Reduced models of large-scale systems of equations are typically constructed in a one-time high-cost offline phase and then used in an online phase to repeatedly compute accurate approximations of the full-model solutions with significantly lower costs. In projection-based model reduction \cite{RozzaPateraSurvey,AntBG10,SIREVSurvey}, a low-dimensional space is constructed that approximates well the high-dimensional solution space of the full model. Then, the full-model equations are projected onto the low-dimensional space and reduced solutions are computed by solving the projected equations. However, solutions of full models that describe transport-dominated behavior, e.g., convection-dominated flows, wave-type phenomena, shock propagation, typically exhibit high-dimensional features, which means that no low-dimensional space (``linear approximation'') exists in which the full-model solutions can be approximated well; the Kolmogorov n-width is high. Thus, traditional model reduction methods fail for transport-dominated problems \cite{OHLBERGER2013901,dahmen_plesken_welper_2014,Tommasao}. In this work, we exploit that transport-dominated problems typically have a rich structure that is local in nature, which we leverage to derive efficient reduced models.

Model reduction projects the full-model equations on a low-dimensional---reduced--space that is spanned by a set of basis vectors. There are many methods for constructing reduced spaces, including proper orthogonal decomposition (POD) \cite{PODFluid,SirovichMethodOfSnapshots}, balanced truncation \cite{moore_principal_1981,mullis_synthesis_1976}, the reduced basis method \cite{prudhomme_reliable_2001,veroy_posteriori_2003,ROZZA20071244,refId0Haasdonk,RozzaPateraSurvey}, and interpolatory model reduction methods \cite{gugercin_2008,AntBG10}. In the following, we will be mostly concerned with full-model equations that are nonlinear in the states, where projecting the full-model equations onto a reduced space is typically insufficient to obtain a reduced model that is faster to solve than the full model, because the nonlinear terms entail computations that scale with the dimension of the full-model solution space. One remedy is empirical interpolation, which approximates nonlinear terms by evaluating them at a few, carefully selected interpolation points and approximating all other components via interpolation in a low-dimensional space  \cite{barrault_empirical_2004,grepl_efficient_2007,doi:10.1137/10081157X}. We will build on the discrete empirical interpolation method (DEIM) \cite{deim2010,QDEIM}, which is the discrete counterpart of empirical interpolation. Note that there are other techniques for nonlinear model reduction, e.g., missing point estimation \cite{4668528} and the Gauss-Newton with approximated tensors (GNAT) method \cite{GNAT}. All these methods build on the assumption that there is a low-dimensional space in which the full-model solutions can be approximated well, which is violated in case of transport-dominated problems. However, the solutions of transport-dominated problems typically are low dimensional if considered locally in time, which we exploit by approximating the full-model solutions in local low-dimensional spaces that are adapted via low-rank basis updates over time \cite{Peherstorfer15aDEIM,ZPW17SIMAXManifold}. The basis updates are derived by querying the full model at selected points in the spatial domain. We derive an adaptive sampling scheme that selects where to query the full model to compute the basis updates. An analysis is developed that shows that if the reduced-model residual is local in the spatial domain---which we observe for transport-dominated problems---then only few sampling points are necessary to adapt the local spaces with the proposed adaptive sampling scheme. Even though online operations of the proposed approach incur costs that depend on the dimension of the full-model states, the dimension of the adapted spaces can be selected lower than with static reduced models, which makes the proposed approach computationally efficient and leads to speedups of up to one order of magnitude in our numerical experiments. We summarize the key contributions of this work as follows:
\begin{itemize}
\item Providing an analysis that relates the locality of the residual in the spatial domain to the number of sampling points required for adapting reduced bases with the proposed approach.
\item Deriving a sampling scheme to select which components of the full model to sample for deriving basis updates.
\item Demonstrating the efficiency of the proposed approach on numerical examples with time-varying coefficients and nonlinear dynamics by reporting speedups of up to one order of magnitude compared to full-model solutions and traditional, static reduced models.
\end{itemize}

For reviewing the literature on model reduction for transport-dominated problems, we broadly distinguish between three lines of research. Literature that we categorize into the first line of research is primarily interested in the stability of reduced models of transport-dominated phenomena. The work \cite{dahmen_plesken_welper_2014} formulates the greedy basis construction methods, developed in the reduced basis community \cite{prudhomme_reliable_2001,veroy_posteriori_2003}, in special norms that are better suited for transport-dominated problems to obtain stable approximations. The authors of \cite{URBAN2012203,doi:10.1142/S0218202514500110} consider time-space discretizations to guarantee stability of the reduced models. Then, there is a large body of work on closure modeling \cite{Wang12PODclosureComp,2017arXiv171003569F,doi:10.1137/18M1177263,PARISH2016758,Iollo2000,ROWLEY2004115} where methods for stabilizing reduced models have been developed. Into the second line of research, we categorize work that transforms the full model to recover a low-rank structure that can then be exploited with a reduced space. The work on transformations in the model reduction community seems to have originated from \cite{OHLBERGER2013901}, where the transport dynamics are separated from the rest of the dynamics via freezing. In \cite{RimLeVeque,Rim2}, transport maps are constructed that reverse the transport and so recover a low-rank structure. Approaches have been developed that aim to find transformations numerically via, e.g., optimization \cite{Cagniart2019}. Shifted POD \cite{Reiss} recovers the shift due to the transport and applies POD after having reversed the shift. The shift operator introduced in \cite{Reiss} is time dependent and separates different transports to be applicable to, e.g., multiple waves traveling at different wave speeds. In \cite{2017arXiv171011481W,WelperInterpolate}, snapshots are transformed to better align them before the reduced bases are constructed. The third line of research on model reduction for transport-dominated problems constructs low-dimensional spaces that explicitly target transport-dominated problems.
There are approaches that exploit local structure, such as the approach introduced in \cite{Tommasao} that builds on the spatial locality of shock fronts. The work \cite{Abgrall2016} constructs bases via $L_1$ optimization and uses a dictionary approach. In \cite{doi:10.1002/nme.4800}, a special type of adaptation is developed that enriches the reduced space during the online phase if an error indicator signals a high error.
In \cite{PhysRevE.89.022923}, reduced bases are adapted over time via an auxiliary equation that describes the dynamics of the bases. Under certain conditions, the approach can be seen as an approximation of Lax pairs discussed in \cite{GERBEAU2014246}. Dynamic low-rank approximation \cite{Koch2007,SAPSIS20092347,NobileDO,Musharbash:231216,MUSHARBASH2018135} adapts bases in time similar to the proposed approach here; however, the bases in dynamic low-rank approximation approaches are constructed so that they are valid for all parameters in a given parameter domain, which is problematic if the solution manifold contains high-dimensional features because the propagating coherent structure depends on parameters, as in, e.g., \cite[Example~2.5]{Tommasao}. In contrast, the approach developed in this work is local in the parameter domain, additionally to being local in time and space. Furthermore, locality in the parameter domain enables deriving basis updates from only few samples of the full model.

This work is organized as follows. Section~\ref{sec:Prelim} sets up the problem and gives the problem formulation. Section~\ref{sec:ABAS} demonstrates the local structure in transport-dominated problems and proposes an approach to adaptively sample the full model to construct basis updates. Numerical results in Section~\ref{sec:NumExp} on benchmark examples and on a model combustor of a single-element rocket engine demonstrate that the proposed approach achieves significant speedups and is applicable to a wide range of problems. In particular, the numerical results indicate that the proposed approach faithfully approximates interactions between propagating coherent structures traveling at different speeds. Concluding remarks are provided in Section~\ref{sec:Conc}.

\section{Preliminaries}
\label{sec:Prelim}
We briefly discuss model reduction with empirical interpolation in Section~\ref{sec:Prelim:MOR} and then demonstrate on a toy example  in Section~\ref{sec:Prelim:ProblemFormulation} why these traditional model reduction methods fail for problems exhibiting transport-dominated phenomena.

\subsection{Model reduction with empirical interpolation}
\label{sec:Prelim:MOR}
Consider the system of discretized equations that is obtained after discretizing a partial differential equation (PDE) in space and time
\begin{equation}
\bfq_{k}(\bfmu) = \bff(\bfq_{k+1}(\bfmu), \bfmu)\,,\qquad k = 1, \dots, K\,,
\label{eq:Prelim:FOMProblem}
\end{equation}
where $\bfq_k(\bfmu) \in \mathbb{R}^{\nh}$ is the $\nh$-dimensional state at time $k$ and parameter $\bfmu \in \Dcal$, with parameter domain $\Dcal$. The number of time steps is $K \in \mathbb{N}$. The function $\bff: \mathbb{R}^{\nh} \times \Dcal \to \mathbb{R}^{\nh}$ describes the operators of the discretized PDE and typically is nonlinear in the state $\bfq_{k + 1}(\bfmu)$. The time discretization is implicit in time, which means that at each time step $k = 1, \dots, K$, a potentially nonlinear, large-scale system of equations has to be solved, e.g., with Newton's method. The formulation of the full model \eqref{eq:Prelim:FOMProblem} is different from the formulation in \cite{deim2010} where the linear and nonlinear terms of the equations are treated separately.

To derive a reduced model of the full model \eqref{eq:Prelim:FOMProblem} with empirical interpolation \cite{barrault_empirical_2004,deim2010,QDEIM}, consider the trajectory at parameter $\bfmu \in \Dcal$
\[
\bfQ(\bfmu) = [\bfq_1(\bfmu), \dots, \bfq_K(\bfmu)] \in \mathbb{R}^{\nh \times K}\,,
\]
which is the matrix with the states $\bfq_1(\bfmu), \dots, \bfq_K(\bfmu)$ as columns. Let the columns of $\bfU = [\bfu_1, \dots, \bfu_{\nr}] \in \mathbb{R}^{\nh \times \nr}$ be the POD basis of dimension $\nr \ll \nh$ obtained from the snapshot matrix
\[
\bfQ = [\bfQ(\bfmu_1), \dots, \bfQ(\bfmu_M)] \in \mathbb{R}^{\nh \times MK}\,,
\]
with parameters $\bfmu_1, \dots, \bfmu_M \in \Dcal$. The space spanned by the columns of $\bfU$ is denoted as $\Ucal \subset \mathbb{R}^{\nh}$ and is a subspace of $\mathbb{R}^{\nh}$. The critical assumption of traditional model reduction is that the singular values of $\bfQ$ decay fast so that only few basis vectors are necessary to approximate well the columns of $\bfQ$ in the corresponding space $\Ucal$. Following QDEIM, introduced in \cite{QDEIM}, select the interpolation points $p_1, \dots, p_{\nr} \in \{1, \dots, \nh\}$ and define the corresponding interpolation points matrix $\bfP = [\bfe_{p_1}, \dots, \bfe_{p_{\nr}}] \in \mathbb{R}^{\nh \times \nr}$, where $\bfe_{p_i} \in \mathbb{R}^{\nh}$ is the $p_i$-th canonical unit vector with entry 1 at the $p_i$-th component and entry 0 at all other components.
Define
\[
\tilde{\bff}(\bfq(\bfmu); \bfmu) = (\bfP^T\bfU)^{-1}\bfP^T\bff(\bfq(\bfmu); \bfmu)\,,
\]
so that $\bfU\tilde{\bff}(\bfq(\bfmu); \bfmu)$ is the DEIM approximation of $\bff(\bfq(\bfmu); \bfmu)$ at state $\bfq(\bfmu) \in \mathbb{R}^{\nh}$ and parameter $\bfmu \in \Dcal$. Note that computing $\bfP^T\bff(\bfq(\bfmu); \bfmu)$ typically requires evaluating $\bff(\bfq(\bfmu); \bfmu)$ at the $\nr$ interpolation points $p_1, \dots, p_{\nr}$ only, see \cite{barrault_empirical_2004,deim2010,QDEIM}. We overload the notation $\tilde{\bff}$ in the following so that if we have a reduced state $\tilde{\bfq}(\bfmu) \in \mathbb{R}^{\nr}$, then $\tilde{\bff}(\tilde{\bfq}(\bfmu); \bfmu) = (\bfP^T\bfU)^{-1}\bfP^T\bff(\bfU\tilde{\bfq}(\bfmu); \bfmu)$. The reduced model corresponding to $\tilde{\bff}$ is
\begin{equation}
\tilde{\bfq}_k(\bfmu) = \tilde{\bff}(\tilde{\bfq}_{k + 1}(\bfmu); \bfmu)\,,\qquad k = 1, \dots, K\,,
\label{eq:Prelim:RedProblem}
\end{equation}
with the reduced trajectory $\tilde{\bfQ}(\bfmu) = [\tilde{\bfq}_1(\bfmu), \dots, \tilde{\bfq}_K(\bfmu)] \in \mathbb{R}^{\nr \times K}$. Note that we approximate the state and the nonlinear function in the same space $\Ucal$, which is in contrast to the original use of DEIM in \cite{deim2010} and similar to model reduction via missing point estimation \cite{4668528}. Once a reduced model \eqref{eq:Prelim:RedProblem} is constructed in the offline phase, it is solved in the online phase. The one-time high costs of constructing the reduced model are compensated by approximating the full-model solutions with reduced-model solutions for a large number of parameters online, see, e.g., \cite{RozzaPateraSurvey,AntBG10,SIREVSurvey,PWG17MultiSurvey} for details on the wide range of outer-loop and many-query applications where such an offline/online splitting is beneficial.

\begin{figure}
\begin{tabular}{cc}
{\LARGE\resizebox{0.45\columnwidth}{!}{\input{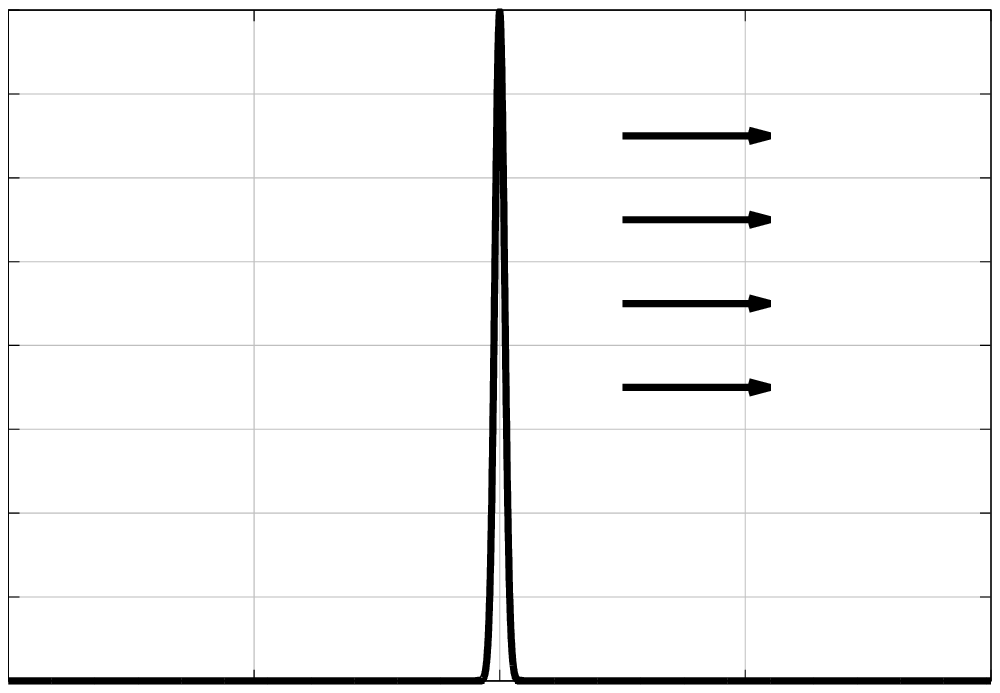}}} & {\LARGE\resizebox{0.45\columnwidth}{!}{\input{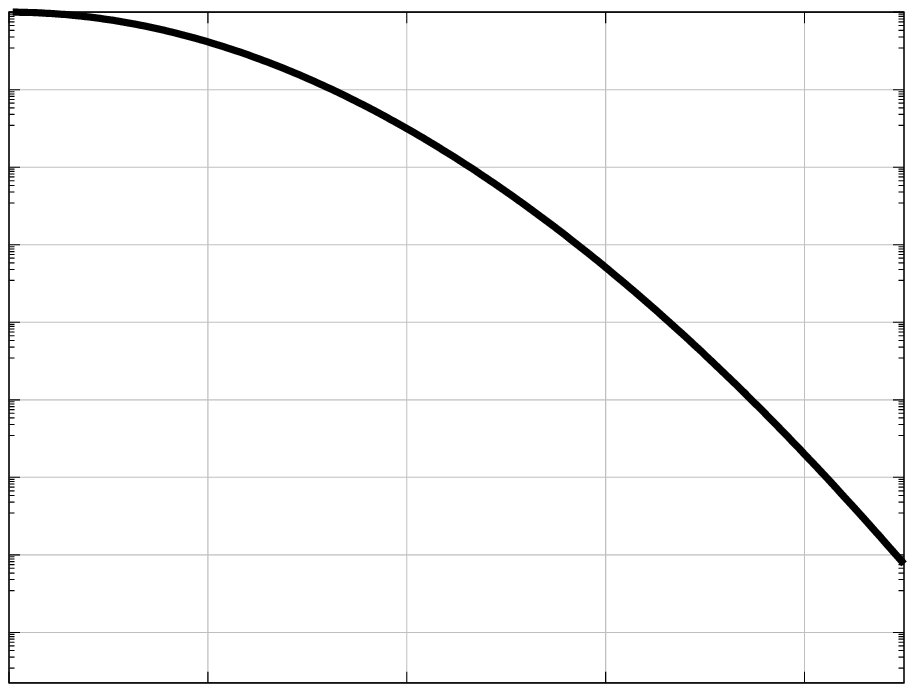}}}\\
{\scriptsize (a) initial condition} & {\scriptsize (b) singular values}
\end{tabular}
\caption{Advection equation: Plot (a) shows the initial condition and the direction of the transport. Plot (b) shows the slow decay of the singular values of the snapshots.}
\label{fig:AdVec:Motivation}
\end{figure}

\subsection{Problem formulation}
\label{sec:Prelim:ProblemFormulation}
It has been observed that states of problems with transport-dominated behavior can require DEIM (reduced) spaces $\Ucal$ with high dimensions, see, e.g., \cite{OHLBERGER2013901,dahmen_plesken_welper_2014,Tommasao}. The following toy example illustrates this by demonstrating the slow decay of the singular values of the snapshot matrix of solutions of the advection equation. Let $\Omega = (-1, 1) \subset \mathbb{R}$ be the spatial domain and consider the advection equation
\begin{equation}
\partial_tq(x, t) + \mu\partial_xq(x,t) = 0\,,\qquad x \in \Omega\,,
\label{eq:Prelim:AdVec}
\end{equation}
with periodic boundary conditions $q(1, t) = q(-1, t)$ and time $t \in [0, \infty)$. Set the initial condition to
\[
q(x, 0) = \frac{1}{\sqrt{\pi 0.02}}\exp\left(-\frac{x^2}{0.0002}\right)\,,
\]
which is the probability density function of a normal random variable with standard deviation 0.01 and mean 0, see Figure~\ref{fig:AdVec:Motivation}a. Discretize \eqref{eq:Prelim:AdVec} with a second-order upwind scheme and $\nh = 8192$ inner grid points in the spatial domain and time step size $\delta t = 10^{-6}$ and end time $T = 0.08$. The singular values of the trajectory $\bfQ(\mu)$ are plotted in Figure~\ref{fig:AdVec:Motivation}b for $\mu = 10$. According to the decay of the singular values, a DEIM space of dimension $\nr \geq 200$ is necessary to approximate the trajectory $\bfQ(\mu)$ with a projection error below $10^{-15}$ in the Euclidean norm, which is a rather slow decay compared to the fast decay observed in many other problems \cite{RozzaPateraSurvey,SIREVSurvey}. This example demonstrates that efficient reduced models of transport-dominated problems will have to exploit structure beyond the classical low-rank structure that traditional model reduction methods rely on.

\begin{figure}
\begin{tabular}{cc}
{\LARGE\resizebox{0.45\columnwidth}{!}{\input{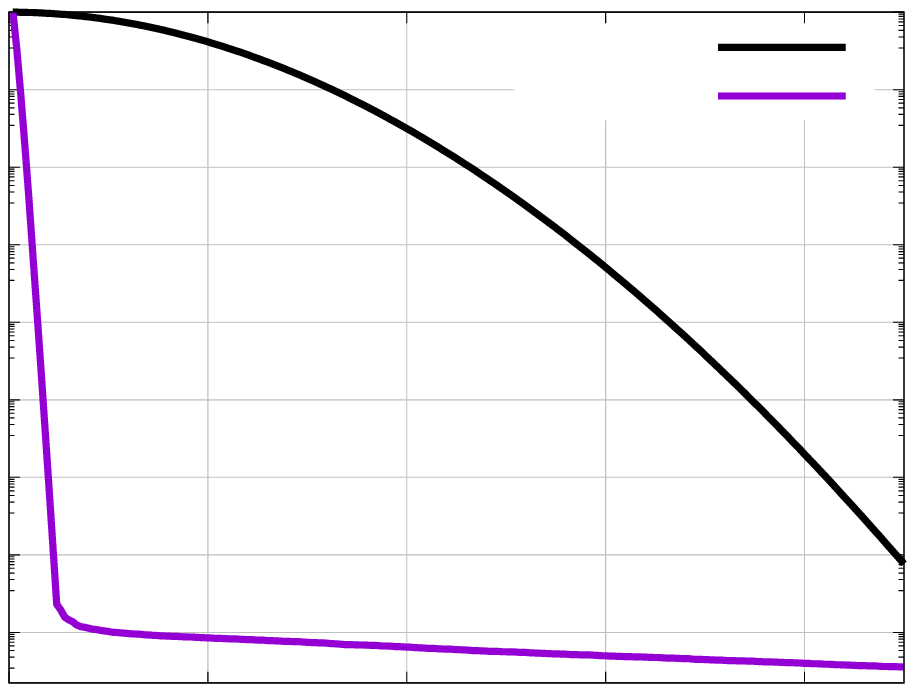}}} & {\LARGE\resizebox{0.45\columnwidth}{!}{\input{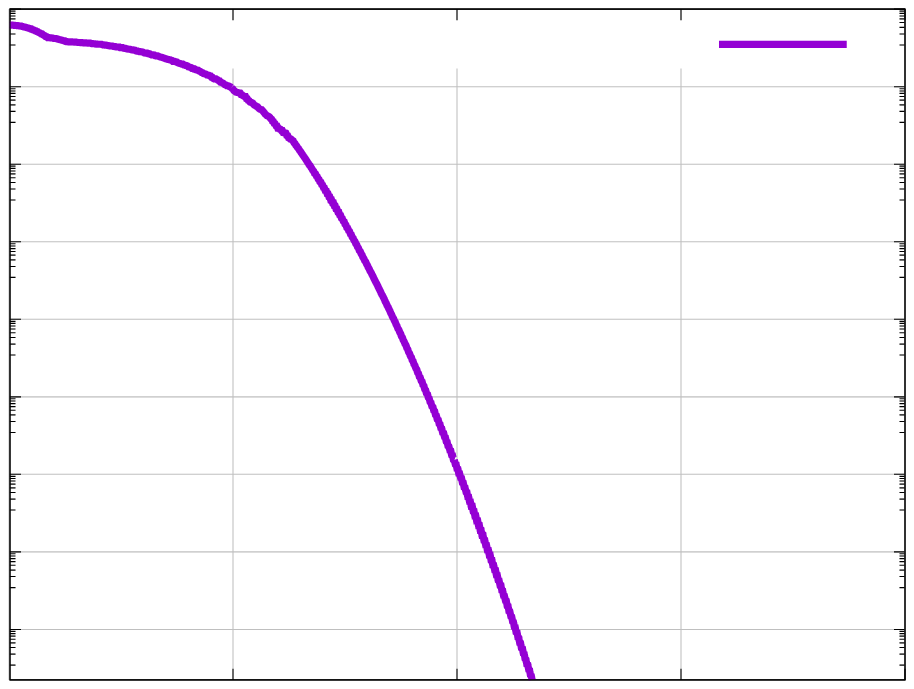}}}\\
{\scriptsize (a) local low-rank structure} & {\scriptsize (b) decay of residual (local coherence)}
\end{tabular}
\caption{Advection equation: The plot in (a) shows that the singular values of a local trajectory decay orders of magnitude faster than the singular values of a trajectory that is global in time in this example, which we exploit via online adaptive basis updates. Plot (b) indicates that the squared residual of DEIM approximations of states of transport-dominated problems decays rapidly, where the index on the x-axis refers to the component of the sorted squared residual. We will show that this fast decay of the residual means that basis updates can be derived from only few components---samples---of the residual.}
\label{fig:AdVec:SingularValues}
\end{figure}

\section{Exploiting local structure via online adaptive basis updates}
\label{sec:ABAS}
We propose \abas{} (adaptive bases and adaptive sampling) to exploit local structure to construct online adaptive reduced models of full models that exhibit transport-dominated behavior. We focus on two types of local structure: Local low-rankness that we exploit via online adaptive bases and local coherence that enables updating the bases from few samples of the full model. Section~\ref{sec:ABAS:LocalStructure} describes local low-rankness and local coherence in more detail. Section~\ref{sec:ABAS:LocalLowRank} discusses the adaptation of the DEIM basis to exploit local low-rank structure and Section~\ref{sec:ABAS:LocalCoherence} shows that only few samples are necessary to derive basis updates if the adapted DEIM spaces are locally coherent. Algorithm~\ref{alg:ABAS} in Section~\ref{sec:ABAS:Algorithm} summarizes the proposed approach and provides practical considerations.

In most of this section, we drop the dependency on the parameter $\bfmu$ of the function $\bff$, the state $\bfq_k$ at time step $k$, and the trajectory $\bfQ$, as well as their reduced counterparts. Parametrization of the proposed reduced models based on \abas{} is discussed in Section~\ref{sec:ABAS:Algorithm}.

\subsection{Local structure in transport-dominated problems}
\label{sec:ABAS:LocalStructure}
Consider the toy example introduced in Section~\ref{sec:Prelim:ProblemFormulation}. Let $\win \in \mathbb{N}$ be a window size and consider the local trajectory $\bfQ_k = [\bfq_{k - \win + 1}, \dots, \bfq_k] \in \mathbb{R}^{\nh \times \win}$ at time step $k$, which consists of the $\win$ states from time steps $k - \win + 1$ to time step $k$. Figure~\ref{fig:AdVec:SingularValues}a compares the singular values of the local trajectory $\bfQ_k$ to the singular values of the global trajectory $\bfQ$ for $\win = 500$ and $k = 1$. The results indicate that the singular values of the local trajectory $\bfQ_k$ decay orders of magnitude faster than the singular values of the global trajectory $\bfQ$. We call this behavior---that the singular values of local trajectories decay fast while the singular values of the global trajectory decay slowly---local low-rank structure. By adapting the DEIM spaces, we will exploit this local low-rank structure in Section~\ref{sec:ABAS:LocalLowRank}. Similar observations about local low-rank structure are exploited in \cite{Koch2007,SAPSIS20092347,NobileDO,Musharbash:231216,MUSHARBASH2018135,GERBEAU2014246,LDEIM,pehersto15dynamic}, cf.~Section~\ref{sec:Intro}. In Appendix~\ref{sec:Appx:Example}, we analyze analytically an example to demonstrate its low-rank structure.

Let us now consider the local DEIM space $\Ucal_k$ of dimension $\nr = 3$ and the corresponding local interpolation points matrix $\bfP_k$ obtained from the local trajectory $\bfQ_k$. The residual of approximating the state $\bfq_k$ at time step $k = 500$ with DEIM in $\Ucal_k$ is
\[
\bfr_k = \bfq_k - \bfU_k(\bfP_k^T\bfU_k)^{-1}\bfP_k^T\bfq_k.
\]
Figure~\ref{fig:AdVec:SingularValues}b shows the decay of the sorted squared components of $\bfr_k$. The squared residual decays fast in this example, which means that the residual is localized in a few components only. Thus, the residual is low at most of the $\nh = 8192$ components. Intuitively, such a fast decay of the residual means that the basis matrix $\bfU_k$ of the local DEIM space needs to be corrected at a few components only. We will show that such a fast decay of the residual can be the result of a local coherence structure of the local DEIM spaces, which we will make precise and will exploit with adaptive sampling in Section~\ref{sec:ABAS:LocalCoherence}.

\subsection{Exploiting local low-rank structure: Basis updates}
\label{sec:ABAS:LocalLowRank}
To exploit local low-rank structure as described in Section~\ref{sec:ABAS:LocalStructure}, we adapt the DEIM spaces and the DEIM interpolation points during the time steps $k = 1, \dots, K$ in the online phase. The adaptation is initialized with the DEIM basis matrix $\bfU_1$ and interpolation points matrix $\bfP_1$ at time step $k = 1$. Then, at each time step $k = 1, \dots, K$, the DEIM basis matrix $\bfU_k$ is adapted via an additive low-rank update to $\bfU_{k + 1}$. The update is based on ADEIM \cite{Peherstorfer15aDEIM}. The interpolation points matrix $\bfP_{k + 1}$ is derived with QDEIM from the adapted basis matrix $\bfU_{k + 1}$. In the following description, the DEIM interpolant is adapted at each time step $k = 1, \dots, K$ for ease of exposition, even though all of the following directly applies to situations where the adaptation is performed at selected time steps only, e.g., every other time step or via a criterion that decides adaptively when to update the DEIM interpolant.

\subsubsection{Adaptation with ADEIM}
Let $k$ be the current time step and $\bfU_k$ and $\bfP_k$ the DEIM basis matrix and the DEIM interpolation points matrix, respectively. Consider the matrix $\bfF_k \in \mathbb{R}^{\nh \times \win}$ and the coefficient matrix $\bfC_k = (\bfP_k^T\bfU_k)^{-1}\bfP_k^T\bfF_k$. The residual of the DEIM approximation of the columns of $\bfF_k$ is $\bfR_k = \bfU_k\bfC_k - \bfF_k$. Let now $\bfS_k = [\bfe_{s_1^{(k)}}, \dots, \bfe_{s_{\nms}^{(k)}}] \in \mathbb{R}^{\nh \times \nms}$ be the sampling points matrix corresponding to the $s_1^{(k)}, \dots, s_{\nms}^{(k)} \in \{1, \dots, \nh\}$ samplings points with $\nms > \nr$. ADEIM \cite{Peherstorfer15aDEIM} adapts the DEIM basis matrix $\bfU_k$ to $\bfU_{k + 1}$ with a low-rank update $\bfalpha_k\bfbeta_k^T \in \mathbb{R}^{\nh \times \nr}$
\[
\bfU_{k + 1} = \bfU_k + \bfalpha_k\bfbeta_k^T\,,
\]
with $\bfalpha_k \in \mathbb{R}^{\nh \times \nab}, \bfbeta_k \in \mathbb{R}^{\nr \times \nab}$ with rank $\nab \in \mathbb{N}$. The ADEIM update $\bfalpha_k\bfbeta_k^T$ is the rank-$\nab$ matrix that minimizes the residual $\bfU_{k + 1}\bfC_k - \bfF_k$ at the sampling points $\bfS_k$ in the Frobenius norm, which means that the ADEIM update $\bfalpha_k\bfbeta_k^T$ minimizes
\begin{equation}
\|\bfS_k^T\left((\bfU_k + \bfalpha_k\bfbeta_k^T)\bfC_k - \bfF_k\right)\|_F^2\,,
\label{eq:ABAS:Basis:MinObj}
\end{equation}
see \cite{Peherstorfer15aDEIM} for details on how to compute $\bfalpha_k\bfbeta_k^T$ and the computational costs of computing the update. Minimizing \eqref{eq:ABAS:Basis:MinObj} with respect to the update $\bfalpha_k\bfbeta_k^T$ keeps the coefficient matrix $\bfC_k$ fixed, which avoids introducing nonlinear interactions between $\bfC_{k + 1}$ and $\bfU_{k + 1}$ and so allows an efficient computation of the update $\bfalpha_k\bfbeta_k^T$ as described in \cite{Peherstorfer15aDEIM}. Note, first, that $\|\bfS_k^T(\bfU_{k + 1}\bfC_{k + 1} - \bfF_k)\|_F^2$ with coefficient matrix $\bfC_{k + 1}$ is upper bounded by \eqref{eq:ABAS:Basis:MinObj}, which means that by minimizing \eqref{eq:ABAS:Basis:MinObj} one minimizes an upper bound of the error that takes the updated coefficient matrix $\bfC_{k + 1}$ into account. Second, the basis update is performed many times (e.g., every other time step), and so the overall approach alternates between updating the basis by minimizing \eqref{eq:ABAS:Basis:MinObj} and updating the coefficient matrix when making a time step, which requires computing $\bfC_{k+1}$ corresponding to the updated basis matrix $\bfU_{k + 1}$; we refer to \cite{Peherstorfer15aDEIM} for details.

Critical for the adaptation is the matrix $\bfF_k$, because ADEIM adapts the space $\Ucal_k$ such that the residual of approximating the columns of $\bfF_k$ is minimized at the sampling points. A potentially good choice for the columns of $\bfF_k$ would be the states $\bfq_{k - \win + 1}, \dots, \bfq_{k}$ of the full model at time steps $k - \win + 1, \dots, k$; however, the states of the full model are unavailable because their availability would mean the full model has been solved. Instead, we set the columns of $\bfF_k$ as follows. Let $\breve{\bfS}_k \in \mathbb{R}^{\nh \times (\nh - \nms)}$ be the complementary sampling points matrix derived from the points $\{1, \dots, \nh\} \setminus \{s_1^{(k)}, \dots, s_{\nms}^{(k)}\}$ that have not been selected as sampling points. Let further $\hat{\bfq}_k \in \mathbb{R}^{\nh}$ be the vector with
\begin{equation}
\bfS_k^T\hat{\bfq}_k = \bfS_k^T\bff(\bfU_k\tilde{\bfq}_k)\,,\qquad \breve{\bfS}_k^T\hat{\bfq}_k = \breve{\bfS}_k^T\bfU_k(\bfP_k^T\bfU_k)^{-1}\bfP_k^T\bff(\bfU_k\tilde{\bfq}_k)\,,
\label{eq:ABAS:Basis:FChoice}
\end{equation}
which means that the components of $\hat{\bfq}_k$ corresponding to the sampling points in $\bfS_k$ are equal to the components of $\bff(\bfU_k\tilde{\bfq}_k)$ and all other components are approximated via DEIM given by $\bfU_k$ and $\bfP_k$. Note that $\tilde{\bfq}_k$ is the reduced state at time $k$ and that $\bff$ is the function that defines the full model \eqref{eq:Prelim:FOMProblem}. The matrix $\bfF_k$ that we use in the following for adaptation is
\[
\bfF_k = [\hat{\bfq}_{k - \win + 1}, \dots, \hat{\bfq}_k]\,.
\]
The vectors $\hat{\bfq}_{k - \win + 1}, \dots, \hat{\bfq}_k$ serve as surrogates of the full-model states $\bfq_{k - \win}, \dots, \bfq_{k - 1}$ to which the DEIM space is adapted, because the full-model states $\bfq_{k + 1}$ satisfy $\bfq_k = \bff(\bfq_{k + 1})$ for $k = 1, \dots, K$, see equation \eqref{eq:Prelim:FOMProblem}.

\subsubsection{Analysis of the error of the adapted space}
\label{sec:AADEIM:ErrorAnalysisUpdate}
We now provide an analysis of the ADEIM adaptation. Let $\Ucal$ and $\bar{\Ucal}$ be two $\nr$-dimensional subspaces of $\mathbb{R}^{\nh}$. We measure the distance between $\Ucal$ and $\bar{\Ucal}$ as
\begin{equation}
d(\bar{\Ucal}, \Ucal) = \|\bar{\bfU} - \bfU\bfU^T\bar{\bfU}\|_F^2\,,
\label{eq:ADEIM:AnalysisError}
\end{equation}
where $\bfU$ and $\bar{\bfU}$ are orthonormal basis matrices of $\Ucal$ and $\bar{\Ucal}$, respectively. The distance $d(\bar{\Ucal}, \Ucal)$ is symmetric and invariant under orthogonal basis transformations, which is true because  $d(\bar{\Ucal}, \Ucal) = \nr - \|\bfU^T\bar{\bfU}\|_F^2$ holds. The $\nr$-dimensional subspace of $\mathbb{R}^{\nh}$ to which we want to adapt at iteration $k$ is denoted as $\bar{\Ucal}_{k + 1}$ and the adapted space is $\Ucal_{k + 1}$. The following lemma quantifies the reduction of the residual after an ADEIM update and establishes Proposition~\ref{prop:Convergence} that bounds the error $d(\bar{\Ucal}_{k + 1}, \Ucal_{k + 1})$ of the adapted space $\Ucal_{k + 1}$ with respect to the space $\bar{\Ucal}_{k + 1}$.

\begin{lemma}
Let $\bfP^T_k\bfF_k$ have rank $\nr$ and let $\bfC_k = (\bfP^T_k\bfU_k)^{-1}\bfP_k^T\bfF_k$ be the coefficient matrix and $\bfR_k = \bfU_k\bfC_k - \bfF_k$ the residual matrix. Let further $\bfS_k$ be the sampling points matrix and $\bfU_{k + 1} = \bfU_k + \bfalpha_k\bfbeta_k^T$ the adapted basis matrix of the adapted space $\Ucal_{k + 1}$. Let $\bar{\nab}$ be the rank of $\bfS_k^T\bfR_k$ and let $\nab \in \mathbb{N}$ with $\nab \leq \bar{\nab}$ be the rank of the update $\bfalpha_k\bfbeta_k^T$. Then,
\[
\left\|\bfS^T_k\left(\bfU_{k + 1}\bfC_k - \bfF_k\right)\right\|_F^2 = \|\bfS^T_k\bfR_k\|_F^2 - \sum_{i = 1}^r \sigma_i^2\,,
\]
where $\sigma_1 \geq \sigma_2 \geq \dots \geq \sigma_{\bar{\nab}} > 0$ are the singular values of $\bfS_k^T\bfR_k$.
\label{lm:ADEIM:SingValReduction}
\end{lemma}
\begin{proof}
Because $\bfP^T_k\bfF_k$ has rank $\nr$, and because $\bfP^T_k\bfU_k$ is regular given the interpolation points are selected by standard empirical interpolation algorithms \cite{barrault_empirical_2004,deim2010,QDEIM}, the coefficient matrix $\bfC_k$ has full row rank. Then, the statement of this lemma follows from \cite[Lemma~3.5]{Peherstorfer15aDEIM} by transforming the generalized symmetric positive definite eigenproblem into the symmetric eigenproblem with matrix $(\bfS_k^T\bfR_k)^T(\bfS_k^T\bfR_k)$, which then is equivalent to computing the singular value decomposition of $\bfS_k^T\bfR_k$ so that the squared singular values of $\bfS_k^T\bfR_k$ are equal to the eigenvalues of the generalized eigenproblem. This shows this lemma by using the last identity of \cite[proof of Lemma~3.5]{Peherstorfer15aDEIM}.
\end{proof}

\begin{proposition}
Consider the same setup as in Lemma~\ref{lm:ADEIM:SingValReduction} and additionally assume that $\bfF_k$ has rank $\nr$ and its columns are in the $\nr$-dimensional space $\bar{\Ucal}_{k + 1}$. Let further $\Ucal_{k + 1}$ be the adapted space derived with the rank-$\nab$ ADEIM update and sampling points matrix $\bfS_k$. Then, it holds
\begin{equation}
d(\Ucal_{k + 1}, \bar{\Ucal}_{k + 1}) \leq \frac{\rho_k^2}{\sigma_{\text{min}}^2(\bfF_k)}
\label{eq:ABAS:ErrorBound}
\end{equation}
with
\begin{equation}
\rho_k^2 = \|\breve{\bfS}_k^T\bfR_k\|_F^2 + \sum_{i = \nab + 1}^{\bar{\nab}}\sigma_i^2\,,
\label{eq:ADEIMProp:FactorRho}
\end{equation}
where $\breve{\bfS}_k$ denotes the complementary sampling matrix of $\bfS_k$ and $\sigma_{\text{min}}(\bfF_k)$ is the smallest non-zero singular value of $\bfF_k$.
\label{prop:Convergence}
\end{proposition}
\begin{proof}
First, note that the matrix $\breve{\bfS}_k^T\bfalpha_k\bfbeta_k^T$ has zero entries only, because the ADEIM update $\bfalpha_k\bfbeta_k^T$ updates only rows of $\bfU_k$ corresponding to the sampling points $\bfS_k$, see \cite[Lemma~3.5]{Peherstorfer15aDEIM}. Thus, it holds
\begin{equation}
\breve{\bfS}_k^T\bfU_{k+1} = \breve{\bfS}_k^T\bfU_k\,.
\label{eq:ADEIMProp:NoChange}
\end{equation}
Consider now the norm of the residual with respect to the adapted space $\Ucal_{k + 1}$
\begin{align}
\|\bfU_{k+1}\bfC_k - \bfF_k\|_F^2 = & \|\bfS_k^T(\bfU_{k+1}\bfC_k - \bfF_k)\|_F^2 + \|\breve{\bfS}_k^T(\bfU_{k+1}\bfC_k - \bfF_k)\|_F^2\label{eq:ADEIMProp:A}\\
= & \|\bfS_k^T(\bfU_k\bfC_k - \bfF_k)\|_F^2 - \sum_{i = 1}^{\nab} \sigma_i^2 + \|\breve{\bfS}_k^T(\bfU_{k+1}\bfC_k - \bfF_k)\|_F^2\label{eq:ADEIMProp:B} \\
= & \|\bfS_k^T(\bfU_k\bfC_k - \bfF_k)\|_F^2 -  \sum_{i = 1}^{\nab} \sigma_i^2 + \|\breve{\bfS}_k^T(\bfU_k\bfC_k - \bfF_k)\|_F^2\label{eq:ADEIMProp:C} \\
= & \sum_{i = \nab+1}^{\bar{\nab}} \sigma_i^2 + \|\breve{\bfS}_k^T\bfR_k\|_F^2 \label{eq:ADEIMProp:D}\,.
\end{align}
Equation \eqref{eq:ADEIMProp:A} follows because $\breve{\bfS}_k$ is the complementary sampling points matrix of $\bfS_k$. Equation \eqref{eq:ADEIMProp:B} follows from Lemma~\ref{lm:ADEIM:SingValReduction}. Equation \eqref{eq:ADEIMProp:C} holds because of \eqref{eq:ADEIMProp:NoChange}. Equation~\eqref{eq:ADEIMProp:D} uses $\bfR_k = \bfU_k\bfC_k - \bfF_k$ and that $\sum_{i = 1}^{\bar{r}} \sigma_i^2 = \|\bfS_k^T\bfR_k\|_F^2$. Consider now the projection of the columns of $\bfF_{k}$ onto $\Ucal_{k + 1}$ and observe that
\begin{equation}
\|\bfF_k - \bfU_{k + 1}\bfU_{k + 1}^T\bfF_k\|_F^2 \leq \|\bfU_{k + 1}\bfC_k - \bfF_k\|_F^2\,.
\label{eq:ADEIMProp:ProjBetterThanReg}
\end{equation}
Note that the basis matrix $\bfU_{k + 1}$ obtained with the ADEIM update is not necessarily orthonormal and therefore it might be necessary to consider the oblique projection $\bfU_{k + 1}(\bfU_{k + 1}^T\bfU_{k + 1})^{-1}\bfU_{k + 1}^T\bfF_k$ onto the column span $\Ucal_{k + 1}$ of $\bfU_{k + 1}$ in \eqref{eq:ADEIMProp:ProjBetterThanReg}; however, the projection error in the Frobenius norm is the same for both projections. With \eqref{eq:ADEIMProp:D} and the definition of $\rho_k^2$ follows
\begin{equation}
\|\bfF_k - \bfU_{k + 1}\bfU_{k + 1}^T\bfF_k\|_F^2 \leq \rho_k^2\,.
\label{eq:ADEIMProp:HelperHelper1}
\end{equation}
Let now $\bar{\bfU}_{k + 1}$ be an orthonormal basis matrix of $\bar{\Ucal}_{k + 1}$. Since $\bar{\Ucal}_{k + 1}$ is spanned by the columns of $\bfF_k$, there exists a full-rank matrix $\tilde{\bfF}_k \in \mathbb{R}^{\nr \times \win}$ such that $\bfF_k = \bar{\bfU}_{k + 1}\tilde{\bfF}_{k}$. We obtain
\begin{equation}
\begin{aligned}
\|\bfF_k - \bfU_{k + 1}\bfU_{k + 1}^T\bfF_k\|_F^2  = & \|(\bar{\bfU}_{k + 1} - \bfU_{k + 1}\bfU_{k + 1}^T\bar{\bfU}_{k + 1})\tilde{\bfF}_k\|_F^2\\
 \geq & \|\bar{\bfU}_{k + 1} - \bfU_{k + 1}\bfU_{k + 1}^T\bar{\bfU}_{k + 1}\|_F^2\sigma_{\text{min}}^2(\tilde{\bfF}_k)\\
 = & \|\bar{\bfU}_{k + 1} - \bfU_{k + 1}\bfU_{k + 1}^T\bar{\bfU}_{k + 1}\|_F^2\sigma_{\text{min}}^2(\bfF_k)\,,
\end{aligned}
\label{eq:ADEIMProp:HelperHelper2}
\end{equation}
where we used $\sigma_{\text{min}}(\tilde{\bfF}_k) = \sigma_{\text{min}}(\bfF_k)$, which holds because $\bar{\bfU}_{k + 1}$ is orthonormal. The matrix $\bfF_k$ has rank $\nr$ and therefore there are $\nr$ non-zero singular values corresponding to the basis vectors in $\bar{\bfU}_{k + 1}$. Combining \eqref{eq:ADEIMProp:HelperHelper2} with \eqref{eq:ADEIMProp:HelperHelper1} and the definition of $d(\bar{\Ucal}_{k + 1}, \Ucal_{k + 1})$ shows \eqref{eq:ABAS:ErrorBound}. Note that if $\bfU_{k + 1}$ is not orthonormal, then the oblique projection is used in \eqref{eq:ADEIMProp:HelperHelper2} to obtain $\|\bar{\bfU}_{k + 1} - \bfU_{k + 1}(\bfU_{k + 1}^T\bfU_{k + 1})^{-1}\bfU_{k + 1}^T\bar{\bfU}_{k + 1}\|_F^2 \sigma_{\text{min}}^2(\bfF_k) \leq \rho_k^2$, which leads to \eqref{eq:ABAS:ErrorBound} as well because $\|\bar{\bfU}_{k + 1} - \bfU_{k + 1}(\bfU_{k + 1}^T\bfU_{k + 1})^{-1}\bfU_{k + 1}^T\bar{\bfU}_{k + 1}\|_F^2 = \|\bar{\bfU}_{k + 1} - \hat{\bfU}_{k + 1}\hat{\bfU}_{k + 1}^T\bar{\bfU}_{k + 1}\|_F^2 = d(\bar{\Ucal}_{k + 1}, \Ucal_{k + 1})$ for an orthonormal basis matrix $\hat{\bfU}_{k + 1}$ that spans the same column space $\Ucal_{k + 1}$ as $\bfU_{k + 1}$.

\end{proof}

\subsection{Exploiting local coherence: Adaptive sampling}
\label{sec:ABAS:LocalCoherence}
Proposition~\ref{prop:Convergence} shows that the choice of the sampling points $\bfS_k$ influences the bound of the error $d(\bar{\Ucal}_{k + 1}, \Ucal_{k + 1})$ of the adapted space $\Ucal_{k + 1}$. In this section, we derive an adaptive sampling strategy that minimizes the upper bound derived in Proposition~\ref{prop:Convergence} in case of full-rank ADEIM updates. Then, we show that if the adaptive sampling strategy is used, the error $d(\bar{\Ucal}_{k + 1}, \Ucal_{k + 1})$ decays at least as fast with the number of sampling points $\nms$ as the norm of the DEIM residual, which in turn means that only few sampling points are necessary to adapt the basis if the residual decays fast.

\subsubsection{Adaptive sampling strategy based on residual}
\label{sec:ABAS:LocalCoherence:AdaptSampling}
Following Proposition~\ref{prop:Convergence} and the decay factor $\rho_k$ defined in \eqref{eq:ADEIMProp:FactorRho}, we select the sampling points so that $\|\breve{\bfS}_k^T\bfR_k\|_F^2$ is minimized. Note that $\sigma_{\text{min}}(\bfF_k)$ in \eqref{eq:ABAS:ErrorBound} is independent of the sampling points $\bfS_k$ and therefore it is sufficient to derive sampling points that lead to a small decay factor $\rho_k$.

Consider the component-wise residual
\begin{equation}
r_i = \|\bfR_k^T\bfe_i\|_2^2
\label{eq:ABAS:Coherence:CompResidual}
\end{equation}
and let $j_1, \dots, j_{\nh}$ be an ordering of $1, \dots, \nh$ such that $r_{j_1} \geq r_{j_2} \geq \dots \geq r_{j_{\nh}}$. Then, select the first $j_1, \dots, j_{\nms}$ components as sampling points and form the corresponding sampling points matrix
\begin{equation}
\bfS_k = [\bfe_{j_1}, \dots, \bfe_{j_{\nms}}]\,.
\label{eq:ABAS:Coherence:AdaptSamplingS}
\end{equation}
If a full-rank ADEIM update is applied, i.e., $\bar{\nab} = \nab$ in Proposition~\ref{prop:Convergence}, then this choice of sampling points is optimal in the sense that the bound $\rho_k$ is minimized.

We now show that a fast decay in the residual implies a fast decay in the error bound of $d(\bar{\Ucal}_{k + 1}, \Ucal_{k + 1})$.

\begin{proposition}
Consider the same setup as in Proposition~\ref{prop:Convergence}. Let $j_1, \dots, j_{\nh}$ be an ordering of $\{1, \dots, \nh\}$ such that the component-wise residual defined in \eqref{eq:ABAS:Coherence:CompResidual} decays as
\begin{equation}
r_{j_i} \leq c_1 \mathrm e^{-c_2 i^{\arate}}\,,
\label{eq:ABAS:Coherence:ResidualDecay}
\end{equation}
with rate $\arate > 1$ and constants $c_1, c_2 > 0$. The error of the adapted space $\Ucal_{k + 1}$ is bounded as
\begin{equation}
d(\bar{\Ucal}_{k + 1}, \Ucal_{k + 1}) \leq c_3\mathrm e^{-c_2 \nms^{\arate}}\,,
\label{eq:ABAS:Coherence:ErrorBoundNrSamples}
\end{equation}
if a full-rank ADEIM update is applied and $\nms$ sampling points are selected via the adaptive sampling \eqref{eq:ABAS:Coherence:AdaptSamplingS}. The constant $c_3 = c_1/((1 - \mathrm e^{-c_2})\sigma_{\text{min}}^2(\bfF_k))$ is independent of $\nms$. In particular, setting the number of sampling points to
\begin{equation}
m \geq \operatorname{min}\left\{\nh, \left(-\frac{1}{c_2}\log\left(\frac{\epsilon }{c_3}\right)\right)^{1/\arate}\right\}
\label{eq:CorConvergenceExp:M}
\end{equation}
guarantees $d(\bar{\Ucal}_{k + 1}, \Ucal_{k + 1}) \leq \epsilon$ for a threshold $\epsilon > 0$.
\label{cor:NumberM}
\end{proposition}
\begin{proof}
In case of a full-rank update, the factor $\rho_k^2$ in \eqref{eq:ABAS:ErrorBound} is
\[
\rho_k^2 = \|\breve{\bfS}_k^T\bfR_k\|_F^2\,,
\]
as shown in Proposition~\ref{prop:Convergence}. Then, we obtain with \eqref{eq:ABAS:Coherence:ResidualDecay} and the adaptive sampling \eqref{eq:ABAS:Coherence:AdaptSamplingS}
\begin{equation}
\begin{aligned}
\|\breve{\bfS}_k^T\bfR_k\|_F^2 = & \sum_{i = \nms + 1}^{\nh} \|\bfR_k^T\bfe_{j_i}\|_2^2 \leq c_1 \sum_{i = \nms + 1}^{\nh} \mathrm e^{-c_2 i^{\arate}} \leq c_1 \sum_{i = \nms + 1}^{\infty} \mathrm e^{-c_2 i^{\arate}}\\
\leq & c_1\sum_{i = 0}^{\infty} \mathrm e^{-c_2 (i + (\nms + 1))^{\arate}} \leq c_1 \mathrm e^{-c_2 (\nms + 1)^{\arate}} \sum_{i = 0}^{\infty} \mathrm e^{-c_2 i^{\arate}}\,,
\end{aligned}
\label{eq:NumberM:HelperZ}
\end{equation}
where the last inequality holds because $i \geq 0, \arate > 1$ and thus $(i + (m + 1))^{\arate} \geq i^{\arate} + (m + 1)^{\arate}$. Using that $\arate > 1$, we obtain $i^\arate \geq i$ and $(m + 1)^{\arate} \geq m^{\arate}$ and thus
\begin{equation}
\|\breve{\bfS}_k^T\bfR_k\|_F^2 \leq c_1 \mathrm e^{-c_2 (\nms + 1)^{\arate}} \sum_{i = 0}^{\infty} \mathrm e^{-c_2 i^{\arate}} \leq c_1 \mathrm e^{-c_2 \nms^{\arate}} \sum_{i = 0}^{\infty} \mathrm e^{-c_2 i} = c_1 e^{-c_2 \nms^{\arate}} \frac{1}{\mathrm 1 - e^{-c_2}}\,.
\label{eq:NumberM:Helper2}
\end{equation}
Set $c_3 = c_1/((1 - \mathrm e^{-c_2})\sigma^2_{\text{min}}(\bfF_k))$ to obtain \eqref{eq:ABAS:Coherence:ErrorBoundNrSamples} with Proposition~\ref{prop:Convergence}. Setting $\nms$ as in \eqref{eq:CorConvergenceExp:M} shows $d(\bar{\Ucal}_{k + 1}, \Ucal_{k + 1}) \leq \epsilon$.
\end{proof}

Proposition~\ref{cor:NumberM} is formulated for an exponential decay of the residual \eqref{eq:ABAS:Coherence:ResidualDecay}; however, the steps in the proof of Proposition~\ref{eq:ABAS:Coherence:ResidualDecay} seem to generalize to other decay behavior of the residual such as an algebraic decay. The key in the proof of Proposition~\ref{eq:ABAS:Coherence:ResidualDecay} is \eqref{eq:NumberM:HelperZ}, which builds an argument based on a geometric series that the sum of the norms of the residual decays exponentially in the number of sampling points $\nms$. In case of, e.g., an algebraic decay of the sum of norms, we expect that such an upper bound would decay algebraically. Then, the rate of the residual decay can be carried over to the decay of the error $d(\bar{\Ucal}_{k + 1}, \Ucal_{k + 1})$ of the adapted space $\Ucal_{k + 1}$ with mild modifications of the proof of Proposition~\ref{cor:NumberM}.

\subsubsection{Local coherence}
Following, e.g., \cite{Candes2009}, define the coherence of a space $\Ucal$ as
\[
\gamma(\Ucal) = \frac{\nh}{\nr}\max_{i = 1, \dots, \nh} \|\bfU^T\bfe_i\|_2^2\,,
\]
where $\bfU$ is an orthonormal basis of $\Ucal$ and $\bfe_i$ is the $i$-th canonical unit vector. Define further the local coherence as
\[
\gamma_i(\Ucal) = \frac{\nh}{\nr}\|\bfU^T\bfe_i\|_2^2\,,
\]
for $i = 1, \dots, \nh$, see, e.g., \cite{6651836,NIPS2015_6018}. We refer to \cite{0266-5611-23-3-008,Candes2009,6651836,NIPS2015_6018} for details.

We now show that the rate of the decay of the local coherence of the current DEIM space $\Ucal_{k}$ and of the space $\bar{\Ucal}_{k + 1}$ to which we want to adapt is reflected in the decay of the bound of the error $d(\bar{\Ucal}_{k + 1}, \Ucal_{k + 1})$ of the adapted space $\Ucal_{k + 1}$ with respect to the number of sampling points $\nms$. Thus, we now relate the decay of the error $d(\bar{\Ucal}_{k + 1}, \Ucal_{k + 1})$ of the adapted space to a property of the spaces $\Ucal_k$ and $\bar{\Ucal}_{k + 1}$, namely their decay of the local coherence.

\begin{lemma}
Let $j_1, \dots, j_{\nh}$ be an ordering of $\{1, \dots, \nh\}$ such that
\begin{equation}
\gamma_{j_i}(\Ucal_k) \leq c_4 \exp\left(- c_5 i^{\arate}\right)\,,\qquad \gamma_{j_i}(\bar{\Ucal}_{k + 1}) \leq \bar{c}_4 \exp\left(-\bar{c}_5 i^{\bar{\arate}}\right)
\label{eq:LinkCohResidual:Assumptions}
\end{equation}
holds for $i = 1, \dots, \nh$, with $c_4, \bar{c}_4, c_5, \bar{c}_5 > 0$ and $\arate, \bar{\arate} > 1$. Let the columns of $\bfF_k$ be in $\bar{\Ucal}_{k + 1}$. Define $\bfC_k = (\bfP^T_k\bfU_k)^{-1}\bfP_k^T\bfF_k$, then for the residual $\bfR_k = \bfU_k\bfC_k - \bfF_k$ holds
\begin{equation}
\|\bfR^T_k\bfe_{j_i}\|_2^2 \leq c_7(\Lambda_k)\frac{\nr}{\nh}\|\bfF_k\|_2^2\mathrm e^{-\min\left\{c_5, \bar{c}_5\right\} i^{\min\left\{\arate, \bar{\arate}\right\}}}\,,\qquad i = 1, \dots, \nh\,,
\label{eq:LinkCohResidual:Result}
\end{equation}
with a constant $c_7(\Lambda_k) > 0$ that is independent of $\arate, \bar{\arate}, c_5, \bar{c}_5$ and that depends on $\Lambda_k = \|(\bfP^T_k\bfU_k)^{-1}\|_2$.
\label{prop:LinkCohResidual}
\end{lemma}
\begin{proof}
Denote with $\bfU_k^{(i)}$, $\bar{\bfU}_{k + 1}^{(i)}$, and $\bfF_k^{(i)}$ the $i$-th row of $\bfU$, $\bar{\bfU}_{k + 1}$, and $\bfF_k$, respectively. The matrix $\bar{\bfU}_{k + 1}$ is an orthonormal basis matrix of $\bar{\Ucal}_{k + 1}$. Let further $\tilde{\bfF}_k \in \mathbb{R}^{\nr \times \win}$ be a matrix such that $\bfF_k = \bar{\bfU}_{k + 1}\tilde{\bfF}_k$. Note that $\|\bfF_k\|_2 = \|\tilde{\bfF}_k\|_2$, because $\bar{\bfU}_{k + 1}$ is orthonormal. Then, with $\bfC_k = (\bfP_k^T\bfU_k)^{-1}\bfP_k^T\bfF_k$, follows
\[
\begin{aligned}
\|\bfR^T_k\bfe_{j_i}\|_2^2 = & \|\bfF_k^{(j_i)} - \bfU_k^{(j_i)}\bfC_k\|_2^2 \\
= & \|\bfF_k^{(j_i)}\|_2^2 - 2 \bfF_k^{(j_i)}(\bfU_k^{(j_i)}\bfC_k)^T  + \|\bfU_k^{(j_i)}\bfC_k\|_2^2\\
\leq & \|\bfF_{k}^{(j_i)}\|_2^2 + 2 \|\bfF_k^{(j_i)}\|_2\|\bfU_k^{(j_i)}\|_2\|\bfC_k\|_2 + \|\bfU_k^{(j_i)}\|_2^2\|\bfC_k\|_2^2\,.
\end{aligned}
\]
Now make the following estimate
\[
\|\bfF_k^{(j_i)}\|_2^2 \leq \|\bar{\bfU}_{k + 1}^{(j_i)}\|_2^2 \|\tilde{\bfF}_k\|_2^2 = \|\bar{\bfU}_{k + 1}^{(j_i)}\|_2^2\|\bfF_k\|_2^2\,,
\]
because $\|\bfF_k\|_2 = \|\tilde{\bfF}_k\|_2$. Further, we have
\[
\|\bfC_k\|_2 \leq \|(\bfP_k^T\bfU_k)^{-1}\|_2\|\bfF_k\|_2\,.
\]
Set now $\bar{c}_6 = \bar{c}_4$ and $c_6 = c_4 \|(\bfP_k^T\bfU_k)^{-1}\|_2^2$ to bound
\[
\|\bfR_k^T\bfe_{j_i}\|_2^2 \leq \frac{\nr}{\nh}\|\bfF_k\|_2^2\left(\bar{c}_6\mathrm e^{-\bar{c}_5 i^{\bar{\arate}}} + 2\sqrt{\bar{c}_6 c_6}\mathrm e^{-\frac{\bar{c}_5}{2}i^{\bar{\arate}}-\frac{c_5}{2}i^{\arate}} + c_6\mathrm e^{-c_5i^{\arate}}\right)\,.
\]
Now set $c_7(\Lambda_k) = c_6 + \bar{c}_6 + 2\sqrt{c_6\bar{c}_6}$ to obtain
\[
\|\bfR_k^T\bfe_{j_i}\|_2^2 \leq c_7(\Lambda_k)\frac{\nr}{\nh}\|\bfF_k\|_2^2 \mathrm e^{-\min\{c_5,\bar{c}_5\}i^{\min\{\arate, \bar{\arate}\}}}\,,
\]
which shows the proposition.
\end{proof}

Since we consider subspaces of finite-dimensional spaces only, i.e., $\nh$ is finite, the bounds in \eqref{eq:LinkCohResidual:Assumptions} hold for any subspace by increasing the constants and choosing the rate close to 1. However, Lemma~\ref{prop:LinkCohResidual} still is meaningful because it shows that the constants and rates that appear in \eqref{eq:LinkCohResidual:Assumptions} are obtained in the bound of the residual in \eqref{eq:LinkCohResidual:Result} as well. Thus, the local coherence structure is directly reflected in the decay of the DEIM residual. The constant $c_7(\Lambda_k)$ depends on $\Lambda_k = \|(\bfP_k^T\bfU_k)^{-1}\|_2$, which is a well-studied quantity in the context of empirical interpolation and typically found to be low in practice \cite{barrault_empirical_2004,deim2010,QDEIM,PDG18ODEIM}.

The following proposition combines the local coherence structure exploited in Lemma~\ref{prop:LinkCohResidual} and the adaptive sampling of Proposition~\ref{cor:NumberM} to derive the number of sampling points $\nms$ that are required to achieve $d(\bar{\Ucal}_{k + 1}, \Ucal_{k + 1}) \leq \epsilon$ in probability. We draw random samples from the space $\bar{\Ucal}_{k + 1}$ to avoid making assumptions on the specific right-hand side matrix $\bfF_k$, which still allows us to bound in probability the condition number of $\bfF_k$ that is required when applying Proposition~\ref{cor:NumberM} in the following. We draw a random sample from the space $\bar{\Ucal}_{k + 1}$ by taking an i.i.d.~standard Gaussian matrix $\tilde{\bfF}_k$ and multiplying it with the basis matrix $\bar{\bfU}_{k + 1}$. Other distributions over the elements of $\bar{\Ucal}_{k + 1}$ can be considered to derive similar statements as in the proposition below.

\begin{proposition}
Assume that Lemma~\ref{prop:LinkCohResidual} applies and consider the same setting as in Proposition~\ref{prop:Convergence} except that $\bfF_k = \bar{\bfU}_{k + 1}\tilde{\bfF}_{k}$ with $\tilde{\bfF}_k$ being an $\nr \times \nr$ matrix with independent and identically distributed (i.i.d.) standard Gaussian entries and $\bfP_k^T\bar{\bfU}_{k + 1}$ having full rank $\nr$. Then, $d(\bar{\Ucal}_{k + 1}, \Ucal_{k + 1}) \leq \epsilon$ with probability at least $1 - \delta$ if a full-rank update is applied and
\begin{equation}
m \geq \operatorname{min}\left\{\nh, \left(-\frac{1}{\min\{c_5, \bar{c}_5\}} \log\left(\frac{\epsilon \delta^2 \nh}{c_8(\Lambda_k)\nr}\right)\right)^{1/\min\{\arate, \bar{\arate}\}}\right\}\,,
\label{eq:ABAS:Coherence:RandomM}
\end{equation}
where $c_8(\Lambda_k)$ is a constant independent of $\nms, \arate$, and $\bar{\arate}$.
\label{prop:Random}
\end{proposition}
\begin{proof}
Since Lemma~\ref{prop:LinkCohResidual} applies, we have with \eqref{eq:NumberM:Helper2} that
\[
\|\breve{\bfS}_k^T\bfR_k\|_F^2 \leq \frac{c_7(\Lambda_k)\frac{\nr}{\nh}\|\bfF_k\|_2^2}{1 - \mathrm e^{-\min\{c_5, \bar{c}_5\}}} \mathrm e^{-\min\{c_5, \bar{c}_5\} \nms ^{\min\{\arate, \bar{\arate}\}}}\,.
\]
With Proposition~\ref{prop:Convergence}, which is applicable with a full-rank update because $\bfP^T_k\bfF_k = \bfP^T\bar{\bfU}_{k + 1}\tilde{\bfF}_k$ has full row rank $\nr$ if $\tilde{\bfF}_k$ is regular, follows
\[
d(\bar{\Ucal}_{k + 1}, \Ucal_{k + 1}) \leq \frac{c_7(\Lambda_k)\frac{\nr}{\nh}\kappa^2(\bfF_k)}{1 - \mathrm e^{-\min\{c_5, \bar{c}_5\}}}  \mathrm e^{-\min\{c_5, \bar{c}_5\}\nms^{\text{min}\{\arate, \bar{\arate}\}}}\,,
\]
where $\kappa(\bfF_k)$ is the spectral condition number of $\bfF_k$. Since $\bfF_k = \bar{\bfU}_{k + 1}\tilde{\bfF}_k$ with $\bar{\bfU}_{k + 1}$ orthonormal, we have $\kappa(\bfF_k) = \kappa(\tilde{\bfF}_k)$. We now bound the spectral condition number of $\tilde{\bfF}_k$ with high probability by exploiting that $\tilde{\bfF}_k$ has i.i.d.~standard Gaussian entries. The work \cite[Theorem~1.1]{EdelmanTailBounds} and \cite[Theorem~1.1]{doi:10.1137/S0895479803429764} show that the spectral condition number of $\tilde{\bfF}_k$ is bounded in probability as
\[
P[\kappa(\tilde{\bfF}_k) \geq \delta] \leq c \delta^{-1}\,,
\]
for $\delta > 0$ and with a positive constant $c$ that depends $\nr$. Thus, we have $P[\kappa(\tilde{\bfF}_k) \leq c\delta^{-1}] \geq 1 - \delta$ and $P[\kappa^2(\tilde{\bfF}_k) \leq c^2\delta^{-2}] \geq 1 - \delta$, which leads to
\[
d(\bar{\Ucal}_{k + 1}, \Ucal_{k + 1}) \leq \frac{c_8}{\delta^2}\frac{\nr}{\nh} \mathrm e^{-\min\{c_5, \bar{c}_5\}\nms^{\text{min}\{\arate, \bar{\arate}\}}}
\]
with probability at least $1 - \delta$ and $c_8(\Lambda_k) = (c_7(\Lambda_k) c^2)/(1 - \mathrm e^{-\min\{c_5, \bar{c}_5\}})$. Setting $\nms$ as in \eqref{eq:ABAS:Coherence:RandomM} leads to $d(\bar{\Ucal}_{k + 1}, \Ucal_{k + 1}) \leq \epsilon$ with probability $1 - \delta$.
\end{proof}

In the following, we focus on problems with moving coherent structures that are local in the spatial domain. Then, typically, the residual \eqref{eq:ABAS:Coherence:CompResidual} decays fast, cf.~Section~\ref{sec:Prelim:ProblemFormulation}, and local basis updates with \abas{} are sufficient. In contrast, if the basis updates have to take into account new features that are global in the spatial domain, then the residual \eqref{eq:ABAS:Coherence:CompResidual} might decay slowly. One can then either perform a global update in the sense that the number of sampling points is $\nms = \nh$ or perform multiple local updates. Since the adaptive sampling strategy selects the points with the largest residual, multiple local updates can have a similar effect as one global update.

\subsection{Practical considerations and algorithm}
\label{sec:ABAS:Algorithm}
We now provide details on the practical implementation of the \abas{} approach and summarize \abas{} in Algorithm~\ref{alg:ABAS}.

\subsubsection{Practical considerations}
The \abas{} model is initialized with a DEIM interpolant with basis matrix $\bfU_1$ and $\bfP_1$. We propose to construct $\bfU_1$ and $\bfP_1$ from a local trajectory computed with the full model. This means that the full model is solved for $\winit \in \mathbb{N}$ time steps, with $\winit \ll K$, and the initial DEIM interpolant $(\bfU_1, \bfP_1)$ is constructed from the corresponding full-model states. Initializing the \abas{} model with a DEIM interpolant obtained from full-model states has two benefits. First, no offline phase is necessary to initialize the \abas{} model, which means that it is unnecessary to develop (e.g., greedy) strategies to build an initial \abas{} model offline. Second, because there is no offline phase, the \abas{} model is initialized and then adapted for the parameter at hand. Thus, an explicit parametrization of the \abas{} model is unnecessary, which avoids the common challenges of parametrized model reduction \cite{RozzaPateraSurvey,SIREVSurvey} and sampling---potentially high-dimensional---parameter spaces. However, having no offline phase also means that information available from before the online phase is ignored. For example, if only certain regions in the parameter domain lead to transport-dominated phenomena, then a reduced basis constructed in the traditional way with, e.g., greedy methods \cite{prudhomme_reliable_2001,veroy_posteriori_2003,RozzaPateraSurvey}, could be used without adaptation for parameters corresponding to non-transport-dominated phenomena. Furthermore, one may argue that the initialization of the proposed approach with a few full-model states can be seen as an offline phase. The number of full-model states that are used for initialization is then an offline parameter that has to be chosen adequately. In all our experiments, the number of initial full-model states is set to a fixed value, which demonstrates that reduced models based on \abas{} seem to be robust with respect to the initial window size. A more detailed discuss is given in Section~\ref{sec:NumExp:Burgers:Performance} and Figure~\ref{fig:Burgers:InitWinSize}.

The adaptive sampling strategy as described in Section~\ref{sec:ABAS:LocalCoherence:AdaptSampling} requires the residual $\bfR_k$ at all $\nh$ components. To avoid computing the residual $\bfR_k$ at all components at each adaptation iteration, we adapt the sampling points at every $\nz$-th iteration instead. Thus, only at every $\nz$-th iteration, the residual $\bfR_k$ is computed at all components to adapt the sampling points, whereas at all other iterations, the residual is computed only at the sampling points. Note that even if the sampling points are adapted at each iteration $k = 1, \dots, K$, and so the residual is computed at each component, it still leads to a lower runtime to compute the ADEIM update only at $\nms < \nh$ sampling points, because the costs of computing the ADEIM update scales linearly in the number of sampling points $\nms$ \cite{Peherstorfer15aDEIM}.

One parameter of the proposed approach is the window size $\win$. The analysis developed in Section~\ref{sec:AADEIM:ErrorAnalysisUpdate} gives guidance on how to select the window size. Lemma~\ref{lm:ADEIM:SingValReduction} and Proposition~\ref{prop:Convergence} require that the columns of $\bfF_k \in \mathbb{R}^{\nh \times \win}$ and $\bfP^T\bfF_k \in \mathbb{R}^{\nh \times \win}$ span $\nr$-dimensional spaces. Thus, the condition $\win \geq \nr$ is imposed on the window size $\win$. At the same time, the discussion in Section~\ref{sec:ABAS:LocalStructure} indicates that a smaller window size $\win$ leads to faster decay of the singular values and thus to a potentially lower error of the adaptive reduced model. Given these two insights, we set $\win$ at least $\nr$ and not much larger than $\nr$ in the following. In the numerical examples in Section~\ref{sec:NumExp:Burgers} and Section~\ref{sec:NumExp:Combustion}, the window size is set to $\win = \nr + 1$, which avoids numerical problems if two columns of $\bfF_k$ are numerically linearly dependent. We refer to \cite{Peherstorfer15aDEIM} for further discussions and more extensive numerical studies regarding the window size and its effect on ADEIM.

\subsubsection{Algorithm}

\begin{algorithm}[t]
\caption{Adaptive bases and adaptive sampling (\abas{})}\label{alg:ABAS}
\begin{algorithmic}[1]
\Procedure{\abas{}}{$\bfq_0, \bff, \bfmu, \nr, \winit, \win, \nms, \nz, \nab$}
\State Solve full model for $\winit$ time steps $\bfQ = \texttt{solveFOM}(\bfq_0, \bff, \bfmu)$\label{alg:ABAS:SolveFOM}
\State Set $k = \winit+1$\label{alg:ABAS:StartROMInit}
\State Compute $\nr$-dimensional POD basis $\bfU_k$ of $\bfQ$\label{alg:ABAS:PODBasisConstruction}
\State Compute QDEIM interpolation points $\bfp_k = \texttt{qdeim}(\bfU_k)$\Comment{see appendix}
\State Initialize $\bfF = \bfQ[:, k-\win+1:k-1]$ and $\tilde{\bfq}_{k - 1} = \bfU_k^T\bfQ[:, k - 1]$\label{alg:ABAS:EndROMInit}
\For{$k = \winit + 1, \dots, K$}\label{alg:ABAS:ROMLoop}
\State Solve $\tilde{\bfq}_{k - 1} = \tilde{\bff}(\tilde{\bfq}_k; \bfmu)$ with DEIM interpolant with $\bfU_k$ and $\bfp_k$\label{alg:ABAS:SolveROM}
\State Store $\bfQ[:, k] = \bfU_k\tilde{\bfq}_k$
\If{$\operatorname{mod}(k, \nz) == 0 || k == \winit + 1$}\label{alg:ABAS:IfSampling}
\State Compute $\bfF[:, k] = \bff(\bfQ[:, k]; \bfmu)$\label{alg:ABAS:AdaptSamplingPointsStart}
\State $\bfR_k = \bfF[:, k - \win + 1:k] - \bfU_k(\bfU_k[\bfp_k, :])^{-1}\bfF[\bfp_k, k - \win + 1:k]$
\State $[\sim, \bfs_k] = \texttt{sort}(\texttt{sum}(\bfR_k.\widehat{~~}2, 2), \text{'descend'})$
\State Set $\breve{\bfs}_k = \bfs_k[\nms+1:\text{end}]$ and $\bfs_k = \bfs_k[1:\nms]$\label{alg:ABAS:AdaptSamplingPointsEnd}
\Else
\State Set $\bfs_k = \bfs_{k-1}$ and $\breve{\bfs}_k = \breve{\bfs}_{k - 1}$
\State Compute $\bfF[\bfs_k, k] = \bff(\bfQ[\bfs_k, k]; \bfmu)$\label{alg:ABAS:EvalResidualAtSamplingPoints}
\State Approximate $\bfF[\breve{\bfs}_k, k] = \bfU_k[\breve{\bfs}_k, :](\bfU_k[\bfp_k, :])^{-1}\bfF[\bfp_k, k]$\label{alg:ABAS:ApproxResidualAtSamplingPoints}
\EndIf
\State Set current window $\bfF_k = \bfF[:, k - \win + 1:k]$
\State Set $[\bfU_{k + 1}, \bfp_{k + 1}] = \texttt{adeim}(\bfU_k, \bfp_k, \bfs_k, \bfF_k[\bfp_k, :], \bfF_k[\bfs_k, :], \nab)$\Comment{see appendix}\label{alg:ABAS:AdaptBasisPoints}
\EndFor\\
\Return Return trajectory $\bfQ$
\EndProcedure
\end{algorithmic}
\end{algorithm}
Algorithm~\ref{alg:ABAS} gives details on the \abas{} approach by summarizing time stepping of a reduced model that uses the proposed adaptive basis updates and adaptive sampling. Helper functions used in Algorithm~\ref{alg:ABAS} are given in Appendix~\ref{app:Code}. Inputs of the algorithm are the initial condition $\bfq_0 \in \mathbb{R}^{\nh}$, the full-model function $\bff$ that describes the underlying dynamical system, and the parameter $\bfmu \in \Dcal$. Parameters of the approach are the dimension $\nr$ of the reduced space, the time step $\winit$ until which the full model is solved to initialize the reduced model, the window size $\win$ for the adaptation, the number of sampling points $\nms$, the frequency $\nz$ of the sampling points adaptation, and the rank of the update $\nab$.

Line~\ref{alg:ABAS:SolveFOM} solves the full model until time step $\winit$ to compute the corresponding trajectory $\bfQ_{\winit} = [\bfq_1, \dots, \bfq_{\winit}] \in \mathbb{R}^{\nh \times \winit}$. Lines~\ref{alg:ABAS:StartROMInit}--\ref{alg:ABAS:EndROMInit} initialize the reduced model by constructing the basis matrix $\bfU_k$ from $\bfQ_{\winit}$ and the vector of interpolation points $\bfp_k$. The loop in line~\ref{alg:ABAS:ROMLoop} iterates over the time steps $k = \winit + 1, \dots, K$ at which the reduced model is solved instead of the full model. Line~\ref{alg:ABAS:SolveROM} finds the reduced state $\tilde{\bfq}_k$ that satisfies the reduced model with respect to $\tilde{\bff}_k$ corresponding to the current DEIM interpolant with basis matrix $\bfU_k$ and interpolation points $\bfp_k$. The \texttt{if} branch on line~\ref{alg:ABAS:IfSampling} decides if either the sampling points are adapted or the sampling points from the previous iteration are reused. If the sampling points are adapted, then the full-model function $\bff$ is evaluated at all components and $\bfS_k$ is derived via the adaptive sampling strategy. Note that $\bfR_k.\widehat{~~}2$ is $\textsc{Matlab}$ notation and means that the entries of $\bfR_k$ are squared. If the sampling points are not adapted, then the full-model $\bff$ is evaluated at $\bfs_{k - 1}$ (see line~\ref{alg:ABAS:EvalResidualAtSamplingPoints}) and all other components are approximated (see line~\ref{alg:ABAS:ApproxResidualAtSamplingPoints}). The basis and the interpolation points are adapted with ADEIM at line~\ref{alg:ABAS:AdaptBasisPoints}. After the time stepping, the trajectory $\bfQ \in \mathbb{R}^{\nh \times K}$ is returned. The first $\winit$ columns in $\bfQ$ are computed with the full model and the subsequent $K - \winit + 1$ columns are computed with the adaptive reduced model.

\subsubsection{Costs and online efficiency of approach}
The proposed approach as summarized in Algorithm~\ref{alg:ABAS} has online computational costs that scale with the dimension $\nh$ of the full model. Thus, the proposed approach is \emph{not} online efficient in the sense of traditional model reduction; however, the proposed approach still has the potential to achieve speedups compared to static reduced models and full models because operations with costs that scale with $\nh$ are used economically and the online adaptation allows operating in lower dimensional reduced spaces than with static reduced models: First, the full model is solved for $\winit$ time steps at line~\ref{alg:ABAS:SolveFOM} in Algorithm~\ref{alg:ABAS}. The number of initial time steps $\winit$ is typically chosen orders of magnitude smaller than $K$, and thus only few time steps are computed with the full model. Furthermore, since $\winit \ll K$, constructing the POD basis at line~\ref{alg:ABAS:PODBasisConstruction} is typically cheap as well, even though the costs scale with $\nh$. Second, adapting the sampling points at lines~\ref{alg:ABAS:AdaptSamplingPointsStart}--\ref{alg:ABAS:AdaptSamplingPointsEnd} requires evaluating the full-model function $\bff$ at all $\nh$ components, which incurs costs that scale with $\nh$. The sampling points are adapted in $K/\nz$ iterations, which means that $K/\nz$ evaluations of the full-model function $\bff$ at all $\nh$ components are necessary. This shows that the speedup of the proposed approach is limited by the number of times the sampling points are adapted; however, evaluating $\bff$ is significantly cheaper than performing a time step with the full model, because the latter requires solving a nonlinear system of $\nh$ equations, whereas the former typically requires a function evaluation only. Third, the online adaptation typically enables operating in lower dimensional reduced spaces than with static reduced models if the full model exhibits transport-dominated phenomena. Thus, the computational costs of the nonlinear solve for computing the reduced state at the next time step (line~\ref{alg:ABAS:SolveROM} in Algorithm~\ref{alg:ABAS}) is kept low, which has the potential to lead to speedups compared to static reduced models and full models as demonstrated by the numerical results in Section~\ref{sec:NumExp}.

\section{Numerical results}
\label{sec:NumExp}
This section demonstrates model reduction based on \abas{} on three numerical examples. First, the toy example based on the advection equation introduced in Section~\ref{sec:Prelim:ProblemFormulation} is revisited in Section~\ref{sec:NumExp:AdVec}. Second, \abas{} is demonstrated on the Burgers' equation with a setup that leads to two interacting waves and time-varying viscosity and time-varying transport-direction coefficients. Third, we consider a model of a rocket combustor and demonstrate that the \abas{} approach achieves significant speedups in contrast to static reduced models that take even longer to run than the full model. All runtime results are computed with a \textsc{Matlab} 2017b implementation.

\begin{figure}
\centering{\huge\resizebox{0.8\columnwidth}{!}{\input{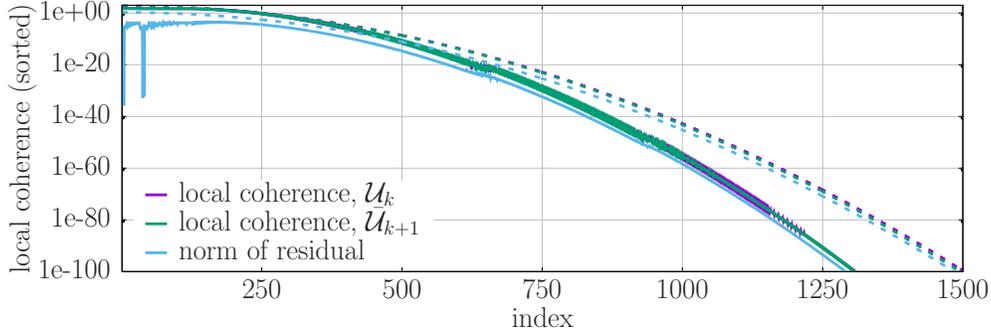}}}\\
\caption{Advection equation: The plot demonstrates that the norm of the rows of the residual $\bfR_k$ is bounded by the decay of the local coherence of the space $\Ucal_{k}$ and $\bar{\Ucal}_{k + 1}$ as proved in Lemma~\ref{prop:LinkCohResidual}. The solid curves are the norm of the residual and the local coherence, respectively, and the corresponding dashed curves are the bounds.}
\label{fig:AdVec:Coherence}
\end{figure}

\subsection{Advection equation}
\label{sec:NumExp:AdVec}
Consider the same setup as in Section~\ref{sec:Prelim:ProblemFormulation}. Set $k = 25$ and $\win = 25$ and let $\Ucal_k$ be the $\nr = 3$ dimensional DEIM space derived from $[\bfq_{k - \win + 1}, \dots, \bfq_{k}] \in \mathbb{R}^{\nh \times \win}$ and let $\bfU_k$ and $\bfP_k$ be the corresponding basis matrix and interpolation points matrix, respectively. Let now $\bar{\Ucal}_{k + 1}$ be the $\nr = 3$ dimensional space derived from $[\bfq_{76}, \dots, \bfq_{100}]$. Figure~\ref{fig:AdVec:Coherence} shows the local coherence $\gamma_i(\Ucal_k)$ and $\gamma_i(\bar{\Ucal}_{k + 1})$ for $i = 1, \dots, 1500$; the local coherence is sorted. The corresponding dashed curve are the bounds as in Lemma~\ref{prop:LinkCohResidual}. Let now $\tilde{\bfF}_k \in \mathbb{R}^{\nr \times \nr}$ have i.i.d.~standard Gaussian entries and consider $\bfF_k = \bar{\bfU}_{k + 1}\tilde{\bfF}_k$ and the corresponding residual $\bfR_k = \bfF_k - \bfU_k(\bfP_k^T\bfU_k)^{-1}\bfP_k^T\bfF_k$. The row-wise squared norm $\|\bfR^T_k\bfe_i\|_2^2$ of the residual $\bfR_k$ is plotted in Figure~\ref{fig:AdVec:Coherence}, together with the bound given by Lemma~\ref{prop:LinkCohResidual}. The results provide evidence that the decay of the residual is inherited from the decay of the local coherence as shown in Lemma~\ref{prop:LinkCohResidual}.

\begin{figure}
\centering{\huge\resizebox{0.8\columnwidth}{!}{\input{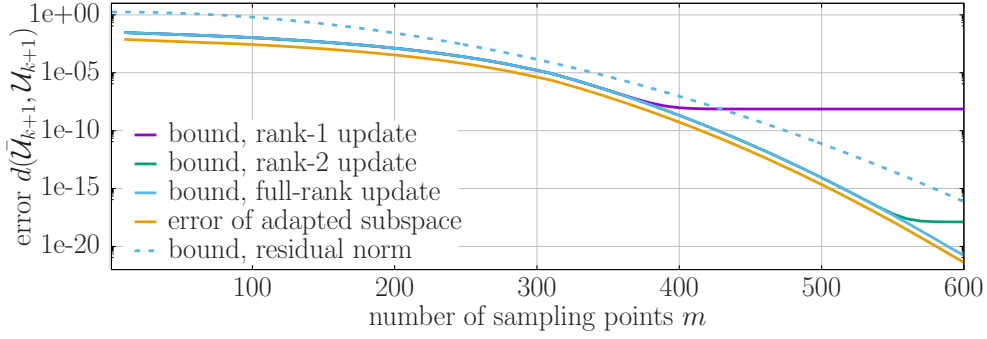}}}
\caption{Advection equation: The plot demonstrates that the fast decay of the error of the adapted space $d(\bar{\Ucal}_{k + 1}, \Ucal_{k + 1})$ is inherited from the fast decay of the residual. Furthermore, the plot shows that the bounds of the error of the adapted space stop decaying if a low-rank update is applied instead of a full-rank update.}
\label{fig:AdVec:ProjectionError}
\end{figure}

Consider now Figure~\ref{fig:AdVec:ProjectionError} that shows the error $d(\bar{\Ucal}_{k + 1}, \Ucal_{k + 1})$ of the adapted space $\Ucal_{k + 1}$, for ADEIM updates with rank $\nab = 1, 2, 3$, against the number of sampling points $\nms$. First, observe that if the rank is $\nab < 3$, then the bound of $d(\bar{\Ucal}_{k + 1}, \Ucal_{k + 1})$ levels off because the singular values $\sigma_{\nab + 1}, \dots, \sigma_{\bar{\nab}}$ dominate the decay factor $\rho_k$ \eqref{eq:ADEIMProp:FactorRho}. Second, the error $d(\bar{\Ucal}_{k + 1}, \Ucal_{k+1})$ is bounded by $\rho^2_k/\sigma_{\text{min}}^2(\bfF_k)$ as proved in Proposition~\ref{prop:Convergence}. Finally, the results show that the fast decay of the error of the adapted space with respect to the number of sampling points $\nms$ is inherited from the fast decay of the residual (shown as the dashed curve in Figure~\ref{fig:AdVec:ProjectionError}), which demonstrates Proposition~\ref{cor:NumberM}.

\subsection{Burgers' equation with time-varying viscosity}
\label{sec:NumExp:Burgers}
We now apply \abas{} to the Burgers' equation with time-varying viscosity and a transport direction that changes with time.

\begin{figure}
\begin{tabular}{c}
{\LARGE\resizebox{0.95\columnwidth}{!}{\input{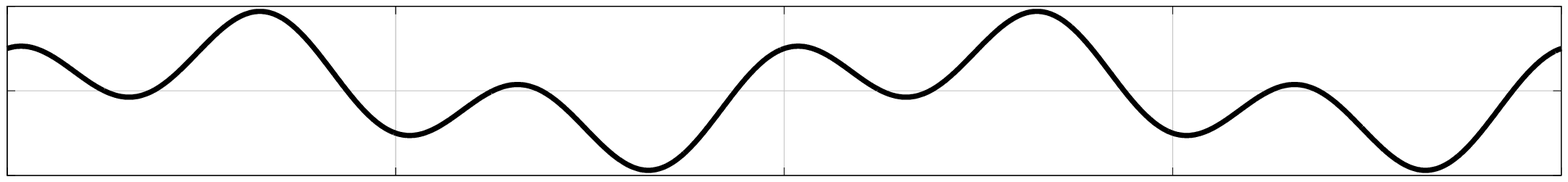}}}\\
{\scriptsize (a) viscosity coefficient}\\
{\LARGE\resizebox{0.95\columnwidth}{!}{\input{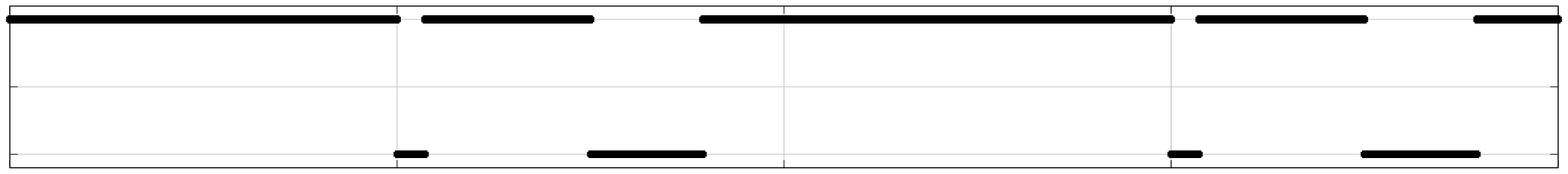}}}\\
{\scriptsize (b) transport direction}
\end{tabular}
\caption{Burgers' example: Plot (a) and (b) show the variation of the viscosity and the transport direction, respectively,  over time. Note that the viscosity \eqref{eq:Burgers:Nu} and transport coefficients \eqref{eq:Burgers:Eta} are periodic and therefore shown only up to time $t = 0.2$. The nominal viscosity is $\mu = 3 \times 10^{-3}$.}
\label{fig:Burgers:NuSign}
\end{figure}

\subsubsection{Problem setup}
\label{sec:NumExp:Burgers:ProblemSetup}
Let $\Omega = (-1, 1) \subset \mathbb{R}$ be the spatial domain, set $\bar{\Omega} = [-1, 1]$, and let $T = 1.5$ be end time. Consider now the Burgers' equation
\begin{equation}
\partial_t q(x, t) + \eta(t)q(x, t)\partial_x q(x, t) = \nu(t) \partial_{x}^2 q(x, t)\,,\qquad x \in \Omega\,,
\label{eq:NumExp:BurgersEq}
\end{equation}
with time $t \in [0, T]$, the solution function $q: \bar{\Omega} \times [0, T] \to \mathbb{R}$, the time-varying viscosity $\nu: [0, T] \to \mathbb{R}$, and the transport direction $\eta: [0, T] \to \{-1, 1\}$. The viscosity is
\begin{equation}
\nu(t) = \mu\left(\sin\left(20 \pi t\right) + \cos\left(60 \pi t\right) + 2\right)\,,
\label{eq:Burgers:Nu}
\end{equation}
where $\mu \in \Dcal \subset \mathbb{R}$ is the nominal viscosity parameter. The transport direction is
\begin{equation}
\eta(t) = \operatorname{sign}\left(\sin\left(20 \pi t\right) + \cos\left(60 \pi t\right) + 1\right)\,.
\label{eq:Burgers:Eta}
\end{equation}
The viscosity and transport direction change over time, as shown in Figure~\ref{fig:Burgers:NuSign} for $\mu = 3 \times 10^{-3}$. The PDE \eqref{eq:NumExp:BurgersEq} is closed with Dirichlet boundary conditions and the initial condition given by
\begin{equation}
q(x, 0) = \begin{cases}
2\,, & x = -2\\
1\,, & -\frac{1}{2} \leq x \leq -\frac{1}{3}\\
0\,, & \text{else}
\end{cases}\,.
\label{eq:NumExp:Burgers:IC}
\end{equation}

\begin{figure}
\begin{tabular}{cc}
{\Huge\resizebox{0.45\columnwidth}{!}{\input{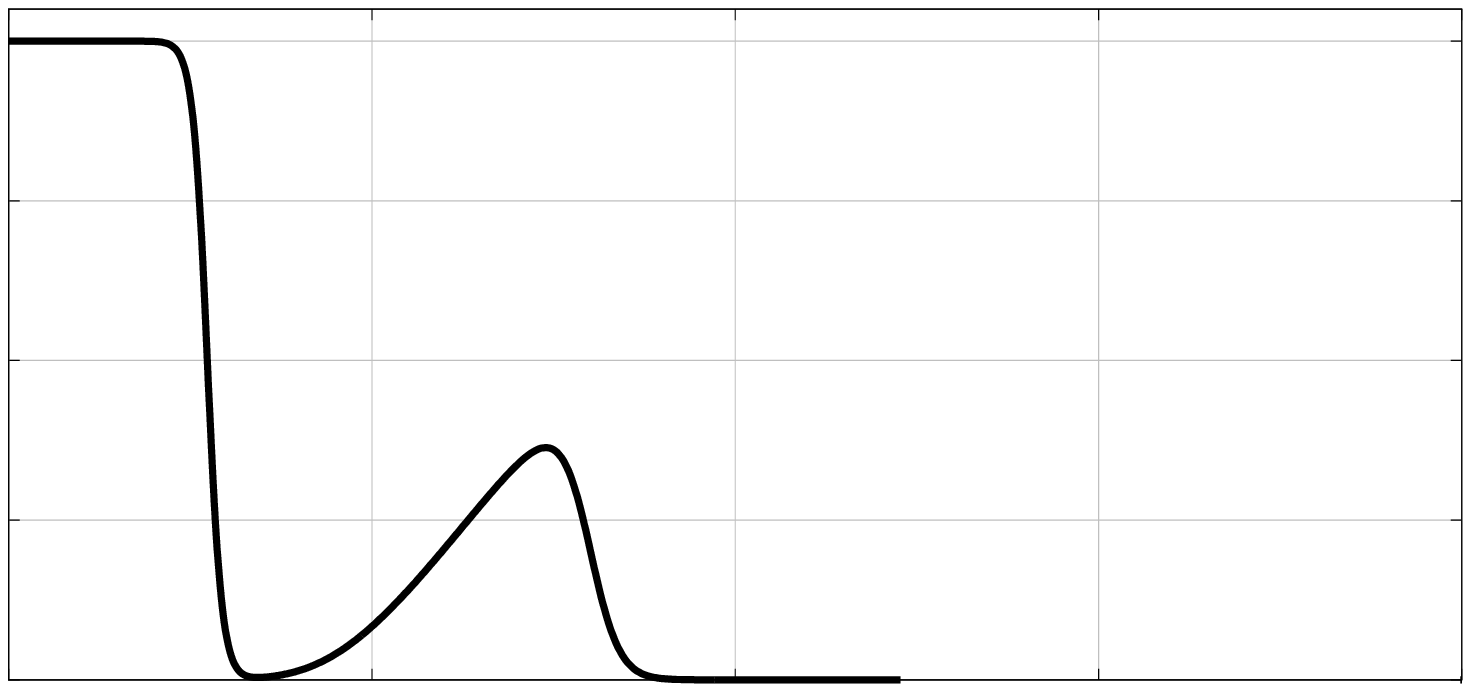}}} & {\Huge\resizebox{0.45\columnwidth}{!}{\input{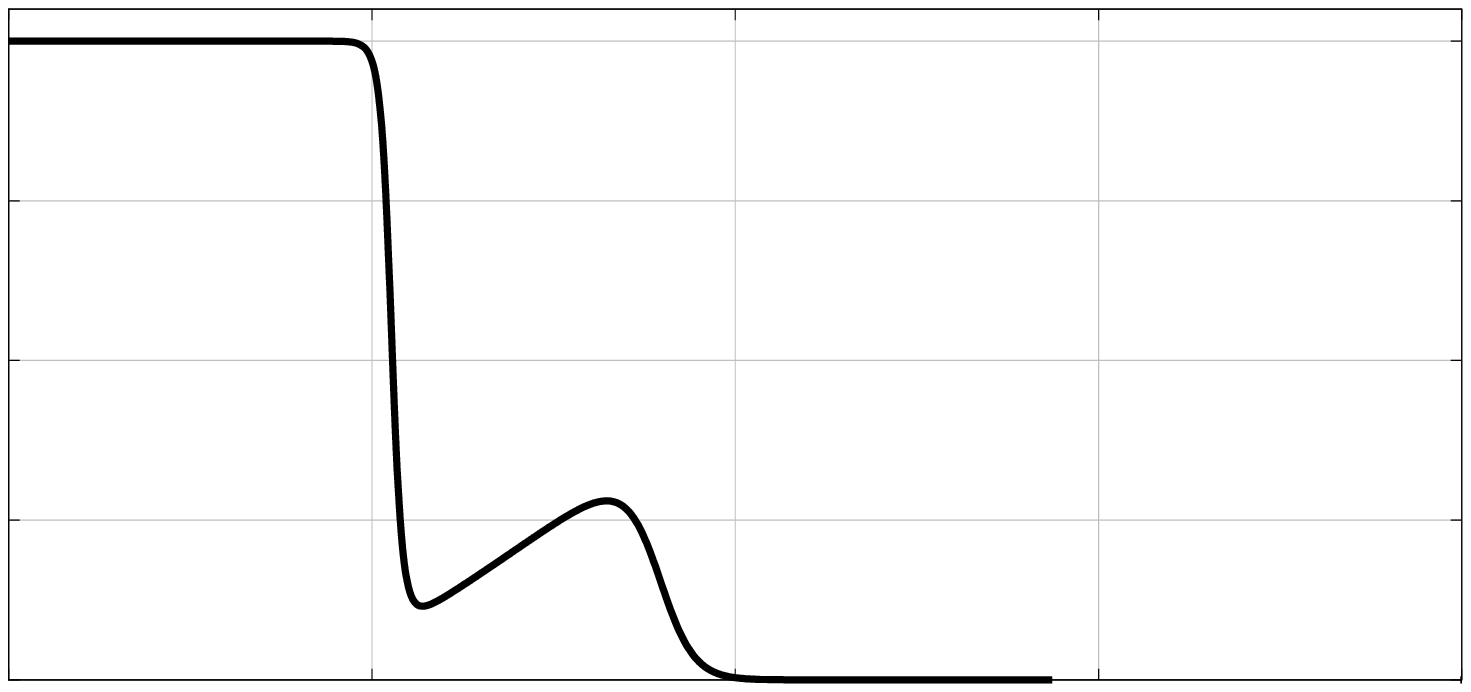}}}\\
{\scriptsize (a) time 0.4125} & {\scriptsize (b) time 0.775}\\
{\Huge\resizebox{0.45\columnwidth}{!}{\input{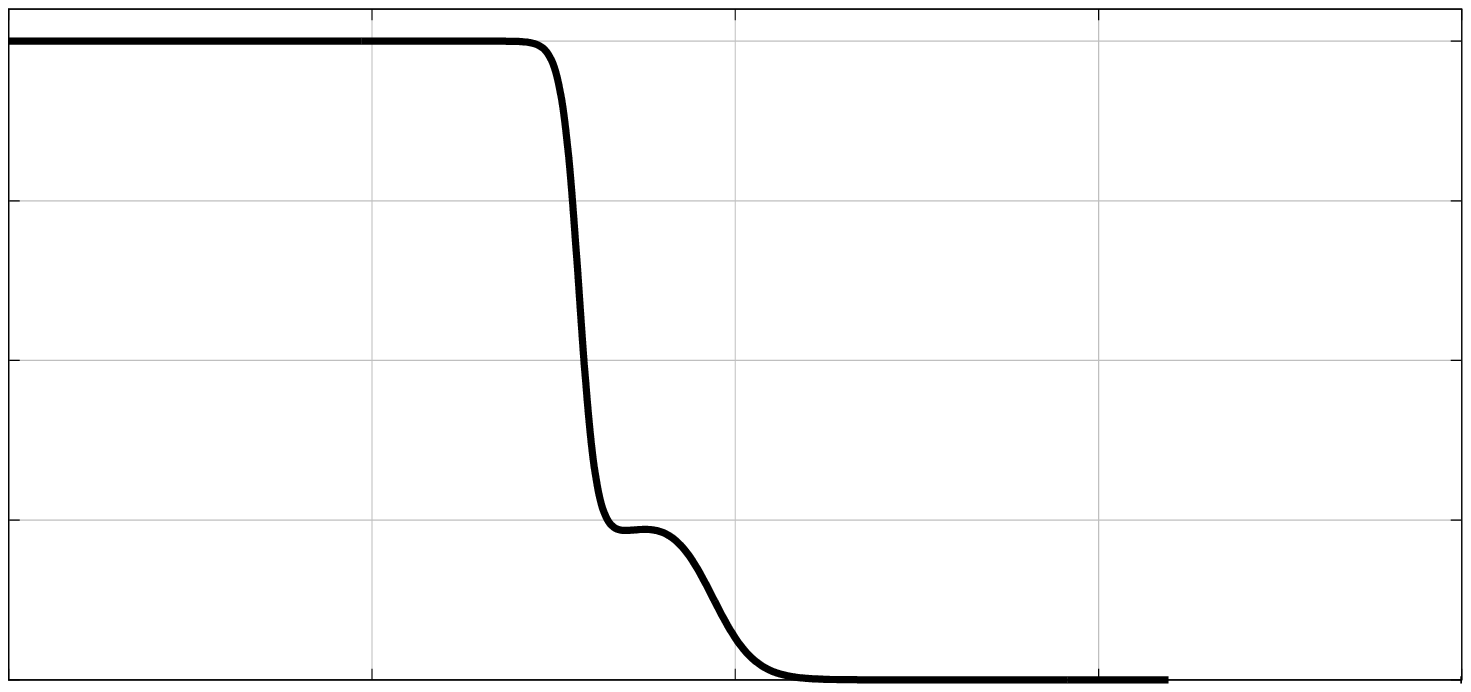}}} & {\Huge\resizebox{0.45\columnwidth}{!}{\input{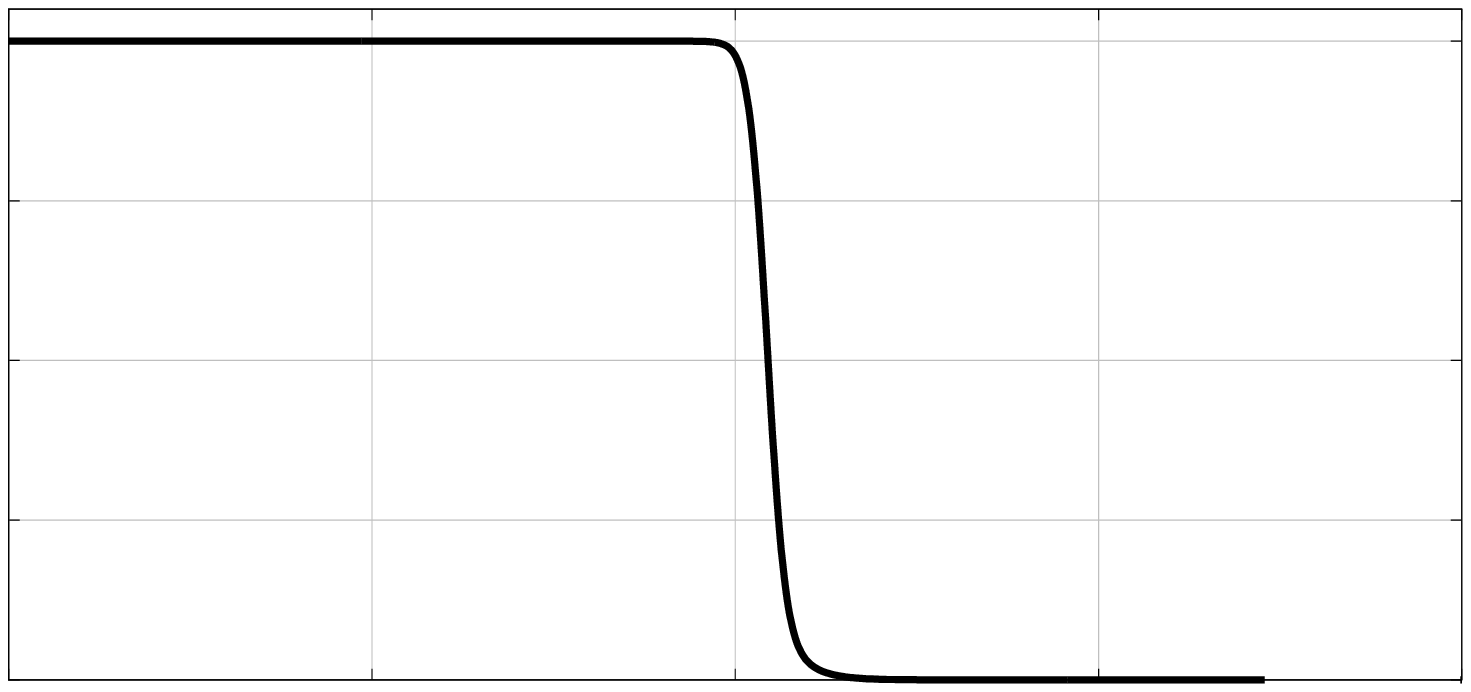}}}\\
{\scriptsize (c) time 1.1375} & {\scriptsize (d) time 1.5}
\end{tabular}
\caption{Burgers' example: The plots show the solution of the full model at different times in the time interval $[0, 1.5]$. The initial condition \eqref{eq:NumExp:Burgers:IC} leads to two waves propagating to the right, with the waves starting to interact around time $t = 0.7$. Note that the transport direction and the viscosity change over time, which cannot be seen in these plots, cf.~Figure~\ref{fig:Burgers:NuSign}.}
\label{fig:Burgers:SolutionsFOM}
\end{figure}

The full model is obtained by discretizing the spatial domain $\Omega$ of the PDE \eqref{eq:NumExp:BurgersEq} on an equidistant grid with $\nh = 1024$ inner grid points. A second-order finite-difference scheme is used. The time domain is discretized with the implicit Euler method and time step size $\delta t = 5 \times 10^{-5}$. In each time step, Newton's method is used to solve the corresponding system of nonlinear equations of the form \eqref{eq:Prelim:FOMProblem}. At each time step, 15 Newton iterations are performed with a step size of 1. Jacobians are derived analytically and returned in assembled form to the Newton routine. The linear solves in each Newton step are performed with \textsc{Matlab}'s backslash operator. No parallelization is used besides \textsc{Matlab}'s built-in parallelization for the backslash operator. The same linear-algebra routines, Newton scheme, and parallelization are used for time-stepping in all models. The solution of the full model for $\mu = 3 \times 10^{-3}$ is shown in Figure~\ref{fig:Burgers:SolutionsFOM}.

Figure~\ref{fig:Burgers:Residual}a reports the decay of the singular values corresponding to local trajectories of length $\win = 50$ at time $t = 0.77, 1.12, 1.50$, respectively, and parameter $\mu = 3 \times 10^{-3}$. The results show an orders of magnitude faster decay of the singular values of the local trajectories than of the global trajectory, which indicates that this problem exhibits a local low-rank structure. Figure~\ref{fig:Burgers:Residual}b shows the decay of the squared residual of approximating the solution at $t = 0.77, 1.12, 1.50$ at the corresponding local spaces of dimension $\nr = 8$ of the trajectories for which the singular values are plotted in Figure~\ref{fig:Burgers:Residual}a. A fast decay of the norm of the rows of the residual is observed.

\begin{figure}
\begin{tabular}{cc}
{\huge\resizebox{0.45\columnwidth}{!}{\input{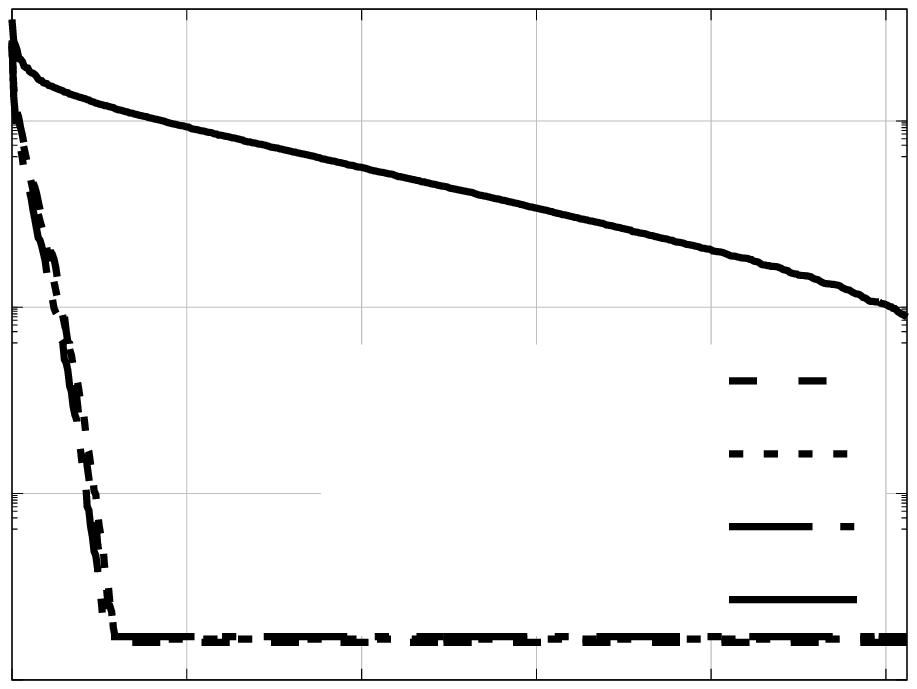}}} & {\huge\resizebox{0.45\columnwidth}{!}{\input{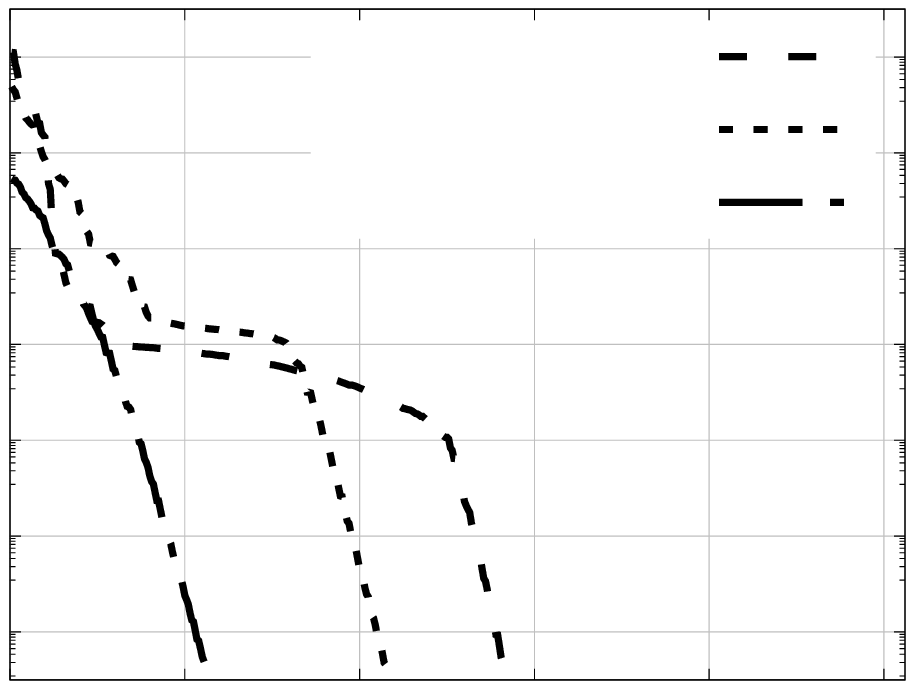}}}\\
{\scriptsize (a) singular values} & {\scriptsize (b) residual}
\end{tabular}
\caption{Burgers' example: The plots indicate that the singular values of local trajectories (see plot (a)) and the corresponding DEIM residual of the local spaces (see plot(b)) decay fast, cf.~Section~\ref{sec:ABAS:LocalStructure}.}
\label{fig:Burgers:Residual}
\end{figure}

\subsubsection{Performance of \abas{}}
\label{sec:NumExp:Burgers:Performance}
We now compare the runtime and accuracy of static reduced models, \abas{} models, and the full model. The static reduced models are derived from the trajectories corresponding to the parameters $\mu \in \{5 \times 10^{-3}, 10^{-3}, 5 \times 10^{-4}\}$, following the procedure outline in Section~\ref{sec:Prelim:MOR}. The dimension of the DEIM space is $\nr$ and the interpolation points are selected with QDEIM \cite{QDEIM}. The \abas{} models are initialized with the first $\winit = 100$ states computed with the full model. The dimension of the DEIM spaces of the \abas{} models is $\nr = 8$. The interpolation points are selected with QDEIM \cite{QDEIM}. The DEIM space $\bfU_k$ and the DEIM interpolation points $\bfP_k$ are adapted every other time step. The interpolation points $\bfP_k$ are adapted by applying QDEIM to the adapted basis. Only updates with rank $\nab = 1$ are applied. The window size is $\win = \nr + 1 = 9$. The sampling points are obtained with the adaptive sampling strategy described in Section~\ref{sec:ABAS:LocalCoherence}. The sampling points are adapted every $\nz = 5$ iteration, except otherwise noted. The error of the static and the \abas{} models are measured in the Frobenius norm
\begin{equation}
\operatorname{err}(\tilde{\bfQ}(\mu)) = \frac{\|\tilde{\bfQ}(\mu) - \bfQ(\mu)\|_F}{\|\bfQ(\mu)\|_F}\,,
\label{eq:NumExp:Burgers:Error}
\end{equation}
where $\bfQ(\mu)$ is the trajectory obtained with the full model at parameter $\mu$ and $\tilde{\bfQ}(\mu)$ is the trajectory obtained either with a static reduced model or an \abas{} model. For computing $\operatorname{err}(\tilde{\bfQ}(\mu))$, only states at time steps $k = 1, \dots, 1000$ and then for $k > 1000$ at every 50-th time step are taken into account.

\begin{figure}
\begin{tabular}{cc}
{\LARGE\resizebox{0.45\columnwidth}{!}{\input{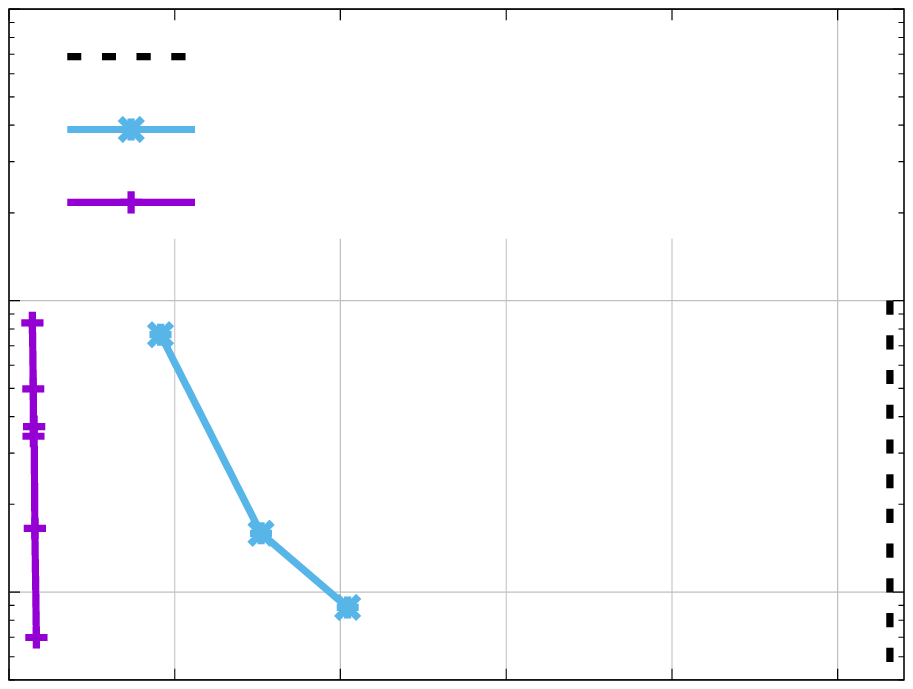}}} & {\LARGE\resizebox{0.45\columnwidth}{!}{\input{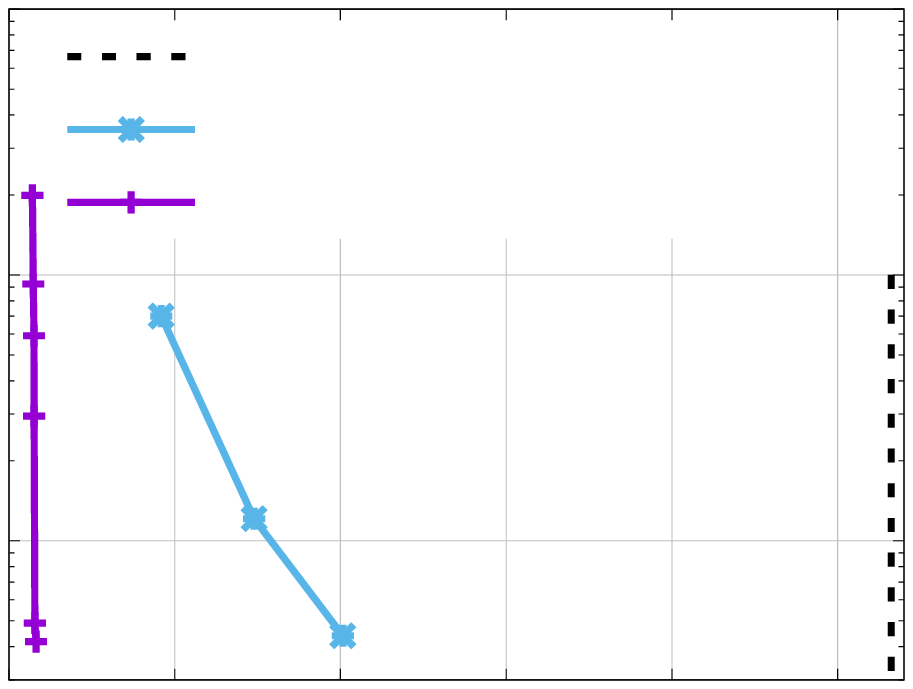}}}\\
{\scriptsize (a) parameter $\mu = 3 \times 10^{-3}$} & {\scriptsize (b) parameter $\mu = 8 \times 10^{-4}$}
\end{tabular}
\caption{Burgers' example: The plots show that \abas{} achieves a speedup of about one order of magnitude compared to the full model. About $\nms = 96$ sampling points out of $\nh = 1024$ are sufficient for the adaptive sampling strategy to obtain an \abas{} model with an error below $10^{-1}$. The results in these plots further show that increasing the number of samplings points only slightly increases the runtime, which indicates that computing the full-model function $\bff$ to update the DEIM space is computationally cheap in this example. The curves corresponding to the static reduced models show error/runtime for dimensions $\nr \in \{125, 175, 225\}$. The error/runtime of the \abas{} models is shown for $\nms \in \{96, 160, 224, 352, 480, 608\}$.}
\label{fig:Burgers:Speedup}
\end{figure}

Let us first consider parameter $\mu = 3 \times 10^{-3}$. Figure~\ref{fig:Burgers:Speedup}a shows the error $\operatorname{err}(\tilde{\bfQ}(\mu))$ of the static reduced models with $\nr \in \{125, 175, 225\}$ and of the \abas{} models with $\nms \in \{96, 160, 224, 352, 480, 608\}$ and dimension $\nr = 8$. The runtime results are obtained on compute nodes with Intel Xeon E5-1660v4 with 64GB RAM. The sampling points in the \abas{} models are adapted very $\nz$-th time step with $\nz = 5$. The dashed line in Figure~\ref{fig:Burgers:Speedup}a marks the runtime of the full model. The plot shows that the static reduced model achieves a speedup compared to the full model in this example; however, more than 100 dimensions are required for the DEIM space to achieve an error below $10^{-1}$. In contrast, the \abas{} model achieves errors below $10^{-1}$ with $\nr = 8$ dimensions and $\nms = 96$ sampling points, which leads to about an order of magnitude speedup compared to the full model. The plot in Figure~\ref{fig:Burgers:Speedup}a further indicates that increasing the number of sampling points from $\nms = 96$ to $\nms = 608$ reduces the error from $10^{-1}$ to $10^{-2}$ without a significant increase of the runtime, which provides evidence that evaluating the full-model function $\bff$ is significantly cheaper than solving the full model in this example. Figure~\ref{fig:Burgers:Speedup}b shows similar behavior of \abas{} for parameter $\mu = 8 \times 10^{-4}$. Figure~\ref{fig:Burgers:SpeedupHistogram}a summarizes the runtime of the static reduced model with dimension $\nr = 225$ and the \abas{} model with $\nms = 608$ sampling points, which are required to achieve an error \eqref{eq:NumExp:Burgers:Error} below $10^{-2}$.

\begin{figure}
\begin{tabular}{cc}
{\LARGE\resizebox{0.45\columnwidth}{!}{\input{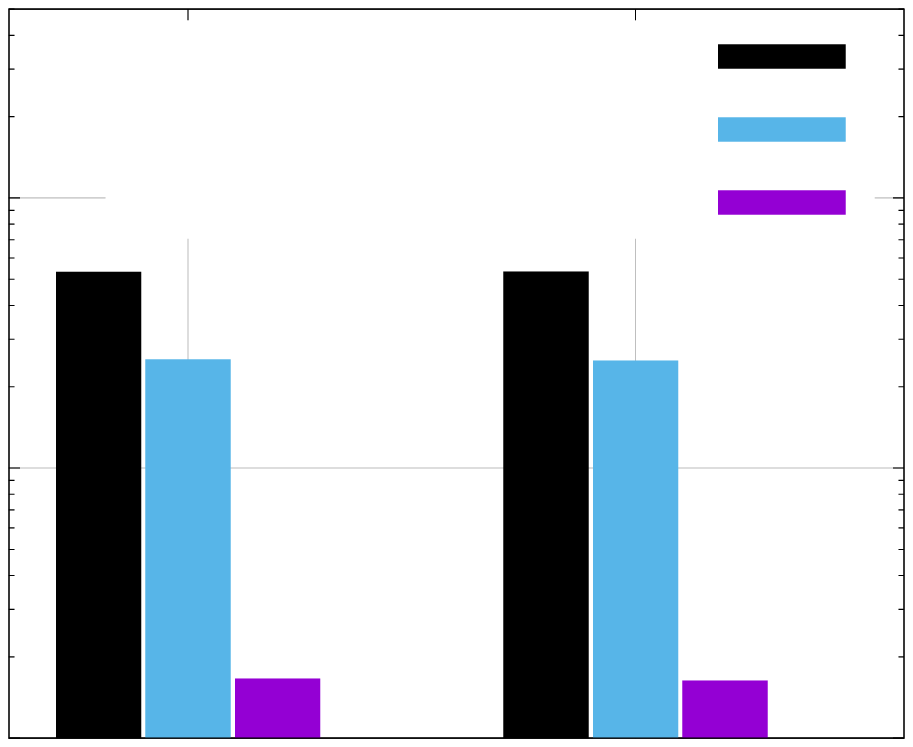}}} & {\LARGE\resizebox{0.45\columnwidth}{!}{\input{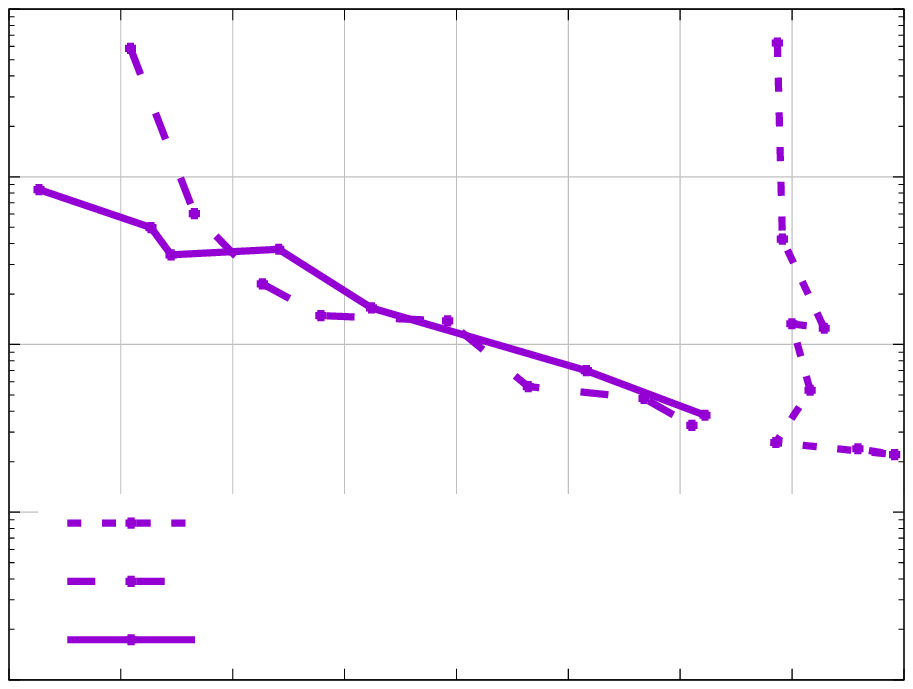}}}\\
{\scriptsize (a) runtime to achieve error \eqref{eq:NumExp:Burgers:Error} below $10^{-2}$} & {\scriptsize (b) sampling points adaptation}
\end{tabular}
\caption{Burgers' example: The plot in (a) shows that the \abas{} model achieves a speedup of about one order of magnitude compared to the static and the full model. Plot (b) indicates that adapting the sampling points every $5$-th time step, i.e., $\nz = 5$, is sufficient in this example.}
\label{fig:Burgers:SpeedupHistogram}
\end{figure}

\begin{figure}
\begin{center}
{\LARGE\resizebox{0.45\columnwidth}{!}{\input{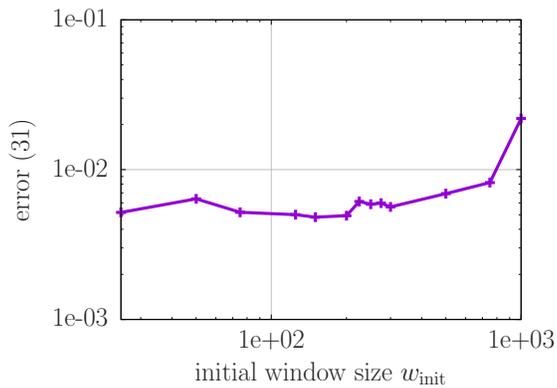}}}
\end{center}
\caption{Burgers' example: Plot shows that error \eqref{eq:NumExp:Burgers:Error} of the AADEIM model changes only slightly with the initial window size $\winit$ if it is near 100 in this example.}
\label{fig:Burgers:InitWinSize}
\end{figure}

We now investigate the effect of adapting the sampling points. The sampling points need to be adapted during the time stepping to ensure that components of the basis matrix corresponding to a high residual are updated. The frequency of adapting the sampling points is controlled by the parameter $z$ and has to be chosen depending on the problem at hand. Just as the time step size of the full model has to be chosen such that the coherent structure moves only so far in the spatial domain per time step that the full model remains stable, the frequency of adapting the sampling points $z$ has to be set such that sampling points can move along with the coherent structure to indicate high-residual components. Thus, if too few adaptations of the sampling points are performed, then high-residual components might be missed. At the same time, too many adaptations can become costly because each adaptation of the sampling points requires an evaluation of the full-model right-hand side function at all components. In Figure~\ref{fig:Burgers:SpeedupHistogram}b, we study the effect of the parameter $z$ on the error \eqref{eq:NumExp:Burgers:Error} of the \abas{} model for parameter $\mu = 3 \times 10^{-3}$. The sampling points are adapted at every time step ($z = 1$), at every $3$-rd time step ($z = 3$), and at every $5$-th time step ($z = 5$). The number of sampling points $\nms$ is varied $\nms = 96, 160, 224, 352, 480, 608$. First, note that,
for $\nz = 1$, i.e., the sampling points are adapted at each time step, the runtime is almost constant for an increasing number of sampling points.
Second, there is almost no difference in the error for adapting the sampling points at every $3$-rd iteration to adapting at every $5$-th iteration, which means that the sampling points are valid over several time steps in this example.

\begin{figure}
\begin{tabular}{cc}
{\LARGE\resizebox{0.45\columnwidth}{!}{\input{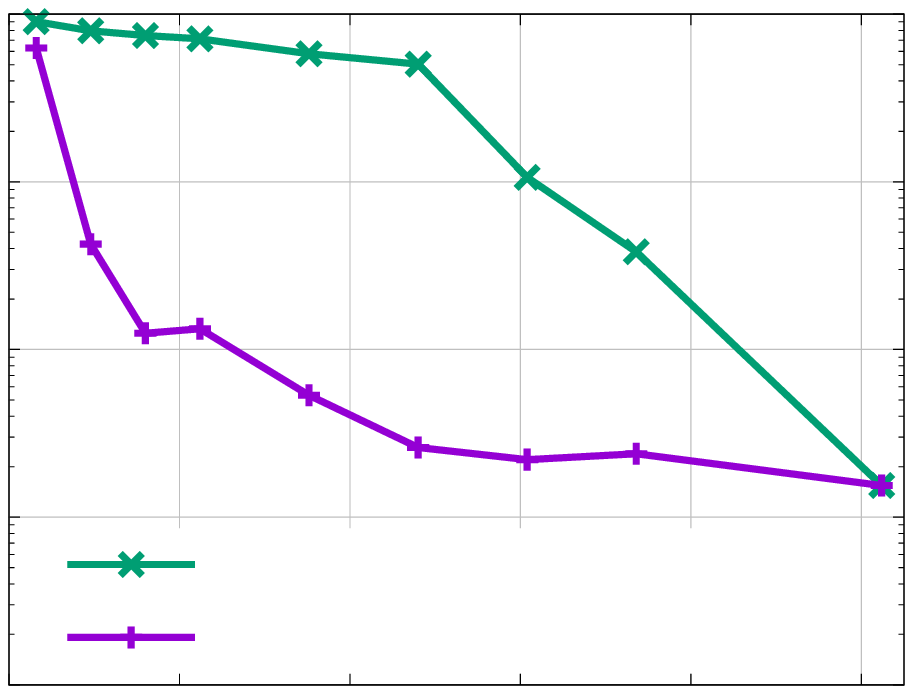}}} &{\LARGE\resizebox{0.45\columnwidth}{!}{\input{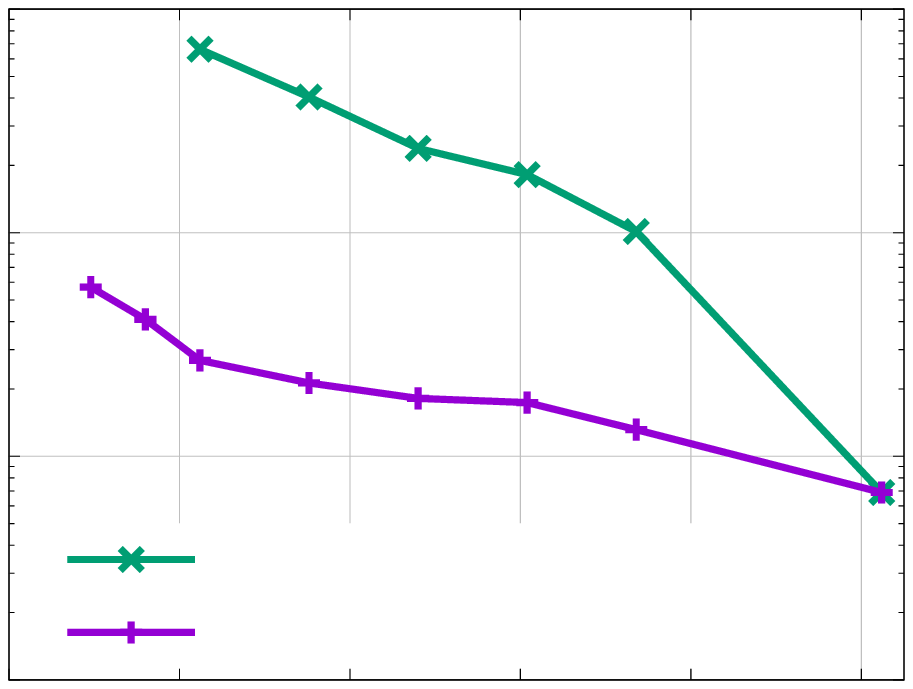}}} \\
{\scriptsize (a) parameter $\mu = 3 \times 10^{-3}$} & {\scriptsize (b) parameter $\mu = 8 \times 10^{-4}$}
\end{tabular}
\caption{Burgers' example: The proposed adaptive sampling strategy reduces the error \eqref{eq:NumExp:Burgers:Error} of the \abas{} model by up to two orders of magnitude compared to uniform sampling without replacement.}
\label{fig:Burgers:Uniform}
\end{figure}

Figure~\ref{fig:Burgers:InitWinSize} reports error \eqref{eq:NumExp:Burgers:Error} of the \abas{} model for different initial window sizes $\winit$. The setup of the \abas{} model is the same as for the runtime results reported in Figure~\ref{fig:Burgers:SpeedupHistogram}a with $\mu = 3 \times 10^{-3}$, i.e., dimension $\nr = 8$, number of sampling points $\nms = 608$, and sampling points are adapted every $5$-th time step. The initial window size $\winit$ is varied in $\{25, 50, 75, 125, 150, 200, 225, 250, 275, 300, 500, 750, 1000\}$. If the window size is set too large, then the local low-rank structure of the snapshots is lost and the correspondingly larger error in the first few time steps accumulates quickly, which can be observed in Figure~\ref{fig:Burgers:InitWinSize} by noting that the curve increases starting around $\winit = 750$. If the window size is chosen too small, then the initial basis might so poorly approximate the dynamics in the first few time steps that the basis adaptation cannot be initialized. Thus, the initial window size $\winit$ is a tuning parameter of the proposed approach; however, as the results in Figure~\ref{fig:Burgers:InitWinSize} show, the error \eqref{eq:NumExp:Burgers:Error} of the \abas{} model remains near $10^{-2}$ for a wide range of initial window sizes and therefore, in this example, the proposed approach is robust with respect to choices of the initial window sizes.

\subsubsection{Performance of adaptive sampling strategy}
We now compare adaptive sampling to random uniform sampling, which is used in, e.g., \cite{Peherstorfer15aDEIM}. Figure~\ref{fig:Burgers:Uniform} shows the error \eqref{eq:NumExp:Burgers:Error} of the \abas{} model if the sampling points are selected with the proposed adaptive sampling strategy and random uniform sampling of $\{1, \dots, \nh\}$. The random uniform sampling is without replacement. The curves in Figure~\ref{fig:Burgers:Uniform} correspond to $\nms \in \{32, 96, 160, 224, 352, 480, 608, 736, \nh\}$ sampling points. Uniform sampling with $\nms = 32$ leads to an unstable model in case of $\mu = 8 \times 10^{-4}$ and therefore its error is not plotted. The results in Figure~\ref{fig:Burgers:Uniform} indicate that adaptive sampling reduces the error \eqref{eq:NumExp:Burgers:Error} of the \abas{} model compared to uniform sampling. Improvements of up to two orders of magnitude can be observed. Note that the two sampling schemes coincide if all points $\nms = \nh$ are selected. Figure~\ref{fig:Burgers:ErrorDim} reports the error \eqref{eq:NumExp:Burgers:Error} of the \abas{} model for dimensions $\nr \in \{4, 6, 8, 10, 12\}$ with $\nms \in \{96, 224, 480\}$ sampling points (adaptive sampling strategy). The results indicate that a larger dimension of the DEIM space leads to a lower error only if sufficiently many sampling points are selected for the adaptation. For example, the \abas{} model with dimension $\nr = 12$ and $\nms = 480$ sampling points achieves an about one order of magnitude lower error than the \abas{} model with the same dimension and $\nms = 96$ sampling points.

\begin{figure}
\begin{tabular}{cc}
{\LARGE\resizebox{0.45\columnwidth}{!}{\input{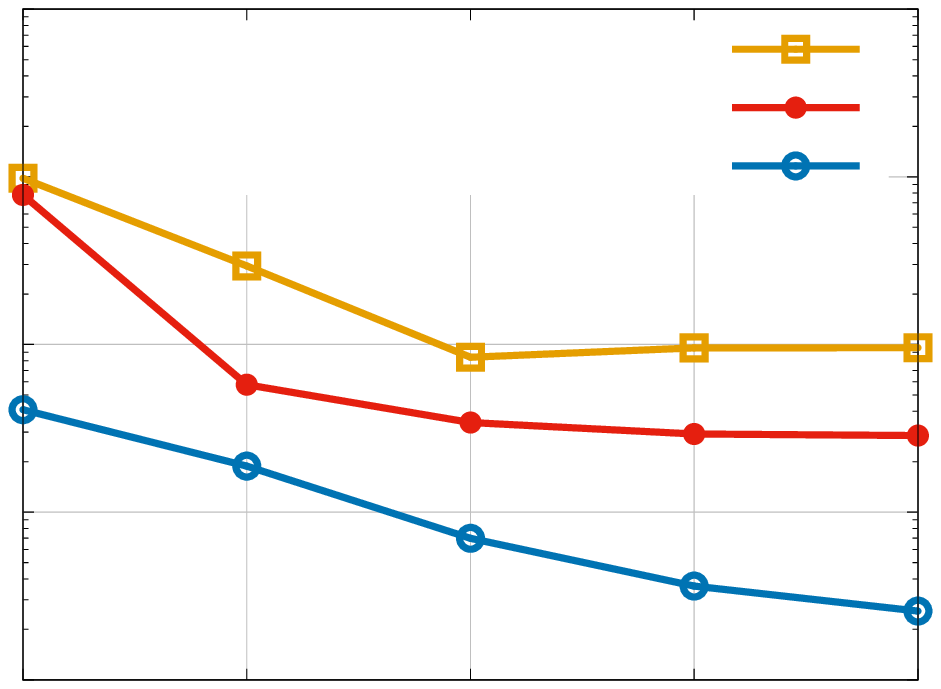}}} &  {\LARGE\resizebox{0.45\columnwidth}{!}{\input{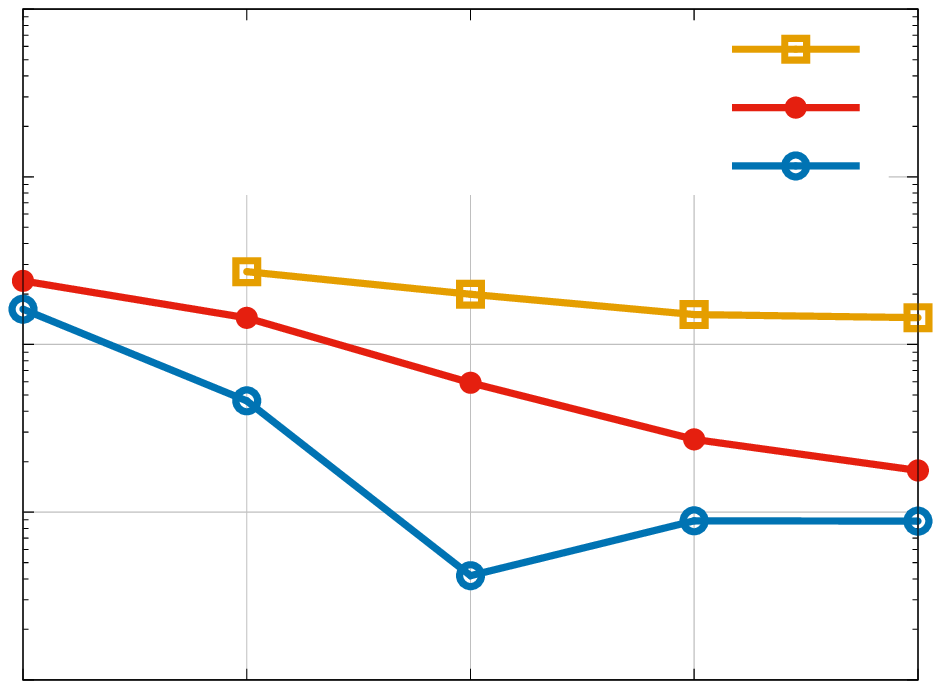}}}\\
{\scriptsize (a) parameter $\mu = 3 \times 10^{-3}$} & {\scriptsize (b) parameter $\mu = 8 \times 10^{-4}$}
\end{tabular}
\caption{Burgers' example: The results in these plots indicate that the dimension of the DEIM space and the number of sampling points $\nms$ need to be traded off with respect to each other. At least in this example, it seems that an increase in the dimension of the DEIM space needs to be accompanied by an increase in the number of sampling points to obtain an \abas{} model with a lower error.}
\label{fig:Burgers:ErrorDim}
\end{figure}

\subsubsection{Performance of \abas{} with respect to full-model dimension}
\label{sec:NumExp:Burgers:N}
We consider the same setup as described in Section~\ref{sec:NumExp:Burgers:ProblemSetup}, except that we now base the comparison on four full models corresponding to $\nh \in \{256, 512, 1024, 2048\}$. For all four full models, the time-step size is set to $10^{-5}$ to avoid instabilities of the full model with dimension $\nh = 2048$, in contrast to Section~\ref{sec:NumExp:Burgers:Performance} where the time-step size is $5 \times 10^{-5}$. We derive static reduced models and \abas{} models for each of the four full models with the same procedure as described in Section~\ref{sec:NumExp:Burgers:ProblemSetup} and Section~\ref{sec:NumExp:Burgers:Performance}, respectively. The dimension of the static reduced models is 125 for $\nh = 256$, 225 for $\nh = 512$, 175 for $\nh = 1024$, and 200 for $\nh = 1024$ so that an error \eqref{eq:NumExp:Burgers:Error} of less than $10^{-2}$ is achieved for parameters $\mu = 3 \times 10^{-3}$ and $\mu = 8 \times 10^{-4}$. The dimension of the \abas{} models is 8 and the numbers of sampling points are 32, 96, 224, and 736 for full-model dimension 256, 512, 1024, and 2048, respectively, which leads to error \eqref{eq:NumExp:Burgers:Error} of less than $10^{-2}$ for parameter $\mu = 3 \times 10^{-3}$ and parameter $\mu= 8 \times 10^{-4}$. As in Section~\ref{sec:NumExp:Burgers:Performance}, the adaptive sampling strategy is used and the sampling points are adapted every $5$-th time step. Runtimes were measured on compute nodes with Intel Xeon E5-2690v4 2.6GHz and 32GB RAM. The results in Figure~\ref{fig:Burgers:ErrorDimN}a show that speedups of about one order of magnitude are achieved by the \abas{} models in all four cases of $\nh \in \{256, 512, 1024, 2048\}$ for parameter $\mu = 3 \times 10^{-3}$. Figure~\ref{fig:Burgers:ErrorDimN}b reports the error \eqref{eq:NumExp:Burgers:Error} versus the ratio $\nms/\nh$. The reported results demonstrate that the error \eqref{eq:NumExp:Burgers:Error} is similar for all $\nh \in \{256, 512, 1024, 2048\}$ with respect to the ratio $\nms/\nh$. Note that the number of grid points within a neighborhood of the coherent structure grows with the full-model dimension $\nh$, and thus the number of sampling points $\nms$ has to grow with $\nh$ as well. However, the results in Figure~\ref{fig:Burgers:ErrorDimN}b further indicate that the number of sampling points $\nms$ has to grow linearly with $\nh$ only. Furthermore, note that even though the number of sampling points $\nms$ grows with $\nh$, for all $\nh \in \{256, 512, 1024, 2048\}$, about the same speedup is achieved by \abas{} as shown in Figure~\ref{fig:Burgers:ErrorDimN}a. Similar conclusions are drawn from the results corresponding to parameter $\mu = 8 \times 10^{-4}$ in Figure~\ref{fig:Burgers:ErrorDimN}c and Figure~\ref{fig:Burgers:ErrorDimN}d.

\begin{figure}
\begin{tabular}{cc}
{\LARGE\resizebox{0.45\columnwidth}{!}{\input{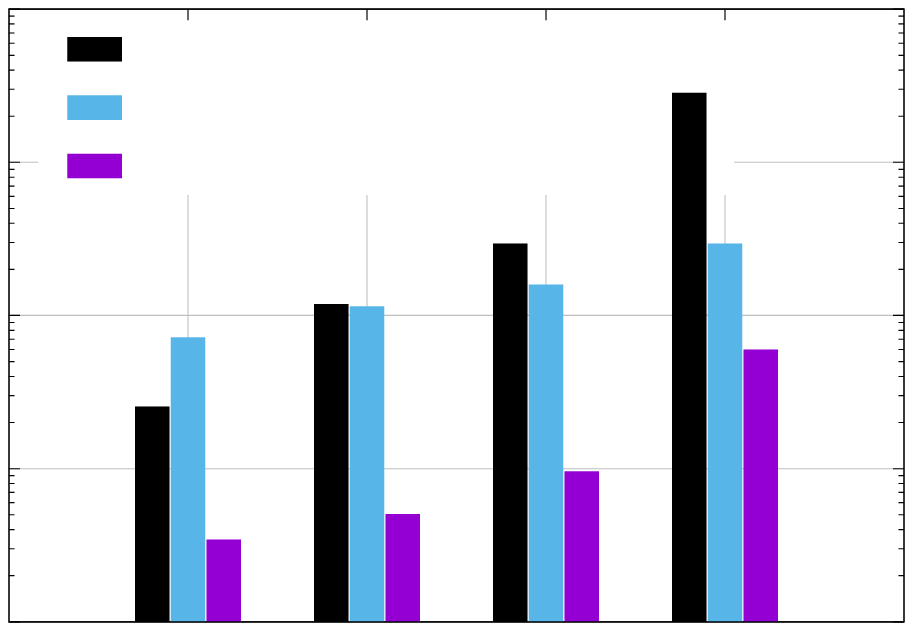}}} &  {\LARGE\resizebox{0.45\columnwidth}{!}{\input{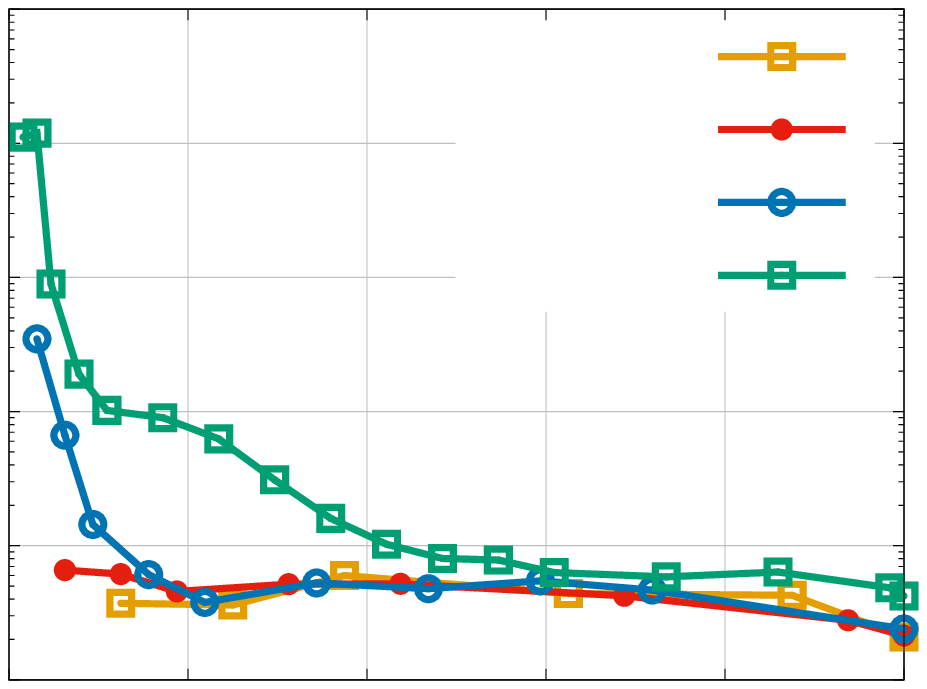}}}\\
{\scriptsize (a) $\mu = 3 \times 10^{-3}$, runtime speedup} & {\scriptsize (b) $\mu = 3 \times 10^{-3}$, number of sampling points}\\
{\LARGE\resizebox{0.45\columnwidth}{!}{\input{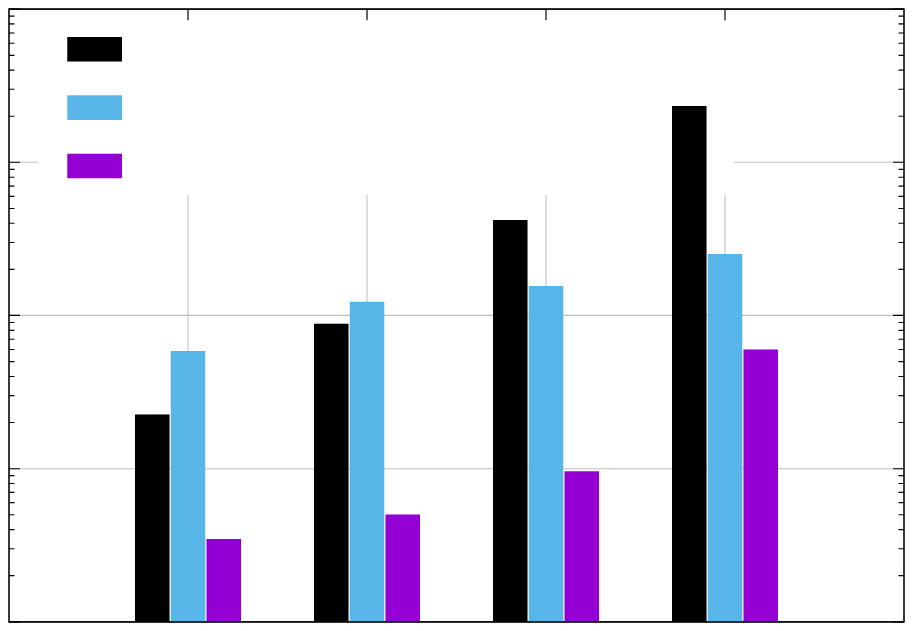}}} &  {\LARGE\resizebox{0.45\columnwidth}{!}{\input{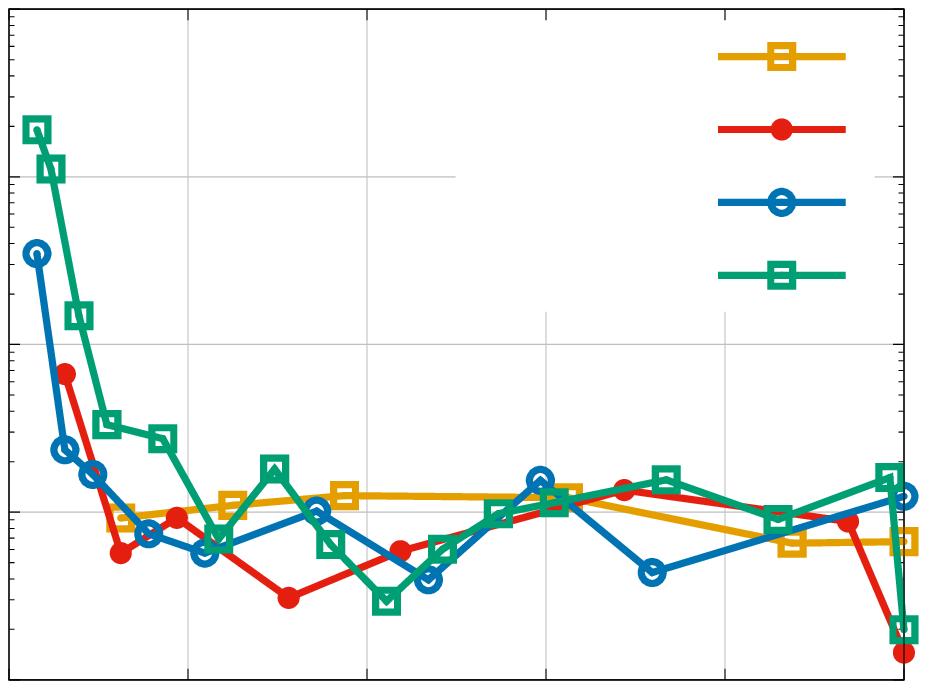}}}\\
{\scriptsize (c) $\mu = 8 \times 10^{-4}$, runtime speedup} & {\scriptsize (d) $\mu = 8 \times 10^{-4}$, number of sampling points}\\
\end{tabular}
\caption{Burgers' example: Plots (a),(c) report speedups of the AADEIM model for full-model dimensions $\nh \in \{256, 512, 1024, 2048\}$. In all cases, the AADEIM model achieves about one order of magnitude speedup. Plots (b),(d) report the error \eqref{eq:NumExp:Burgers:Error} with respect to the ratio $\nms/\nh$, which shows similar behavior for all full-model dimensions $\nh \in \{256, 512, 1024, 2048\}$.}
\label{fig:Burgers:ErrorDimN}
\end{figure}

Figure~\ref{fig:Burgers:TimeFine} shows the runtime of the full, static, and \abas{} models for $\nh = 1024$. The runtime of the full model and the static reduced model is partitioned into 4 parts that correspond to evaluating the right-hand side function (``RHS''), computing the Jacobian (``Jacobian''), and solving the corresponding linear system (``solve'') in each Newton iteration. The rest of the runtime is combined into overhead (``overhead''). The runtime of the \abas{} models is split into 7 parts that include the ones mentioned previously and additionally the runtime of sampling the right-hand side function and adapting the sampling points (``sample''), adapting the basis (``adapt $U$''), and computing the QDEIM interpolation points (``adapt $P$''). The overhead includes the runtime of initializing the \abas{} models with $\winit$ many time steps with the full model. The results in Figure~\ref{fig:Burgers:TimeFine} show that the full model spends most of the runtime in solving the linear system in each Newton iteration (``solve''). The \abas{} models reduce the runtime of solving the linear system drastically and so achieve about one order of magnitude speedup. When comparing the runtime of the \abas{} models with $\nms = 96, 224, 736$, one sees that the share of the runtime spent on sampling (``sample'') increases. The runtime due to the adaption of the basis (``adapt $U$'') and computing the QDEIM interpolation points (``adapt $P$'') is low compared to sampling. In particular, this means that using other methods for adapting reduced bases than ADEIM will have a minor effect on the overall runtime. To summarize, the results in Figure~\ref{fig:Burgers:TimeFine} indicate that \abas{} achieves runtime speedups in this example because the time for the linear solves in the Newton iterations is reduced compared to full and static reduced models.

\begin{figure}
{\LARGE\resizebox{1\columnwidth}{!}{\input{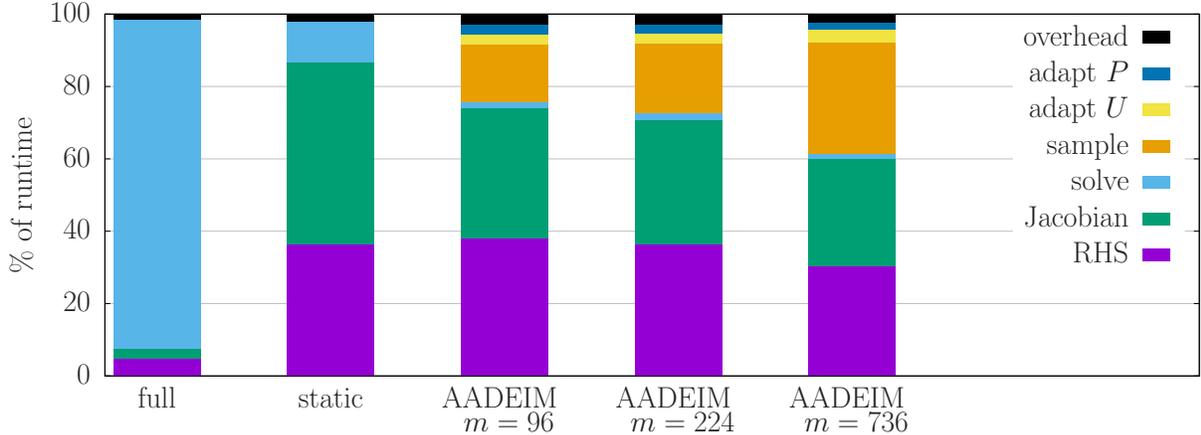}}}
\caption{The plot reports that \abas{} models spend less time on solving the system of nonlinear equations in each time step than the full model, which is one reason for the speedups of the \abas{} models compared to the full and static reduced models reported in Figure~\ref{fig:Burgers:ErrorDimN}a.}
\label{fig:Burgers:TimeFine}
\end{figure}

\subsubsection{Performance of \abas{} with respect to updates with SVD}
Consider a procedure Full+SVD that follows \abas{} except that the full-model right-hand side function $\bff$ is evaluated at all components at all time steps and the adapted basis is computed with a thin SVD. This means that the procedure Full+SVD follows Algorithm~\ref{alg:ABAS} except that lines 10-19 are replaced with $\bfF[:, k] = \bff(\bfQ[:, k]; \bfmu)$ and the new basis is computed with a thin SVD from $\bfF_k$ instead of ADEIM at line 21. The interpolation points $\bfP_k$ are computed with QDEIM, cf.~Algorithm~\ref{alg:qdeim}. Thus, the procedure Full+SVD evaluates the right-hand side function $\bff$ at all components rather than using local sampling to update the DEIM basis as \abas{}. The new basis is computed with $\textsc{Matlab}$'s implementation of the thin SVD.

Consider now the same setup as in Section~\ref{sec:NumExp:Burgers:N} to compare Full+SVD with \abas{}. In case of \abas{}, we adapt the basis every other time step as in Section~\ref{sec:NumExp:Burgers:N}. In case of the procedure Full+SVD, the basis is adapted every third time step so that a comparable error \eqref{eq:NumExp:Burgers:Error} is achieved as with \abas{}; see Figure~\ref{fig:Burgers:SVDDimN}a and Figure~\ref{fig:Burgers:SVDDimN}c for $\mu = 3 \times 10^{-3}$ and $\mu = 8 \times 10^{-4}$, respectively. We compare the runtime $T_A$ of \abas{} to the runtime $T_S$ of Full+SVD. The runtime reduction $T_R$ achieved by \abas{} is
\begin{equation}
T_R = \frac{T_S - T_A}{T_S} \times 100\,.
\label{eq:Burgers:RuntimeRed}
\end{equation}
The runtime reduction $T_R$ is shown in Figure~\ref{fig:Burgers:SVDDimN}b and Figure~\ref{fig:Burgers:SVDDimN}d. \abas{} achieves a runtime reduction of up to 30\% compared to Full+SVD because \abas{} requires computing the full-model right-hand side function $\bff$ at only a few components whereas Full+SVD evaluates the right-hand side function at all components. The costs of adapting the bases by either \abas{} or computing a thin SVD is negligible compared to the costs of to evaluating the right-hand side function, cf.~Figure~\ref{fig:Burgers:TimeFine}.

\begin{figure}
\begin{tabular}{cc}
{\LARGE\resizebox{0.45\columnwidth}{!}{\input{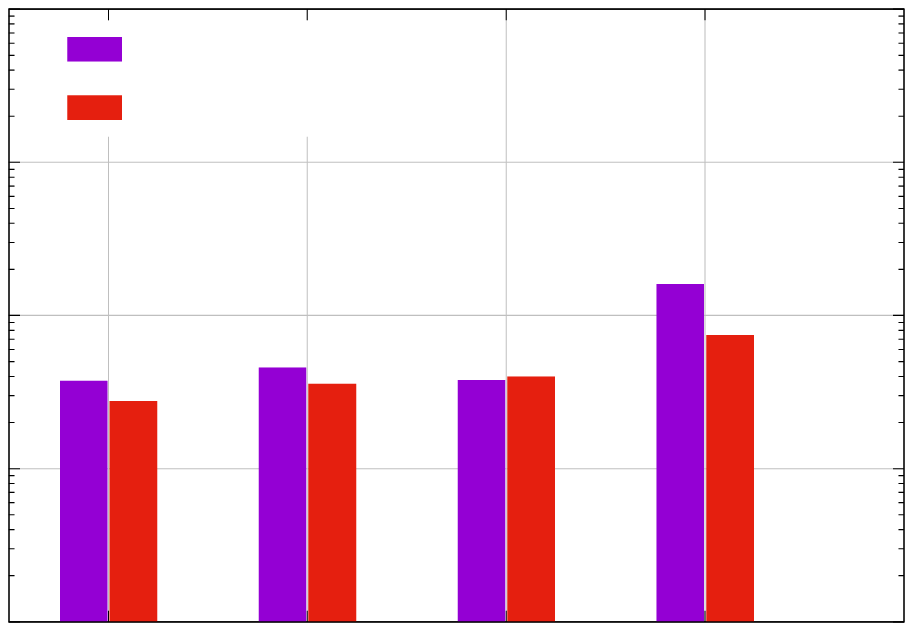}}} & {\LARGE\resizebox{0.45\columnwidth}{!}{\input{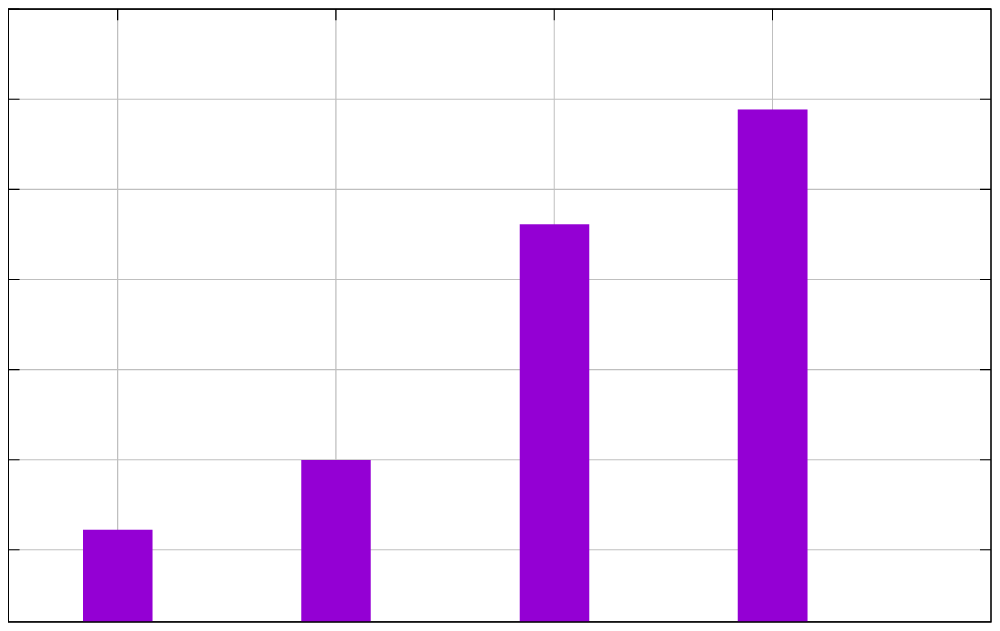}}} \\
{\scriptsize (a) error, $\mu = 3 \times 10^{-3}$} & \hspace*{-0.5cm}{\scriptsize (b) runtime reduction with \abas{}, $\mu = 3 \times 10^{-3}$}\\
{\LARGE\resizebox{0.45\columnwidth}{!}{\input{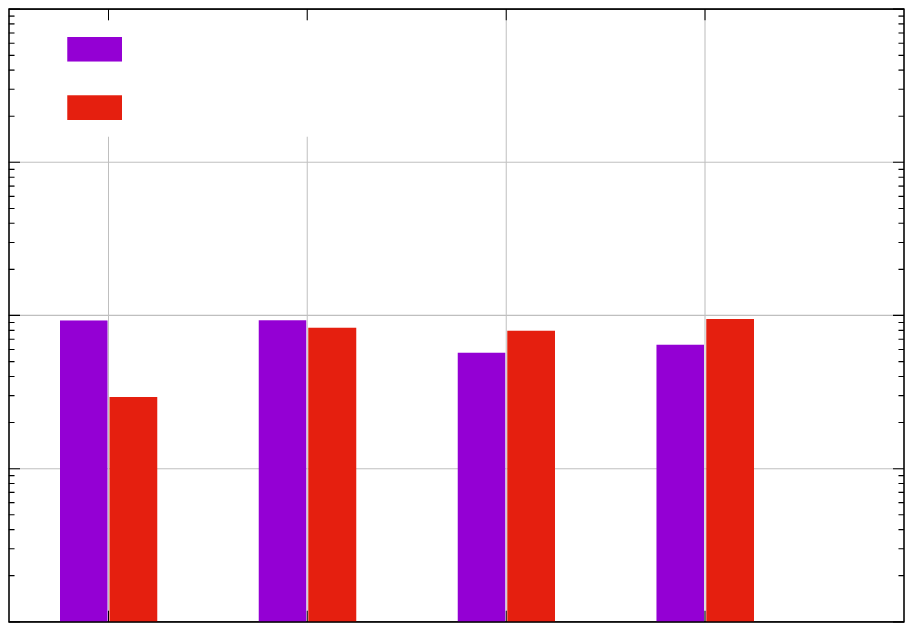}}} & {\LARGE\resizebox{0.45\columnwidth}{!}{\input{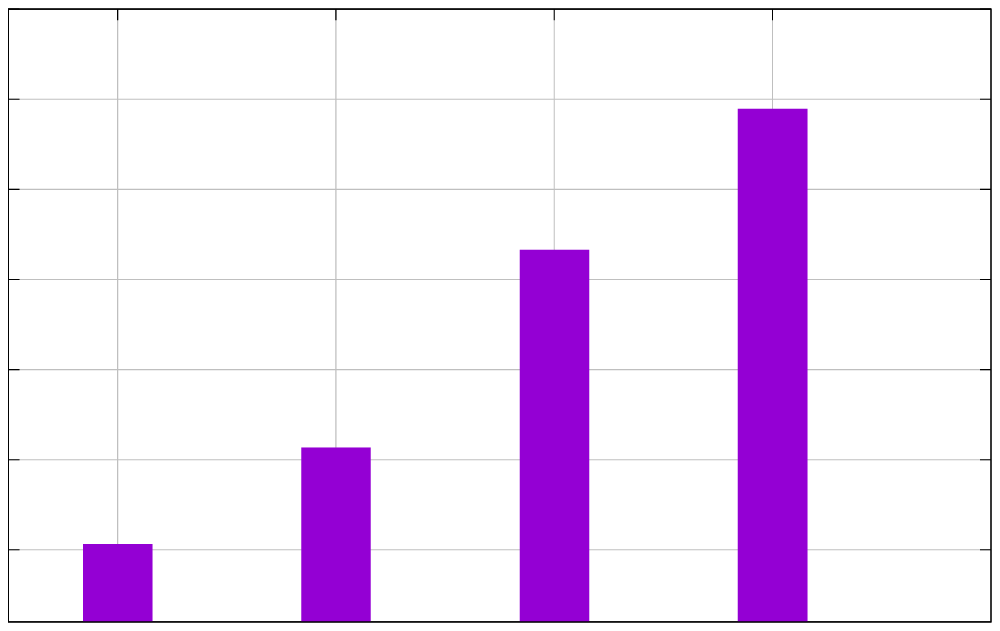}}}\\
{\scriptsize (c) error, $\mu = 8 \times 10^{-4}$} & \hspace*{-0.5cm}{\scriptsize (d) runtime reduction with \abas{}, $\mu = 8 \times 10^{-4}$}
\end{tabular}
\caption{Burgers' example: Plots (a) and (c) report that \abas{} achieves a comparable error \eqref{eq:NumExp:Burgers:Error} as Full+SVD, which evaluates the full-model right-hand side function at all components and then computes basis updates with a thin SVD. Plots (b) and (d) show that \abas{} achieves a runtime reduction \eqref{eq:Burgers:RuntimeRed} of up to 30\% compared to Full+SVD because \abas{} evaluates the full-model right-hand side function at a few components only.}
\label{fig:Burgers:SVDDimN}
\end{figure}

Figure~\ref{fig:Burgers:SVDLessDimN} reports results for Full+SVD if the full-model right-hand side function $\bff$ is evaluated at all components at every third time step only, rather than at all time steps as in the previous numerical experiment. Because of the fewer evaluations of $\bff$, Full+SVD is 2--8\% faster than \abas{}; see Figure~\ref{fig:Burgers:SVDLessDimN}a. However, the runtime decrease of Full+SVD, because of fewer evaluations of $\bff$, leads to poorer updates, which is shown in Figure~\ref{fig:Burgers:SVDLessDimN}b where now Full+SVD is up to one order of magnitude worse than \abas{} with respect to error \eqref{eq:NumExp:Burgers:Error}. Note that \abas{} evaluates $\bff$ at all components at every 5-th time step to update the sampling points and takes sparse evaluations of $\bff$ at all other time steps, which is the same setup as in the previous experiment.

\begin{figure}\begin{tabular}{cc}
 {\LARGE\resizebox{0.45\columnwidth}{!}{\input{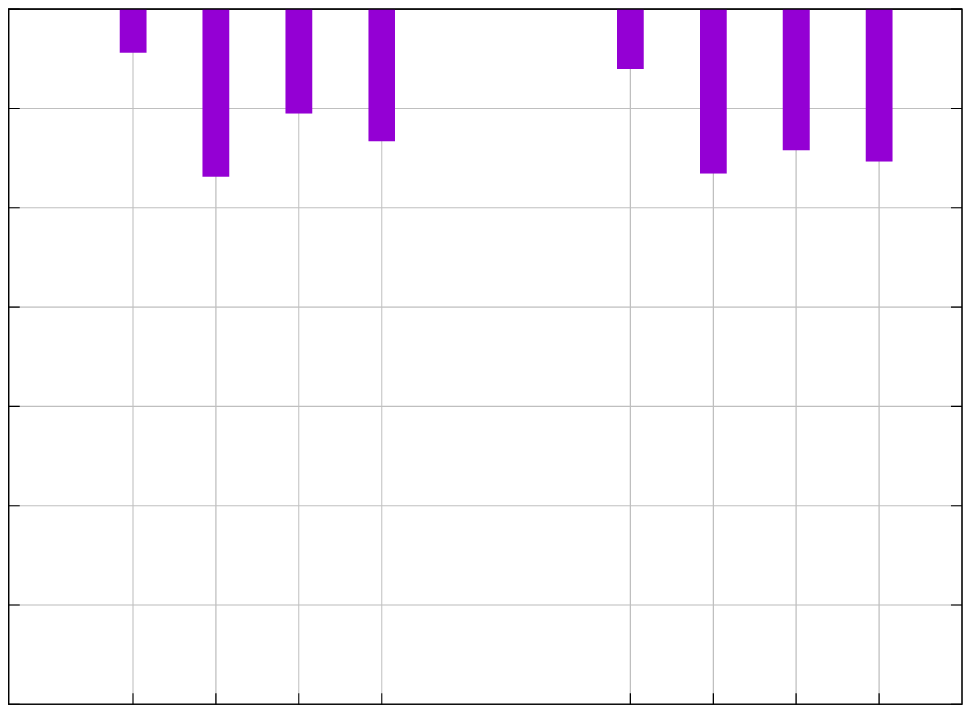}}} &{\LARGE\resizebox{0.45\columnwidth}{!}{\input{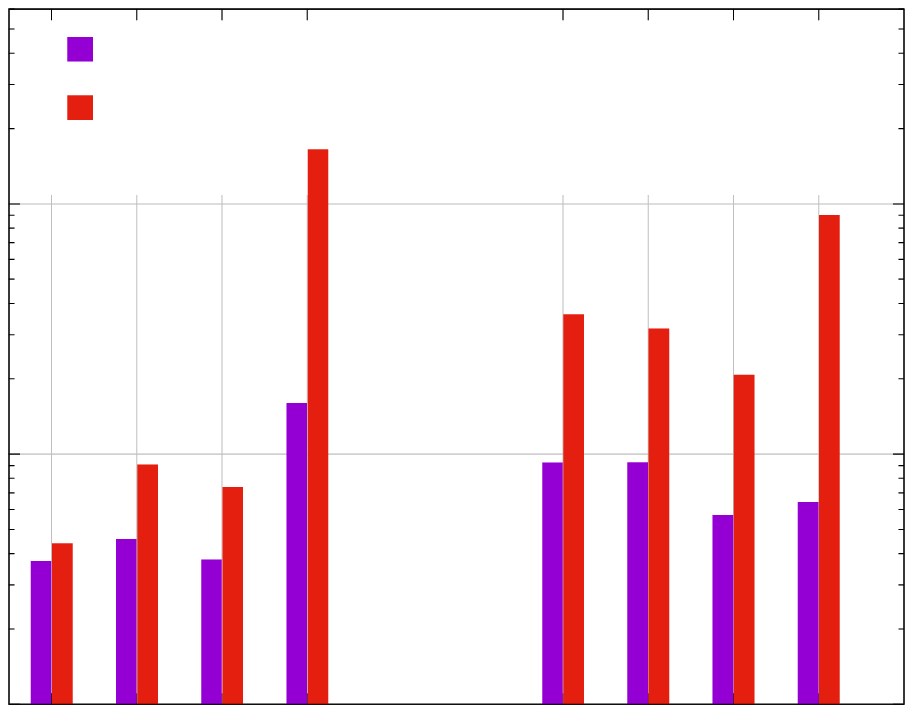}}} \\
& \\
{\scriptsize (a) runtime reduction} & \hspace*{-0.5cm}{\scriptsize (b) error}
\end{tabular}
\caption{Burgers' example: Reducing the number of full-model right-hand side function evaluations in Full+SVD leads to runtimes that are 2--8\% lower than the runtimes of \abas{}. However, fewer full-model right-hand side function evaluations lead to poorer updates that result in errors of Full+SVD that are up to one order of magnitude higher than the errors of \abas{}.}
\label{fig:Burgers:SVDLessDimN}
\end{figure}

\subsection{Combustion model}
\label{sec:NumExp:Combustion}
In this section, we apply \abas{} to a quasi-1D version of a single-element model rocket combustor. The model we use has been developed in the works \cite{smith_computational_2008,frezzotti_numerical_2017,FREZZOTTI2018261}. The goal is to approximate the growth of the amplitude of pressure oscillations at a monitoring point, which provides critical insights for designing engines that avoid combustion instabilities and unbounded growth of the amplitude of the pressure oscillations. The runtime results are obtained on compute nodes with Intel Xeon E5-1660v4 with 64GB RAM.

\begin{figure}
\centering
{\LARGE\resizebox{1\columnwidth}{!}{\input{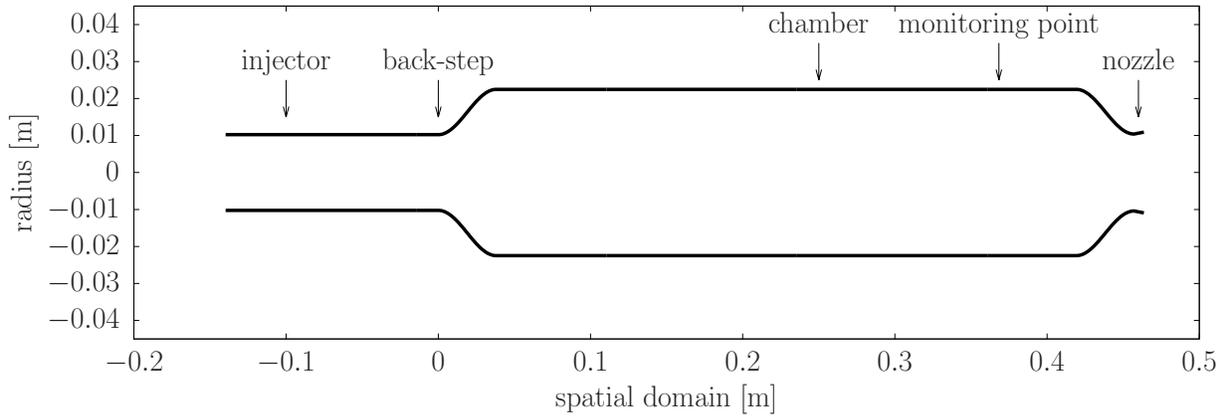}}}
\caption{Combustion example: Geometry of the combustion problem with injector, back-step, combustion camber, and nozzle.}
\label{fig:NumExp:Combustion:Geometry}
\end{figure}

\begin{figure}
\begin{tabular}{cc}
{\LARGE\resizebox{0.45\columnwidth}{!}{\input{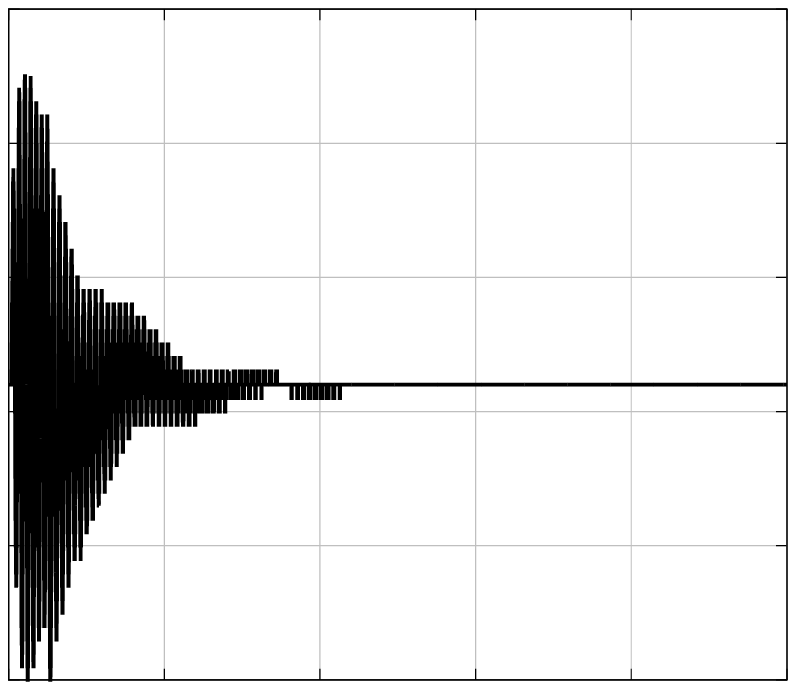}}} & {\LARGE\resizebox{0.45\columnwidth}{!}{\input{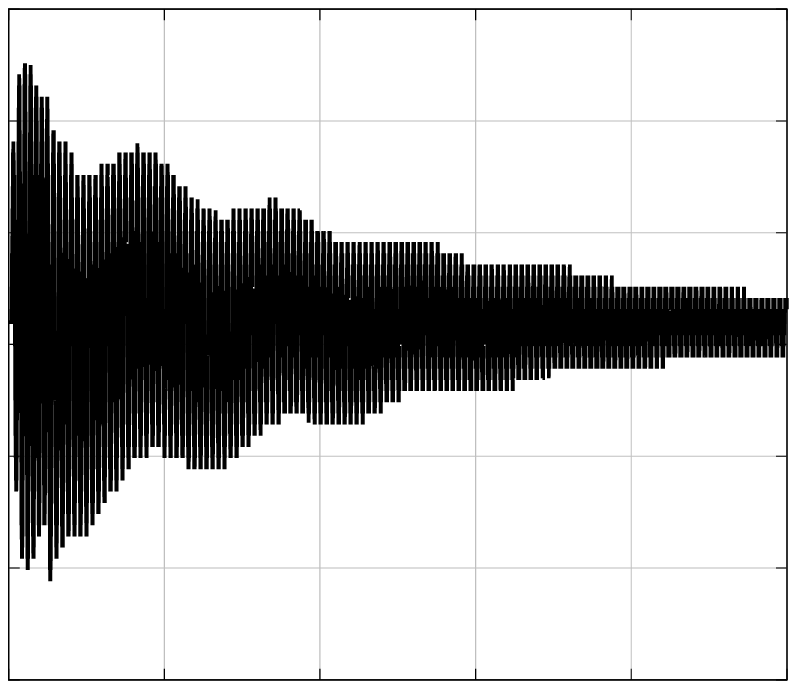}}}\\
{\scriptsize (a) $\mu = 2.4$} & {\scriptsize (b) $\mu = 3.0$}\\
 {\LARGE\resizebox{0.45\columnwidth}{!}{\input{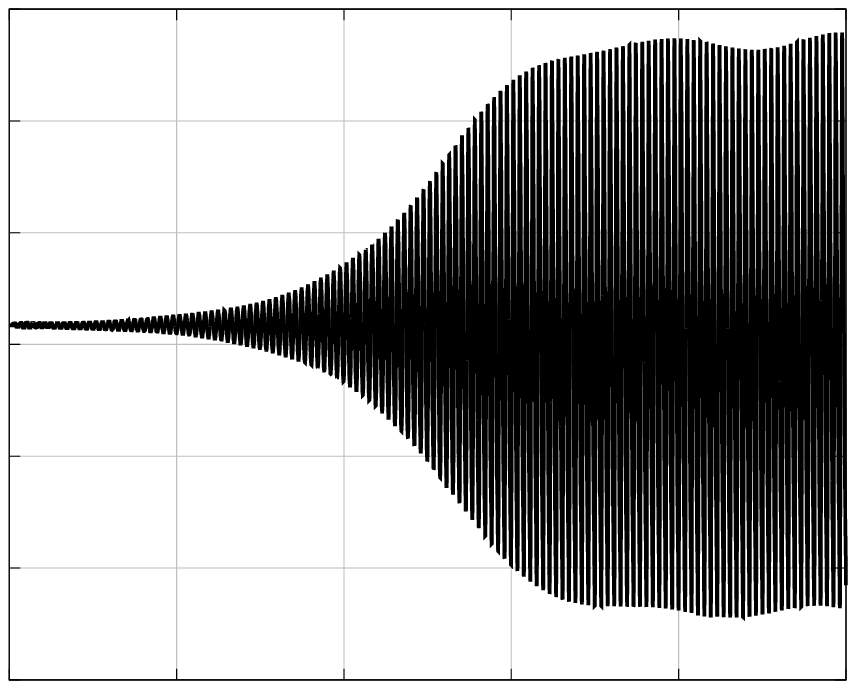}}} & {\LARGE\resizebox{0.45\columnwidth}{!}{\input{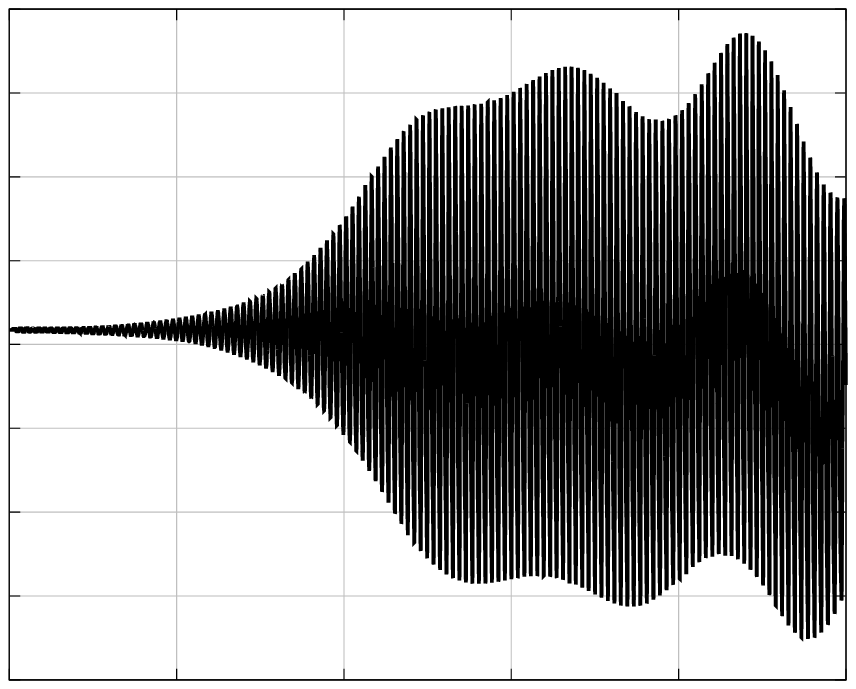}}}\\
{\scriptsize (c) $\mu = 3.8$} & {\scriptsize (d) $\mu = 4.0$}\\
\end{tabular}
\caption{Combustion example: The plots report the pressure at the monitoring point (see Figure~\ref{fig:NumExp:Combustion:Geometry}) for four different heat-release parameters. The results indicate that the parameter domain $\Dcal = [2,4.2]$ leads to solutions with significantly different behavior. Note that the solution for $\mu = 3.8$ seems to enter a limit cycle oscillation.}
\label{fig:NumExp:Combustion:FOMSolutions}
\end{figure}

\subsubsection{Problem setup and full model}
The problem setup and full model follows \cite{xu_reduced-order_2017,Wang:255708}. The problem setup consists of three parts, namely the oxidizer post, the combustion chamber, and the exit nozzle, see Figure~\ref{fig:NumExp:Combustion:Geometry}. The oxidizer is induced and meets the fuel at the back-step, where it reacts instantaneously. The combustion products exit the chamber through the nozzle. The combustion follows a one-step reaction model
\[
\mathrm C \mathrm H_4 + 2 \mathrm O_2 \to \mathrm C \mathrm O_2 + 2 \mathrm H_2 \mathrm O\,,
\]
where the fuel is gaseous methane and the oxidizer is a mixture of oxygen and water. Details on the operating conditions are given in \cite[Table~2]{xu_reduced-order_2017}. The governing equations are in conservative form
\[
\partial_t q + \partial_x g = s_A + s_g + s_q(\mu)\,,
\]
where
\[
q = \begin{bmatrix}
\rho A\\
\rho v A\\
\rho E A\\
\rho Y_{\text{ox}} A
\end{bmatrix}\,,\qquad
g = \begin{bmatrix}
\rho v A\\
(\rho v^2 + p)A\\
(\rho E + p)uA\\
\rho u Y_{\text{ox}} A
\end{bmatrix}\,,
\]
where $\rho$ is the density, $v$ is the velocity, $E$ is the total internal energy, $Y_{\text{ox}}$ is the oxidizer mass fraction, $A$ is the cross sectional area of the fuel duct, and $p$ is the pressure. The source terms $s_A$ and $s_g$ are given in \cite[equation~(2)]{xu_reduced-order_2017}. The source term $s_q(\mu) = [0, 0, q^{\prime}(\mu), 0]^T$ models the heat release, where the parameter $\mu \in \Dcal = [2, 4.2] \subset \mathbb{R}$ in $q^{\prime}(\mu)$ controls the amplification of the heat release, see \cite[equation~(5)]{xu_reduced-order_2017}. Following \cite{xu_reduced-order_2017}, a steady-state solution is first obtained by ignoring the source term $s_q(\mu)$. The steady-state solution is then used as initial condition for computing the time-dependent solution that takes the source term $s_q(\mu)$ into account. The initial condition is perturbed to trigger an instability, which depends on the heat-release parameter $\mu$. The spatial domain is discretized on 300 equidistant grid points. There are four degrees of freedom at each grid point (density, velocity, energy, mass fraction), and so the full model has a total of $\nh = 1200$ degrees of freedom. The spatial discretized is a first-order finite-difference scheme. Time is discretized with a fourth-order implicit scheme and time step size $\delta t = 10^{-7}$ with end time $T = 10^{-1}$. Newton's method is used to solve the corresponding system of nonlinear equations at each time step. Jacobians are derived analytically and passed in assembled form to the Newton solver. \textsc{Matlab}'s backslash operator is used to solve the linear system in each Newton step. The Newton scheme takes 10 iterations and uses a line search based on the Armijo condition \cite[p.~33]{WrightNumOpt} with control parameter $10^{-4}$ as recommended in \cite[p.~33]{WrightNumOpt}. The rest of the setup of the Newton scheme is the same as in Section~\ref{sec:NumExp:Burgers}. The same Newton scheme is used in all models. The pressure is monitored at spatial coordinate $x = 0.3683$, see Figure~\ref{fig:NumExp:Combustion:Geometry}. Figure~\ref{fig:NumExp:Combustion:FOMSolutions} shows the pressure at the monitoring point for parameters $\mu \in \{2.4, 3.0, 3.8, 4\}$. The solutions converge to a steady state for $\mu = 2.4$. A limit cycle oscillation is entered for $\mu = 3.8$. A combustion instability is observed for $\mu = 4.0$.

\begin{figure}
\begin{tabular}{cc}
{\LARGE\resizebox{0.45\columnwidth}{!}{\input{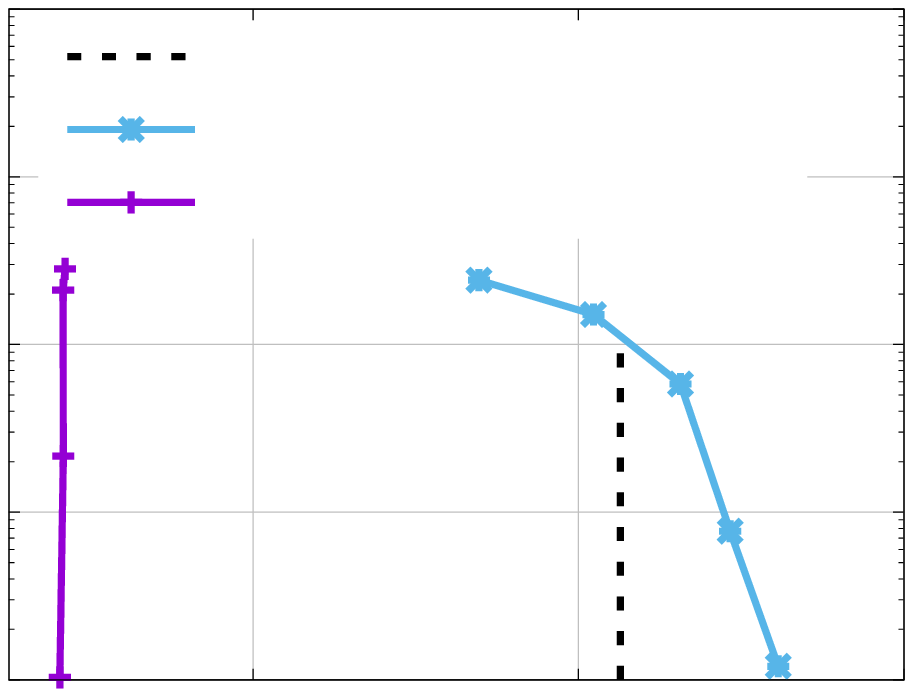}}} &  {\LARGE\resizebox{0.45\columnwidth}{!}{\input{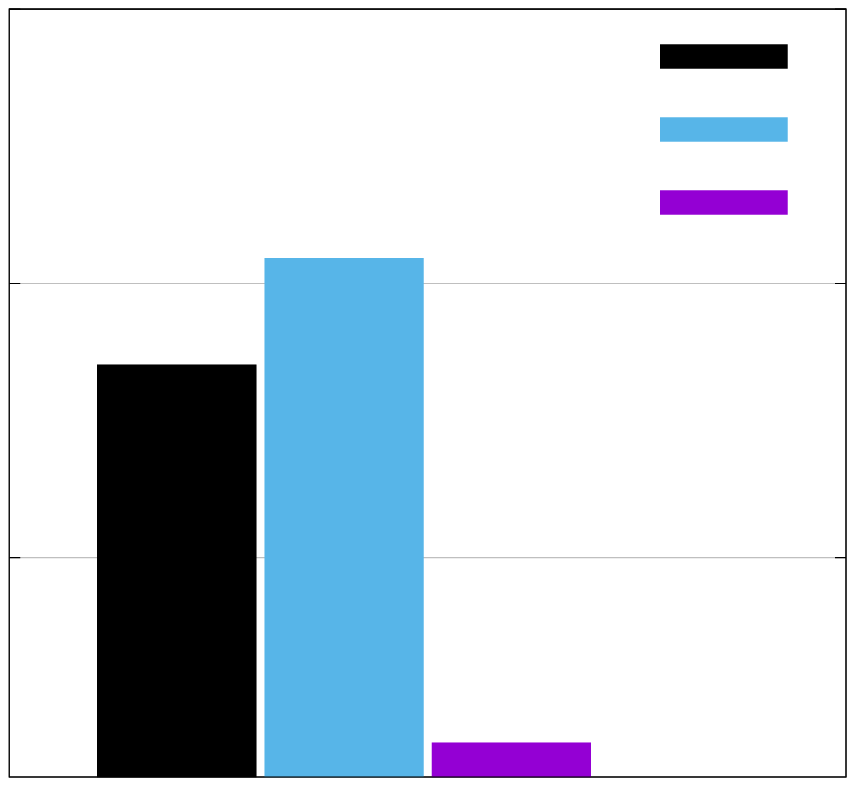}}}\\
{\scriptsize (a) error vs.~runtime} & {\scriptsize (b) runtime to achieve error of about $10^{-4}$}
\end{tabular}
\caption{Combustion example: Plot (a) shows that the static reduced model is even more expensive than the full model in this example. The \abas{} model achieves a speedup of about 6 compared to the full model. Plot (b) visualizes the speedup as a histogram for $\nms = 50$ sampling points in case of the \abas{} model and dimension 225 in case of the static reduced model, so that both reduced models achieve an error of about $10^{-4}$.}
\label{fig:NumExp:Combustion:Speedup}
\end{figure}

\subsubsection{Performance of \abas{}}
We compare static reduced models, \abas{} models, and the full model. The static reduced model is derived from the trajectories corresponding to the parameters $\mu \in \{2, 2.44, 2.88, 3.32, 3.76, 4.2\}$, which are the six equidistant parameters in the parameter domain $\Dcal = [2, 4.2]$. A separate DEIM basis of dimension $\nr$ is computed for each degree of freedom, see, e.g., \cite[Section~2.2]{zhou_model_2012}. The DEIM interpolation points are derived with QDEIM, where $\nr$ points are derived for each of the four DEIM basis and then the union of all four sets of points is used as the set of DEIM interpolation points. The dimension of the \abas{} model is $\nr = 8$ and the window size is $\win = \nr + 1$ in the following. The \abas{} model is initialized with the full-model states obtained until $t = 1.6 \times 10^{-4}$, where only every $50$-th state is used so that $\winit = 4 \times \nr = 32$. The sampling points are derived for each degree of freedom separately, then they are ranked by how often each sampling point has been selected, and then the same $\nms$ sampling points are used for all four bases that have been selected the most. This means that a total of $4 \times \nms$ components of the full-model residual are computed in each iteration. The basis is adapted at every time step and the sampling points are adapted very other time step ($\nz = 2$).
We measure the error of the pressure at the monitoring point. Let $\bfy(\mu) = [y_1(\mu), \dots, y_K(\mu)]^T \in \mathbb{R}^{K}$ be the trajectory of the pressure at the monitoring point computed with the full model. Then, we measure the error
\begin{equation}
\operatorname{err}(\tilde{\bfy}(\mu)) = \frac{\|\tilde{\bfy}(\mu) - \bfy(\mu)\|_2}{\|\bfy(\mu)\|_2}\,,
\label{eq:NumExp:Combustion:Error}
\end{equation}
where $\tilde{\bfy}(\mu)$ is the trajectory computed with either a static or an AADEIM model. The rest of the setup is the same as in Section~\ref{sec:NumExp:Burgers}.

Figure~\ref{fig:NumExp:Combustion:Speedup} reports the error \eqref{eq:NumExp:Combustion:Error} of the reduced models and their runtime. The error and runtime of the static reduced model is plotted for $\nr \in \{60, 70, 80, 90, 100\}$. The results for the \abas{} model are reported for $\nms \in \{20, 30, 40, 50\}$ and $\nr = 8$. The parameter is set to $\mu = 3.8$. The \abas{} model achieves a speedup of about 6 compared to the full model in this example. The static reduced model is slower than the full model. Figure~\ref{fig:NumExp:Combustion:ParameterSweep}a shows the speedup of the \abas{} model with $\nms = 50$ and for a parameter sweep over $\mu \in \{2.4, 3.0, 3.8, 4.0\}$. The dimension of the static reduced model is $\nr = 225$. The static and the \abas{} model achieve about the same error for all four parameters. Note that the four parameters lead to significantly different behaviors in the solutions, see Figure~\ref{fig:NumExp:Combustion:FOMSolutions}. The \abas{} model achieves a significant speedup compared to the full model, whereas the static reduced model is slower than the full model. Figure~\ref{fig:NumExp:Combustion:ParameterSweep}b demonstrates that the proposed adaptive sampling scheme is orders of magnitude more efficient than uniform sampling without replacement. Uniform sampling requires at least $\nms = 175$ sampling points per degree of freedom to prevent the Newton method from diverging and to achieve an error \eqref{eq:NumExp:Combustion:Error} of about $10^{-2}$. With the proposed adaptive sampling scheme, the \abas{} approach achieves an error \eqref{eq:NumExp:Combustion:Error} of about $10^{-4}$ with $\nms = 50$ sampling points per degree of freedom.

\begin{figure}
\begin{tabular}{cc}
   {\LARGE\resizebox{0.45\columnwidth}{!}{\input{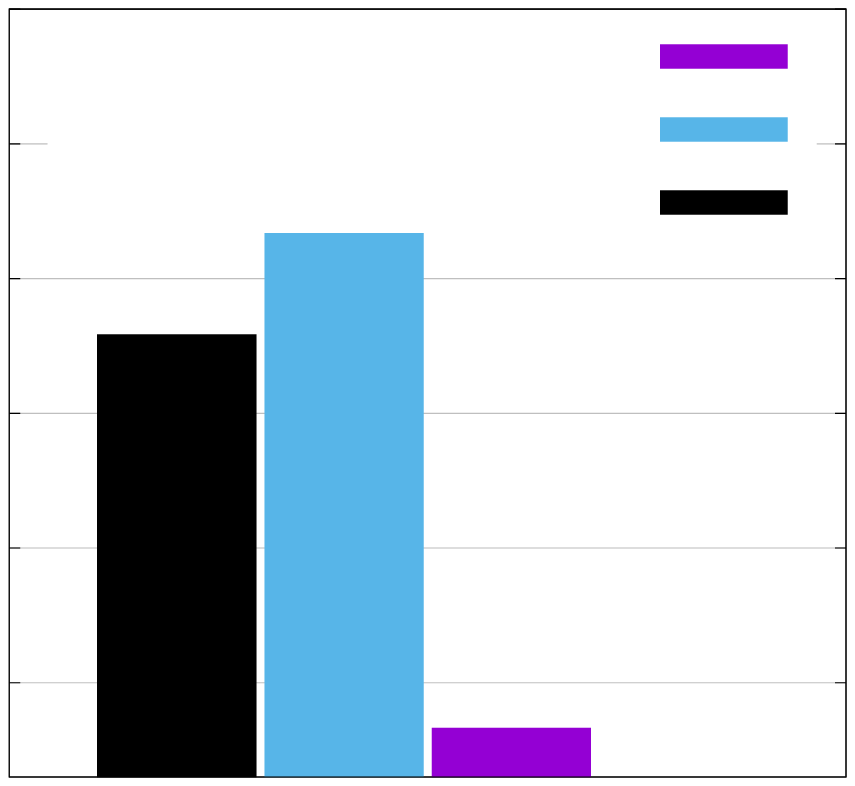}}} & {\LARGE\resizebox{0.45\columnwidth}{!}{\input{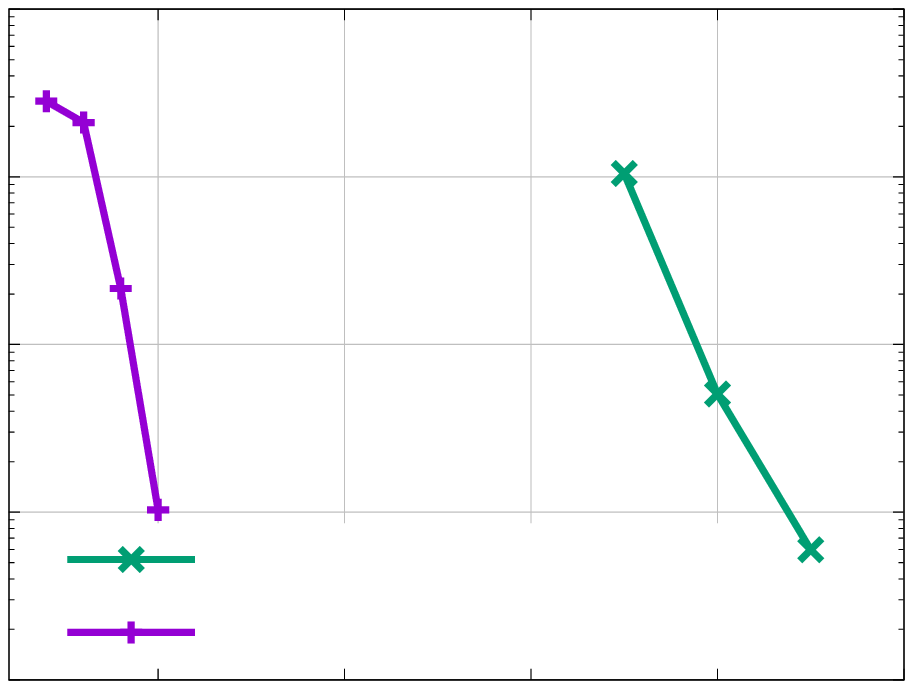}}}\\
{\scriptsize (a) parameter sweep} & {\scriptsize (b) sampling}
\end{tabular}
\caption{Combustion example: Plot (a) shows that the \abas{} model achieves significant speedups compared to the full model for a parameter sweep over the parameter domain $\Dcal = [2, 4.2]$, where solutions show significantly different behavior, cf.~Figure~\ref{fig:NumExp:Combustion:FOMSolutions}. This provides evidence that the proposed \abas{} approach is robust with respect to changes in the parameter. Plot (b) demonstrates that the proposed adaptive sampling scheme achieves orders of magnitude lower errors than uniform sampling. }
\label{fig:NumExp:Combustion:ParameterSweep}
\end{figure}

\section{Conclusions}
\label{sec:Conc}
The proposed \abas{} approach demonstrates that transport-dominated problems have a rich local structure that can be exploited to construct efficient reduced models. We exploit locality in time via adaptive basis updates and locality in space via adaptive sampling. An analysis establishes a connection between the local coherence properties of reduced spaces and the number of samples that are required to adapt the basis. The faster the local coherence decays, the fewer samples are required to adapt the reduced spaces. Reduced models built with \abas{} are implicitly parametrized, which means that there is no offline phase to construct reduced models, rather the basis is adapted online to changes in the parameter. Numerical results demonstrated that \abas{} is applicable to a wide range of problems and faithfully approximates behavior that changes significantly with parameters. At the same time, \abas{} achieves significant runtime speedups compared to full and traditional, static reduced models.

There are many directions for future work. First, the dimension of the reduced space, as well as the window size, can be adapted over time. For example, in problems with strongly time-varying coefficients or transport behavior that changes over time, time-varying windows and time-varying dimensions of the reduced spaces could help to further reduce the runtime. Second, the proposed approach relates the number of sampling points for the basis updates to the locality of the moving coherent structure. It would be interesting to compare to other adaptive basis schemes, e.g., incremental SVDs \cite{BRAND200620}, and to establish how many sampling points are required per adaptation step.

\section*{Acknowledgments}
The author would like to thank Nina Beranek (University of Ulm, Germany) for carefully reading an earlier version of this manuscript and for reporting typos and helpful comments.

\bibliography{convdeim}

\begin{thebibliography}{10}

\bibitem{Abgrall2016}
R.~Abgrall, D.~Amsallem, and R.~Crisovan.
\newblock Robust model reduction by {L1}-norm minimization and approximation
  via dictionaries: application to nonlinear hyperbolic problems.
\newblock {\em Advanced Modeling and Simulation in Engineering Sciences},
  3(1):1, Jan 2016.

\bibitem{AntBG10}
A.~Antoulas, C.~Beattie, and S.~Gugercin.
\newblock Interpolatory model reduction of large-scale dynamical systems.
\newblock In J.~Mohammadpour and K.~Grigoriadis, editors, {\em Efficient
  Modeling and Control of Large-Scale Systems}. Springer-Verlag, 2010.

\bibitem{4668528}
P.~Astrid, S.~Weiland, K.~Willcox, and T.~Backx.
\newblock Missing point estimation in models described by proper orthogonal
  decomposition.
\newblock {\em IEEE Transactions on Automatic Control}, 53(10):2237--2251, Nov
  2008.

\bibitem{doi:10.1137/S0895479803429764}
J.~Azaïs and M.~Wschebor.
\newblock Upper and lower bounds for the tails of the distribution of the
  condition number of a gaussian matrix.
\newblock {\em SIAM Journal on Matrix Analysis and Applications},
  26(2):426--440, 2004.

\bibitem{barrault_empirical_2004}
M.~Barrault, Y.~Maday, N.-C. Nguyen, and A.~Patera.
\newblock An ‘empirical interpolation’ method: application to efficient r
  educed-basis discretization of partial differential equations.
\newblock {\em Comptes Rendus Mathematique}, 339(9):667--672, 2004.

\bibitem{SIREVSurvey}
P.~Benner, S.~Gugercin, and K.~Willcox.
\newblock A survey of projection-based model reduction methods for parametric
  dynamical systems.
\newblock {\em SIAM Review}, 57(4):483--531, 2015.

\bibitem{PODFluid}
G.~Berkooz, P.~Holmes, and J.~L. Lumley.
\newblock The proper orthogonal decomposition in the analysis of turbulent
  flows.
\newblock {\em Annu. Rev. Fluid Mech.}, 25(1):539--575, 1993.

\bibitem{BRAND200620}
M.~Brand.
\newblock Fast low-rank modifications of the thin singular value decomposition.
\newblock {\em Linear Algebra and its Applications}, 415(1):20 -- 30, 2006.

\bibitem{Cagniart2019}
N.~Cagniart, Y.~Maday, and B.~Stamm.
\newblock Model order reduction for problems with large convection effects.
\newblock In B.~N. Chetverushkin, W.~Fitzgibbon, Y.~Kuznetsov,
  P.~Neittaanm{\"a}ki, J.~Periaux, and O.~Pironneau, editors, {\em
  Contributions to Partial Differential Equations and Applications}, pages
  131--150, Cham, 2019. Springer International Publishing.

\bibitem{Candes2009}
E.~J. Cand{\`e}s and B.~Recht.
\newblock Exact matrix completion via convex optimization.
\newblock {\em Foundations of Computational Mathematics}, 9(6):717, Apr 2009.

\bibitem{0266-5611-23-3-008}
E.~Candès and J.~Romberg.
\newblock Sparsity and incoherence in compressive sampling.
\newblock {\em Inverse Problems}, 23(3):969, 2007.

\bibitem{doi:10.1002/nme.4800}
K.~Carlberg.
\newblock Adaptive h-refinement for reduced-order models.
\newblock {\em International Journal for Numerical Methods in Engineering},
  102(5):1192--1210, 2015.

\bibitem{GNAT}
K.~Carlberg, C.~Bou-Mosleh, and C.~Farhat.
\newblock Efficient non-linear model reduction via a least-squares
  {Petrov–Galerkin} projection and compressive tensor approximations.
\newblock {\em International Journal for Numerical Methods in Engineering},
  86(2):155--181, 2011.

\bibitem{deim2010}
S.~Chaturantabut and D.~Sorensen.
\newblock Nonlinear model reduction via discrete empirical interpolation.
\newblock {\em {SIAM} Journal on Scientific Computing}, 32(5):2737--2764, 2010.

\bibitem{dahmen_plesken_welper_2014}
W.~Dahmen, C.~Plesken, and G.~Welper.
\newblock Double greedy algorithms: Reduced basis methods for transport
  dominated problems.
\newblock {\em ESAIM: Mathematical Modelling and Numerical Analysis},
  48(3):623–663, 2014.

\bibitem{QDEIM}
Z.~Drmač and S.~Gugercin.
\newblock A new selection operator for the discrete empirical interpolation
  method---improved a priori error bound and extensions.
\newblock {\em SIAM Journal on Scientific Computing}, 38(2):A631--A648, 2016.

\bibitem{doi:10.1137/10081157X}
M.~Drohmann, B.~Haasdonk, and M.~Ohlberger.
\newblock Reduced basis approximation for nonlinear parametrized evolution
  equations based on empirical operator interpolation.
\newblock {\em SIAM Journal on Scientific Computing}, 34(2):A937--A969, 2012.

\bibitem{EdelmanTailBounds}
A.~Edelman and B.~Sutton.
\newblock Tails of condition number distributions.
\newblock {\em SIAM Journal on Matrix Analysis and Applications},
  27(2):547--560, 2005.

\bibitem{2017arXiv171003569F}
L.~{Fick}, Y.~{Maday}, A.~T. {Patera}, and T.~{Taddei}.
\newblock {A Reduced Basis Technique for Long-Time Unsteady Turbulent Flows}.
\newblock {\em ArXiv e-prints}, Oct. 2017.

\bibitem{frezzotti_numerical_2017}
M.~L. Frezzotti, S.~D’Alessandro, B.~Favini, and F.~Nasuti.
\newblock Numerical issues in modeling combustion instability by quasi-1d
  {Euler} equations.
\newblock {\em International Journal of Spray and Combustion Dynamics},
  9(4):349--366, Dec. 2017.

\bibitem{FREZZOTTI2018261}
M.~L. Frezzotti, F.~Nasuti, C.~Huang, C.~L. Merkle, and W.~E. Anderson.
\newblock Quasi-1d modeling of heat release for the study of longitudinal
  combustion instability.
\newblock {\em Aerospace Science and Technology}, 75:261 -- 270, 2018.

\bibitem{GERBEAU2014246}
J.-F. Gerbeau and D.~Lombardi.
\newblock Approximated {Lax} pairs for the reduced order integration of
  nonlinear evolution equations.
\newblock {\em Journal of Computational Physics}, 265:246 -- 269, 2014.

\bibitem{grepl_efficient_2007}
M.~A. Grepl, Y.~Maday, N.~C. Nguyen, and A.~T. Patera.
\newblock Efficient reduced-basis treatment of nonaffine and nonlinear partial
  differential equations.
\newblock {\em ESAIM: Mathematical Modelling and Numerical Analysis},
  41(03):575--605, 2007.

\bibitem{gugercin_2008}
S.~Gugercin, A.~Antoulas, and C.~Beattie.
\newblock $\mathcal{H}_2$ {Model} {Reduction} for {Large}-{Scale} {Linear}
  {Dynamical} {Systems}.
\newblock {\em SIAM Journal on Matrix Analysis and Applications},
  30(2):609--638, Jan. 2008.

\bibitem{refId0Haasdonk}
B.~Haasdonk and M.~Ohlberger.
\newblock Reduced basis method for finite volume approximations of parametrized
  linear evolution equations.
\newblock {\em ESAIM: M2AN}, 42(2):277--302, 2008.

\bibitem{Iollo2000}
A.~Iollo, S.~Lanteri, and J.-A. D{\'e}sid{\'e}ri.
\newblock Stability properties of {POD-Galerkin} approximations for the
  compressible {Navier-Stokes} equations.
\newblock {\em Theoretical and Computational Fluid Dynamics}, 13(6):377--396,
  Mar 2000.

\bibitem{PhysRevE.89.022923}
A.~Iollo and D.~Lombardi.
\newblock Advection modes by optimal mass transfer.
\newblock {\em Phys. Rev. E}, 89:022923, Feb 2014.

\bibitem{Koch2007}
O.~Koch and C.~Lubich.
\newblock Dynamical low‐rank approximation.
\newblock {\em SIAM Journal on Matrix Analysis and Applications},
  29(2):434--454, 2007.

\bibitem{6651836}
F.~Krahmer and R.~Ward.
\newblock Stable and robust sampling strategies for compressive imaging.
\newblock {\em IEEE Transactions on Image Processing}, 23(2):612--622, 2014.

\bibitem{moore_principal_1981}
B.~Moore.
\newblock Principal component analysis in linear systems: {Controllability},
  observability, and model reduction.
\newblock {\em IEEE Transactions on Automatic Control}, 26(1):17--32, Feb.
  1981.

\bibitem{mullis_synthesis_1976}
C.~Mullis and R.~Roberts.
\newblock Synthesis of minimum roundoff noise fixed point digital filters.
\newblock {\em IEEE Transactions on Circuits and Systems}, 23(9):551--562,
  Sept. 1976.

\bibitem{Musharbash:231216}
E.~Musharbash and F.~Nobile.
\newblock Symplectic dynamical low rank approximation of wave equations with
  random parameters.
\newblock {\em Mathicse Technical Report nr 18.2017}, 2017.

\bibitem{MUSHARBASH2018135}
E.~Musharbash and F.~Nobile.
\newblock Dual dynamically orthogonal approximation of incompressible {Navier
  Stokes} equations with random boundary conditions.
\newblock {\em Journal of Computational Physics}, 354:135 -- 162, 2018.

\bibitem{NobileDO}
E.~Musharbash, F.~Nobile, and T.~Zhou.
\newblock Error analysis of the dynamically orthogonal approximation of time
  dependent random {PDEs}.
\newblock {\em SIAM Journal on Scientific Computing}, 37(2):A776--A810, 2015.

\bibitem{WrightNumOpt}
J.~Nocedal and S.~J. Wright.
\newblock {\em Numerical optimization}.
\newblock Springer, 2006.

\bibitem{OHLBERGER2013901}
M.~Ohlberger and S.~Rave.
\newblock Nonlinear reduced basis approximation of parameterized evolution
  equations via the method of freezing.
\newblock {\em Comptes Rendus Mathematique}, 351(23):901 -- 906, 2013.

\bibitem{doi:10.1137/18M1177263}
S.~Pan and K.~Duraisamy.
\newblock Data-driven discovery of closure models.
\newblock {\em SIAM Journal on Applied Dynamical Systems}, 17(4):2381--2413,
  2018.

\bibitem{PARISH2016758}
E.~J. Parish and K.~Duraisamy.
\newblock A paradigm for data-driven predictive modeling using field inversion
  and machine learning.
\newblock {\em Journal of Computational Physics}, 305:758 -- 774, 2016.

\bibitem{LDEIM}
B.~Peherstorfer, D.~Butnaru, K.~Willcox, and H.~Bungartz.
\newblock Localized discrete empirical interpolation method.
\newblock {\em SIAM Journal on Scientific Computing}, 36(1):A168--A192, 2014.

\bibitem{PDG18ODEIM}
B.~Peherstorfer, Z.~Drmac, and S.~Gugercin.
\newblock Stability of discrete empirical interpolation and gappy proper
  orthogonal decomposition with randomized and deterministic sampling points.
\newblock {\em arXiv:1808.10473}, 2018.

\bibitem{pehersto15dynamic}
B.~Peherstorfer and K.~Willcox.
\newblock Dynamic data-driven reduced-order models.
\newblock {\em Computer Methods in Applied Mechanics and Engineering},
  291:21--41, 2015.

\bibitem{Peherstorfer15aDEIM}
B.~Peherstorfer and K.~Willcox.
\newblock Online adaptive model reduction for nonlinear systems via low-rank
  updates.
\newblock {\em SIAM Journal on Scientific Computing}, 37(4):A2123--A2150, 2015.

\bibitem{PWG17MultiSurvey}
B.~Peherstorfer, K.~Willcox, and M.~Gunzburger.
\newblock Survey of multifidelity methods in uncertainty propagation,
  inference, and optimization.
\newblock {\em SIAM Review}, 60(3):550--591, 2018.

\bibitem{prudhomme_reliable_2001}
C.~Prud’homme, Y.~Maday, A.~T. Patera, G.~Turinici, D.~V. Rovas, K.~Veroy,
  and L.~Machiels.
\newblock Reliable {Real}-{Time} {Solution} of {Parametrized} {Partial}
  {Differential} {Equations}: {Reduced}-{Basis} {Output} {Bound} {Methods}.
\newblock {\em Journal of Fluids Engineering}, 124(1):70--80, Nov. 2001.

\bibitem{Reiss}
J.~Reiss, P.~Schulze, J.~Sesterhenn, and V.~Mehrmann.
\newblock The shifted proper orthogonal decomposition: A mode decomposition for
  multiple transport phenomena.
\newblock {\em SIAM Journal on Scientific Computing}, 40(3):A1322--A1344, 2018.

\bibitem{Rim2}
D.~Rim and K.~Mandli.
\newblock Displacement interpolation using monotone rearrangement.
\newblock {\em SIAM/ASA Journal on Uncertainty Quantification},
  6(4):1503--1531, 2018.

\bibitem{RimLeVeque}
D.~Rim, S.~Moe, and R.~LeVeque.
\newblock Transport reversal for model reduction of hyperbolic partial
  differential equations.
\newblock {\em SIAM/ASA Journal on Uncertainty Quantification}, 6(1):118--150,
  2018.

\bibitem{ROWLEY2004115}
C.~W. Rowley, T.~Colonius, and R.~M. Murray.
\newblock Model reduction for compressible flows using {POD} and {Galerkin}
  projection.
\newblock {\em Physica D: Nonlinear Phenomena}, 189(1):115 -- 129, 2004.

\bibitem{RozzaPateraSurvey}
G.~Rozza, D.~Huynh, and A.~Patera.
\newblock Reduced basis approximation and a posteriori error estimation for
  affinely parametrized elliptic coercive partial differential equations.
\newblock {\em Archives of Computational Methods in Engineering}, 15(3):1--47,
  2007.

\bibitem{ROZZA20071244}
G.~Rozza and K.~Veroy.
\newblock On the stability of the reduced basis method for {Stokes} equations
  in parametrized domains.
\newblock {\em Computer Methods in Applied Mechanics and Engineering},
  196(7):1244 -- 1260, 2007.

\bibitem{SAPSIS20092347}
T.~P. Sapsis and P.~F. Lermusiaux.
\newblock Dynamically orthogonal field equations for continuous stochastic
  dynamical systems.
\newblock {\em Physica D: Nonlinear Phenomena}, 238(23):2347 -- 2360, 2009.

\bibitem{SirovichMethodOfSnapshots}
L.~Sirovich.
\newblock Turbulence and the dynamics of coherent structures.
\newblock {\em Quarterly of Applied Mathematics}, pages 561--571, 1987.

\bibitem{smith_computational_2008}
R.~Smith, M.~Ellis, G.~Xia, V.~Sankaran, W.~Anderson, and C.~L. Merkle.
\newblock Computational {Investigation} of {Acoustics} and {Instabilities} in a
  {Longitudinal}-{Mode} {Rocket} {Combustor}.
\newblock {\em AIAA Journal}, 46(11):2659--2673, Nov. 2008.

\bibitem{Tommasao}
{Taddei, T.}, {Perotto, S.}, and {Quarteroni, A.}
\newblock Reduced basis techniques for nonlinear conservation laws.
\newblock {\em ESAIM: M2AN}, 49(3):787--814, 2015.

\bibitem{URBAN2012203}
K.~Urban and A.~T. Patera.
\newblock A new error bound for reduced basis approximation of parabolic
  partial differential equations.
\newblock {\em Comptes Rendus Mathematique}, 350(3):203 -- 207, 2012.

\bibitem{veroy_posteriori_2003}
K.~Veroy, C.~Prud'homme, D.~Rovas, and A.~Patera.
\newblock A {Posteriori} {Error} {Bounds} for {Reduced}-{Basis} {Approximation}
  of {Parametrized} {Noncoercive} and {Nonlinear} {Elliptic} {Partial}
  {Differential} {Equations}.
\newblock In {\em 16th {AIAA} {Computational} {Fluid} {Dynamics} {Conference}},
  Fluid {Dynamics} and {Co}-located {Conferences}. American Institute of
  Aeronautics and Astronautics, June 2003.

\bibitem{Wang:255708}
Q.~Wang, J.~S. Hesthaven, and D.~Ray.
\newblock Non-intrusive reduced order modeling of unsteady flows using
  artificial neural networks with application to a combustion problem.
\newblock {\em Journal of Computational Physics}, 2018.

\bibitem{Wang12PODclosureComp}
Z.~Wang, I.~Akhtar, J.~Borggaard, and T.~Iliescu.
\newblock Proper orthogonal decomposition closure models for turbulent flows: A
  numerical comparison.
\newblock {\em Computer Methods in Applied Mechanics and Engineering},
  237-240:10--26, 2012.

\bibitem{2017arXiv171011481W}
G.~{Welper}.
\newblock {$h$ and $hp$-adaptive Interpolation by Transformed Snapshots for
  Parametric and Stochastic Hyperbolic PDEs}.
\newblock {\em ArXiv e-prints}, Oct. 2017.

\bibitem{WelperInterpolate}
G.~Welper.
\newblock Interpolation of functions with parameter dependent jumps by
  transformed snapshots.
\newblock {\em SIAM Journal on Scientific Computing}, 39(4):A1225--A1250, 2017.

\bibitem{xu_reduced-order_2017}
J.~Xu and K.~Duraisamy.
\newblock Reduced-{Order} {Modeling} of {Model} {Rocket} {Combustors}.
\newblock In {\em 53rd {AIAA}/{SAE}/{ASEE} {Joint} {Propulsion} {Conference}},
  {AIAA} {Propulsion} and {Energy} {Forum}. American Institute of Aeronautics
  and Astronautics, July 2017.

\bibitem{doi:10.1142/S0218202514500110}
M.~Yano, A.~T. Patera, and K.~Urban.
\newblock A space-time hp-interpolation-based certified reduced basis method
  for {Burgers'} equation.
\newblock {\em Mathematical Models and Methods in Applied Sciences},
  24(09):1903--1935, 2014.

\bibitem{NIPS2015_6018}
H.~Zhang, Y.~Zhou, and Y.~Liang.
\newblock Analysis of robust {PCA} via local incoherence.
\newblock In C.~Cortes, N.~D. Lawrence, D.~D. Lee, M.~Sugiyama, and R.~Garnett,
  editors, {\em Advances in Neural Information Processing Systems 28}, pages
  1819--1827. Curran Associates, Inc., 2015.

\bibitem{zhou_model_2012}
Y.~B. Zhou.
\newblock {\em Model reduction for nonlinear dynamical systems with parametric
  uncertainties}.
\newblock Thesis ({S.M.}), Massachusetts Institute of Technology, Dept. of
  Aeronautics and Astronautics, 2012.

\bibitem{ZPW17SIMAXManifold}
R.~Zimmermann, B.~Peherstorfer, and K.~Willcox.
\newblock Geometric subspace updates with applications to online adaptive
  nonlinear model reduction.
\newblock {\em SIAM Journal on Matrix Analysis and Applications},
  39(1):234--261, 2018.

\end{thebibliography}
\bibliographystyle{abbrv}

\appendix
\section{Analytic example with local low-rank structure}
\label{sec:Appx:Example}
We analyze analytically an example to demonstrate the local low-rank structure of transport-dominated problems. The example follows \cite[Example~2.5]{Tommasao}. Consider the function
\begin{equation}
q(x, t) = \begin{cases}
0\,, & \qquad x \leq t\,,\\
1\,, & \qquad x > t\,,
\end{cases}
\label{eq:ABAS:LocalExpQDef}
\end{equation}
in the spatial domain $x \in [-5, 5]$ and $t \in [0, T]$ with $T = 1$. There is no $\nr$-dimensional Lagrangian space $\Ucal$ for which the  error $\sup_{t \in [0, T]}\inf_{q^* \in \Ucal} \|q(\cdot, t) - q^*\|_{L^2(-5, 5)}$ decays faster than linearly in $1/\sqrt{\nr}$, which means there is no space spanned by $q(\cdot, t_1), \dots, q(\cdot, t_{\nr})$ with $\nr$ pairwise distinct $t_i \in [0, T]$ with $i = 1, \dots, \nr$ that achieves a faster error decay than $1/\sqrt{\nr}$. Instead of considering the whole time domain $[0, T]$, let us now consider $[0, T/\zeta]$ with $\zeta > 0$. Let $\Ucal_{\zeta}$ be the space spanned by $q(\cdot, t_i)$ with $t_i = ih$ for $i = 1, \dots, \nr$ and $h = (T/\zeta)/\nr$. For $t \in [0, T/\zeta]$, the space $\Ucal_{\zeta}$ achieves
\[
\inf_{q^* \in \Ucal_{\zeta}} \|q(\cdot, t) - q^*\|_{L^2(-5, 5)} = \begin{cases}
\sqrt{t_1 - t}\,,&\qquad t < t_1\,,\\
\sqrt{\frac{(t_{i + 1} - t)(t - t_i)}{t_{i + 1} - t_i}}\,,&\qquad t_i \leq t \leq t_{i + 1}\,,i = 1, \dots, \nr - 1\,,
\end{cases}
\]
and thus $\sup_{t \in [0, T/\zeta]} \inf_{q^* \in \Ucal_{\zeta}} \|q(\cdot, t) - q^*\|_{L^2(-5, 5)} \leq \sqrt{h} = \sqrt{T/(\zeta \nr)}$. The rate of the error decay can be increased by letting $\zeta$ depend on $\nr$. For example, setting $\zeta = \mathrm e^{\nr}$ gives $\sup_{t \in [0, T/\zeta]} \inf_{q^* \in \Ucal_{\zeta}} \|q(\cdot, t) - q^*\|_{L^2(-5,5)} \in \mathcal{O}(\mathrm e^{-\nr})$, which shows that a local low-rank structure can be recovered if \eqref{eq:ABAS:LocalExpQDef} is approximated locally in time.

\section{Helper functions for Algorithm~\ref{alg:ABAS}}
\label{app:Code}

\begin{algorithm}[h]
\caption{Interpolation points selection with QDEIM}\label{alg:qdeim}
\begin{algorithmic}[1]
\Procedure{QDEIM}{$\bfU$}[See reference \cite{QDEIM}]
\State $[\sim, \sim, \bfP] = \texttt{qr}(\bfU^T, '\text{vector}');$
\State $\bfP = \bfP[1:\texttt{size}(\bfU, 2)]$\\
\Return $\bfP$
\EndProcedure
\end{algorithmic}
\end{algorithm}

\begin{algorithm}[h]
\caption{Adaptation with ADEIM}\label{alg:adeim}
\begin{algorithmic}[1]
\Procedure{ADEIM}{$\bfU, \bfP, \bfS, \bfF_p, \bfF_S, \nab$}[See reference \cite{Peherstorfer15aDEIM}]
\State $\bfC = \bfU[\bfP, :]\backslash\bfF_p$ \Comment{Coefficients w.r.t.~interpolation points}
\State $\bfR = \bfU[\bfS, :]\bfC - \bfF_s$ \Comment{Residual at sampling points}
\State $[\sim, \boldsymbol{Sv}, \boldsymbol{Sr}] = \texttt{svd}(\bfR, 0)$ \Comment{Compute SVD of residual}
\State $\boldsymbol{Sv} = \texttt{diag}(\boldsymbol{Sv})$
\State $(\bfC^T)^+ = \texttt{pinv}(\bfC^T)$ \Comment{Pseudo inverse of $\bfC^T$}
\State $\nab = \texttt{min}([\nab\,, \texttt{length}(\textbf{Sv})])$ \Comment{Determine rank of update}
\For{$i = 1, \dots, \nab$} \Comment{Apply updates}
\State $\bfalpha = -\bfR\boldsymbol{Sr}[:, i]$
\State $\bfbeta = (\bfC^T)^+\boldsymbol{Sr}[:, i]$
\State $\bfU[\bfS, :] = \bfU[\bfS, :] + \bfalpha\bfbeta^T$
\EndFor
\State Orthogonalize $\bfU$ \Comment{Orthogonalize columns of $\bfU$}
\State $\bfP = \texttt{qdeim}(\bfU)$ \Comment{Recompute QDEIM interpolation points}\\
\Return $\bfU, \bfP$
\EndProcedure
\end{algorithmic}
\end{algorithm}

\end{document}